\documentclass[11pt]{amsart}

\usepackage[margin=1in]{geometry}
\usepackage{amsmath,amssymb,mathtools}
\usepackage{enumitem}
\usepackage{microtype}
\usepackage[colorlinks=true,allcolors=blue]{hyperref}
\usepackage[nameinlink,capitalize,noabbrev]{cleveref}

\newtheorem{theorem}{Theorem}[section]
\newtheorem{proposition}[theorem]{Proposition}
\newtheorem{lemma}[theorem]{Lemma}
\newtheorem{corollary}[theorem]{Corollary}
\theoremstyle{definition}
\newtheorem{definition}[theorem]{Definition}
\theoremstyle{remark}

\title[IPM Blowup with Space-Time Smooth Force]{Extending the C\'ordoba-Mart\'inez-Zoroa IPM Blow-Up to Uniformly Space-Time Smooth Forcing}
\hypersetup{pdftitle={Extending the C\'ordoba-Mart\'inez-Zoroa IPM Blow-Up to Uniformly Smooth Forcing}}
\author{Levent Alp\"oge}
\address{Anthropic, PBC,
  500 Howard St., San Francisco, CA 94105, USA}
  \email{alpoge@fas.harvard.edu}
  
\author{Tristan Buckmaster}
\address{Courant Institute School of Mathematics, Computing, and Data Science,
  New York University, 251 Mercer Street, New York, NY 10012, USA}
  \email{buckmaster@nyu.edu}
\author{Matei P. Coiculescu}
\address{Courant Institute School of Mathematics, Computing, and Data Science,
  New York University, 251 Mercer Street, New York, NY 10012, USA}
  \email{m.coiculescu@nyu.edu}
\date{}

\begin{document}

\begin{abstract}
By adapting the techniques used in the IPM blowup result of C\'ordoba-Mart\'inez-Zoroa with spatially smooth force, we prove finite-time blow-up for the IPM equation on
\(\mathbb T^2\) with a uniformly spacetime smooth force. In particular, we show there exist a smooth odd initial density,
a smooth odd force \(F\in C^\infty([0,1]\times\mathbb T^2)\), and a classical
solution \(\rho\) on \([0,1)\) whose density gradient and spatial velocity
gradient diverge in \(L^\infty\) as \(t\uparrow1\). Nevertheless,
\(\rho(t)\) converges in \(C^\eta\) for every \(0\leq\eta<1\).
\end{abstract}

\maketitle

\section{Introduction}
\label{sec:introduction}

The study of incompressible fluid flow, while an active research field for
over a hundred years, still presents difficult mathematical challenges due to
the nonlinear and nonlocal behavior of its various model equations. Questions
about the global-in-time well-posedness of equations such as the incompressible
Navier--Stokes and Euler equations remain unsolved, and even one-dimensional
model problems such as the surface growth equation or the
C\'ordoba--C\'ordoba--Fontelos equation are not completely understood.
Singularity formation was discovered by Elgindi for \(C^{1,\alpha}\) solutions
of the three-dimensional Euler equations
\cite{ElgindiEuler2021}, but smooth blow-up results remain elusive. On the
other hand, C\'ordoba and Mart\'inez-Zoroa proved a breakthrough result
\cite{CordobaMartinezZoroa2025}, establishing finite-time singularity
formation for the forced IPM equation from smooth initial data and with a
spatially smooth force.

The IPM equation transports a density \(\rho\) by the divergence-free Darcy
velocity that it generates:
\[
  \partial_t\rho+u_{\mathbb T}(\rho)\cdot\nabla\rho=F,
  \qquad \operatorname{div}u_{\mathbb T}(\rho)=0
  \quad\text{on }\mathbb T^2,
\]
where \(u_{\mathbb T}\) is the periodic zero-order Fourier multiplier determined
by Darcy's law; its precise normalization is given in the overview below.
C\'ordoba, Gancedo, and Orive established foundational local existence and
uniqueness results, global continuation criteria, and blow-up in a special
infinite-energy class \cite{CordobaGancedoOrive2007}. Kiselev and Yao
constructed smooth data for the unforced equation, including on the torus, for
which global smoothness entails unbounded growth of derivatives, and they
proved nonlinear instability of a class of stratified steady states
\cite{KiselevYao2023}. At the critical Sobolev threshold, Bianchini, C\'ordoba,
and Mart\'inez-Zoroa proved nonexistence and strong ill-posedness in \(H^2\),
even near the linearly stable profile \(-x_2\)
\cite{BianchiniCordobaMartinezZoroa2025}. A more comprehensive literature
review is provided in the next section.

The blow-up construction of C\'ordoba and Mart\'inez-Zoroa relies on a
multiscale inductive scheme in which each previous iteration of the solution
serves as a background flow that rapidly amplifies an oscillatory wave packet
added to the solution. By carefully setting up the induction, they show that
the multiscale solution leads to infinite gradient growth in finite time. The
theorem of C\'ordoba and Mart\'inez-Zoroa concerns the forced IPM equation on
\(\mathbb{R}^2\), with initial density in \(C_c^\infty(\mathbb{R}^2)\) and a
compactly supported force in \(L_t^\infty C_x^\infty\). They show that the
density is a classical solution before \(t=1\) and that the gradients of both
the density and the velocity field blow up in \(C^0\) as \(t\to1^-\). We also
mention that related blow-up mechanisms have been used by C\'ordoba,
Mart\'inez-Zoroa, and their coauthors for other equations, including the forced
hypodissipative Navier--Stokes and generalized SQG equations
\cite{CordobaMartinezZoroaZheng2024,CordobaDominguezLucasMartinezZoroa2026}.

The goal of our work is to extend the result of C\'ordoba and
Mart\'inez-Zoroa for the forced IPM equation in the following way. We prove
finite-time singularity formation for the forced IPM equation on the torus
\(\mathbb{T}^2\), with initial density in \(C^\infty(\mathbb{T}^2)\) and a
space--time smooth force in \(C^\infty([0,1]\times\mathbb{T}^2)\). We also
obtain blow-up of the \(C^0\) norms of the gradients of the density and the
velocity field as \(t\to1^-\). We have endeavored to make the proof presented
here as clear and readable as possible. Before giving a precise statement of
the main theorem and an overview of the proof strategy, we wish to address two
possible points of contention. First, we directly compare the present work
with that of C\'ordoba and Mart\'inez-Zoroa and identify what is new and which
ideas are attributable to their earlier work. Second, we address the use of
language models in the preparation of this work.

The idea of a background flow amplifying a high-frequency perturbation dates
back at least to the work of Bourgain and Li, who used it to prove norm
inflation for the Euler equations in endpoint critical Sobolev spaces
\cite{BourgainLi2015}. C\'ordoba and Mart\'inez-Zoroa successfully implemented
this concept for the forced IPM equation, obtaining spatially smooth forced
blow-up. Many techniques in the present work originate in their work, including
transporting the perturbation's phase directly by the previous frequency level
while maintaining its oscillatory structure and tracking the orientation of
the oscillatory wave packet relative to its ``parent'' field. Of special note
is the idea of tying the amplitude of the wave packet to the oscillation
parameter in a frequency-dependent way, together with a delicate induction on
the number of controlled derivatives. More details are provided below. New
features introduced here for forced IPM include Fourier space truncations that
avoid a loss of derivatives and a mixed induction in time and space that
ultimately shows that the force is smooth in space--time. In addition, we transport the oscillatory packets by the Fourier truncated previous scales, while the work of C\'ordoba and Mart\'inez-Zoroa transport the next scale by the jet truncated velocity. Another improvement
over the work of C\'ordoba and Mart\'inez-Zoroa is that our solution is given by
an exact series and requires no nonconstructive passage through a nonlinear
limit. Everything is made as explicit as possible without sacrificing
readability.

% levent: i tried rewriting to be more precise below, in particular people i think
% it is important to credit openai's models here as well!
% We now address the use of LLMs in this work. We used LLMs to identify the key
% elements of the previous proof of C\'ordoba and Mart\'inez-Zoroa and to
% reproduce their argument. We used LLMs to write the main body of the text,
% delegating delicate bookkeeping of inductive orders and constants to the LLM
% while implementing our ideas for adapting the forced IPM construction of
% C\'ordoba and Mart\'inez-Zoroa to the torus and to a space--time smooth force.
% We used LLMs to reduce the length of initial versions of this work. After
% substantial review by the authors of previous LLM-generated versions, we used
% LLMs to revise the paper to our specifications. For instance, the LLM
% originally generated a proof of an overly complicated theorem statement,
% whereas we simplified it to the proof of our main theorem (which parallels the
% C\'ordoba--Mart\'inez-Zoroa IPM result) stated in the next section. This
% introduction (apart from TeX compatibility with the main body, formatting, and
% the second-paragraph literature review) was written by the listed authors, but
% the remainder of the paper was written primarily by an LLM, under the strict
% direction of the authors.

We now address the use of language models in this work. We used Claude to identify the key elements of the previous proof of C\'ordoba and Mart\'inez-Zoroa and to reproduce their argument. We used Claude and Codex to write the main body of the text, delegating delicate bookkeeping of inductive orders and constants to the models while implementing our ideas for adapting the forced IPM construction of C\'ordoba and Mart\'inez-Zoroa to the torus and to a space--time smooth force. We used the models in conjunction with our own ideas to reduce the length of initial versions of this work. Claude originally generated an overly complicated version of the result on the whole space following a proof architecture closer in resemblance to  C\'ordoba and Mart\'inez-Zoroa's work \cite{CordobaMartinezZoroa2025}. In simplifying the proof, we incorporated many of our own ideas, some of which are inspired by analogous ideas from the convex integration literature. This introduction (apart from TeX compatibility with the main body, formatting, and the second-paragraph literature review) was written by the listed authors, but the remainder of the paper was written with the assistance of Claude and Codex, under the strict direction of the authors. Finally, we used Claude and Codex to formalize the argument in this paper in Lean. The relevant Lean files will shortly be released to a public repository at \texttt{https://github.com/tristanbuckmaster/fluid\_lean}. % this said LLM, and anyway we were considering editing it out, im down to modify this phrasing anyway (also Matei was doing the heroic writing by hand so if that works that's even better) We also used language

\subsection{Blow-up for two-dimensional inviscid Boussinesq and  three-dimensional incompressible
Euler equations with uniformly space-time smooth forcing}

In forthcoming work, the first and second authors prove finite-time
singularity formation from smooth, compactly supported initial data for the
forced two-dimensional inviscid Boussinesq equations on \(\mathbb{R}^2\) and
the forced three-dimensional incompressible Euler equations on
\(\mathbb{R}^3\), in the latter case within the class of axisymmetric
solutions with swirl. In both results, the external forcing has fixed compact
spatial support and remains \(C^\infty\) in space and time up to and including
the singular time, with every mixed derivative uniformly bounded.

The proofs of the Boussinesq and Euler results
are currently being prepared. Both proofs have been formally verified in Lean, the necessary files may be found in the public repository at \texttt{https://github.com/tristanbuckmaster/fluid\_lean}.
The complete human-readable proofs will be released shortly by the first and second authors. We ask the community for
patience while these materials are prepared with the care they require.

\section{Overview}
\label{sec:overview}

Let $
  \mathbb T^2=(\mathbb R/2\pi\mathbb Z)^2$ be the two-dimensional torus
and normalize Fourier coefficients by
\[
  \widehat f(k)=\frac1{(2\pi)^2}
  \int_{\mathbb T^2}f(x)e^{-ik\cdot x}\,dx,
  \qquad k\in\mathbb Z^2.
\]
The periodic incompressible porous media velocity is the zero-order
multiplier
\begin{equation}
  \widehat{u_{\mathbb T}(\rho)}(k)
  =2\pi
  \left(\frac{k_1k_2}{|k|^2},-\frac{k_1^2}{|k|^2}\right)
  \widehat\rho(k)
  \quad(k\ne0),
  \qquad
  \widehat{u_{\mathbb T}(\rho)}(0)=0.
  \label{eq:periodic-velocity-symbol}
\end{equation}
The vector on the right is orthogonal to \(k\), so
\(\operatorname{div}u_{\mathbb T}(\rho)=0\). We study
\begin{equation}
  \partial_t\rho+u_{\mathbb T}(\rho)\cdot\nabla\rho=F
  \qquad\text{on }[0,1)\times\mathbb T^2.
  \label{eq:ipm}
\end{equation}
The factor \(2\pi\) in \eqref{eq:periodic-velocity-symbol} fixes the
normalization used below. The goal of this work is to prove:

\begin{theorem}[Finite-time blow-up]
\label{thm:main}
There exist functions
\[
  \rho_{\mathrm{in}}\in C^\infty(\mathbb T^2),
  \qquad
  F\in C^\infty([0,1]\times\mathbb T^2),
\]
and a solution \(\rho\) of \eqref{eq:ipm} with the following properties.
The initial datum \(\rho_{\mathrm{in}}\) is odd and has zero spatial mean,
while \(F(t,\cdot)\) is odd and has zero spatial mean for every
\(t\in[0,1]\). Moreover,
\[
  \rho(0)=\rho_{\mathrm{in}},
  \qquad
  \rho\in C^\infty([0,T]\times\mathbb T^2)
  \quad\text{for every }T<1.
\]
There is one function
\[
  \rho_*\in\bigcap_{0\leq\eta<1}C^\eta(\mathbb T^2)
\]
such that, for every \(0\leq\eta<1\),
\begin{equation}
  \rho(t)\longrightarrow\rho_*
  \quad\text{in }C^\eta(\mathbb T^2)
  \quad\text{as }t\uparrow1.
  \label{eq:endpoint-holder-convergence}
\end{equation}
Nevertheless,
\begin{equation}
  \lim_{t\uparrow1}\|\nabla\rho(t)\|_{L^\infty}
  =\lim_{t\uparrow1}
    \|D_xu_{\mathbb T}(\rho(t))\|_{L^\infty}
  =\infty.
  \label{eq:main-blowup-limits}
\end{equation}
\end{theorem}

C\'ordoba and Mart\'inez-Zoroa proved the corresponding finite-time
\(C^1\) blow-up, including blow-up of the velocity gradient, on
\(\mathbb R^2\) with compactly supported smooth initial density and a force
in \(L_t^\infty C_x^\infty\)
\cite[Theorem and Remark~2]{CordobaMartinezZoroa2025}. Their construction
introduces the multiscale induction, transported phase, a finite asymptotic
expansion of the velocity generated by a localized high-frequency
oscillation, and the phase-harmonic cancellation mechanism used here
\cite[Sections~1.2.1--1.2.4]{CordobaMartinezZoroa2025}. Joint smoothness in
time and space is anticipated there but is not part of that theorem
\cite[Remark~1]{CordobaMartinezZoroa2025}. The present argument implements
the mechanism on the torus, with nested Fourier cutoffs and an exact cutoff
commutator, and proves that the resulting force belongs to
\(C^\infty([0,1]\times\mathbb T^2)\).
\newpage
\subsection{Strategy}

\paragraph{\textbf{Packets and the parent--child step.}}

The elementary model is an oscillatory steady state. For
\(A\in\mathbb R\) and \(k\in\mathbb Z^2\setminus\{0\}\), the plane wave
\(\vartheta(x)=A\sin(k\cdot x)\) has velocity parallel to the multiplier vector in
\eqref{eq:periodic-velocity-symbol}. Since this vector is perpendicular to
\(k\), one has
\(u_{\mathbb T}(\vartheta)\cdot\nabla\vartheta=0\). Thus a single plane
wave is a stationary solution of the unforced equation on the torus. The following construction
localizes and transports this model. A localized oscillation is no longer an
exact unforced steady state, so at the initial level we take its nonlinear
residual as the force. At every later level the same principle is used: all
remainders are estimated, but none are omitted from the equation.

If \((\rho^{(q)},F^{(q)})\) are the density and force, respectively, at frequency level $q$, then the \emph{parent} at the next frequency level \(q+1\) is the complete preceding density after a
smoothing Fourier cutoff, \(P_{q+1}=S_{\mu_{q+1}}\rho^{(q)}\). The operator
\(S_\mu\) and its commutator are defined in
\eqref{eq:periodic-cutoff} and
\eqref{eq:periodic-cutoff-commutator-definition}; the child-dependent Fourier scale
\(\mu_{q+1}\) is specified below. The parent therefore
contains the contributions of all earlier levels, rather than
only the most recent oscillation. The new, much higher-frequency
perturbation \(\theta_{q+1}\) is the \emph{child}. If
\(\mathcal C_{\mu_{q+1}}(\rho^{(q)})\) denotes the exact commutator created
by the cutoff and \(E_{q+1}^{\mathrm{pert}}\) denotes the exact residual of
the child relative to \(P_{q+1}\), then the density and force at level $q+1$ are
\begin{equation}
  \begin{aligned}
    \rho^{(q+1)}&=P_{q+1}+\theta_{q+1},\\
    F^{(q+1)}&=S_{\mu_{q+1}}F^{(q)}
      +\mathcal C_{\mu_{q+1}}(\rho^{(q)})
      +E_{q+1}^{\mathrm{pert}}.
  \end{aligned}
  \label{eq:intro-stage-map}
\end{equation}
This is an identity, so every finite stage solves the forced IPM equation
exactly.

Let \(\lambda_{q+1}\) be the child frequency and let
\(0<\beta_{q+1}\leq1/8\) be its scale exponent. At leading order the child
has the same oscillatory form as the plane-wave model:
\begin{equation}
  \theta_{q+1}^{\mathrm{prin}}(x,t)
  =w_{q+1}(t)b_{q+1}(x,t)
    \sin\bigl(\lambda_{q+1}X_{q+1}(x,t)\bigr),
  \qquad
  b_{q+1}(0,t)
  =a_{q+1}(t)\lambda_{q+1}^{\beta_{q+1}-1}.
  \label{eq:intro-principal-perturbation}
\end{equation}
Here \(b_{q+1}\) is a compactly supported, slowly varying coefficient
propagated along the parent flow, \(X_{q+1}\) is the transported phase,
\(a_{q+1}\) controls the leading amplitude, and \(w_{q+1}\) switches the
packet on smoothly. More precisely, \(w_{q+1}=0\) up to the activation time
\(s_{q+1}\), vanishes there to infinite order, rises smoothly from zero to
one, and then remains one. Thus the packet is not switched off after its
activation interval. Its local frequency vector is
\(\lambda_{q+1}\nabla X_{q+1}\).

The principal packet alone does not have a sufficiently small residual. We
therefore construct the full child
$\theta_{q+1}=\Theta_M:=\zeta_1+\cdots+\zeta_M$, where
$M=\lceil\beta_{q+1}^{-2}\rceil$, as a finite nonlinear
\emph{parametrix}: an explicit approximate solution whose exact residual is
quantitatively controlled. The term $\zeta_1$ is the principal first
harmonic, and each later $\zeta_j$ corrects the residual left by the
preceding partial sum. For a pure plane wave, the translation-invariant velocity multiplier acts
exactly by multiplication by a fixed vector. For a localized packet with
transported phase, it admits a finite inverse-frequency expansion
$u^{[M]}$ about the local frequency vector. Its zeroth-order term is
obtained by evaluating the multiplier symbol at that vector, while the
higher-order terms contain derivatives of the envelope and of the non-affine
part of the phase. For an individual harmonic $z$, the exact remainder
$(u_{\mathbb T}-u^{[M]})[z]$ is retained intact and included in
$E_{q+1}^{\mathrm{pert}}$. This multiplier expansion is one ingredient of
the nonlinear parametrix.

The density ansatz has a finite phase-harmonic structure. For
\(1\leq j\leq M\),
\[
  \zeta_j
  =\sum_{\ell=1}^{2^j}
   \left(
     g^s_{\ell,j}\sin(\ell\lambda_{q+1}X_{q+1})
     +g^c_{\ell,j}\cos(\ell\lambda_{q+1}X_{q+1})
   \right),
\]
where \(g^s_{\ell,j}\) and \(g^c_{\ell,j}\) are slowly varying coefficient
functions.
The principal term \(\zeta_1\) is carried by the phase harmonic \(\ell=1\)
and is distinct from the leading frozen-frequency velocity response. The true
zero harmonic \(\ell=0\) is a nonoscillatory coefficient created when equal
harmonics interact; it is not the Fourier zero mode and need not be spatially
constant. Products of two harmonics generate their sums and
differences, so the structured residual after level \(j\) contains indices
up to \(2^{j+1}\). For \(1\leq j<M\), a weighted Duhamel formula chooses
\(\zeta_{j+1}\) to cancel the complete correctable nonzero-harmonic defect
\(\mathcal Q_j\).
The nonlinear interactions are then recomputed, producing a new and smaller
defect \(\mathcal Q_{j+1}\). At coefficient order zero, each successive
correction is smaller by one factor of
\(\lambda_{q+1}^{-\beta_{q+1}}\), up to logarithmic factors. Taking
\(M=\lceil\beta_{q+1}^{-2}\rceil\) correction levels supplies enough decay
to absorb the losses in the prescribed mixed-derivative range.

The correction procedure reduces the residual to the prescribed size; it
does not make it vanish. At the next level, the activation error, the
zero-harmonic terms due to nonlinear interaction, the exact multiplier
remainders, and the final
nonzero-harmonic defect all remain in
\(E_{q+1}^{\mathrm{pert}}\). This residual and the cutoff commutator are
placed in the force through \eqref{eq:intro-stage-map}. The higher
corrections are lower order in the origin-gradient geometry, so the
principal packet continues to carry the instability.

The amplification is geometric. The activation time is the unique time at
which the child phase-gradient direction is perpendicular to the leading
phase-gradient direction of the \(q\)-th packet, which dominates
\(\nabla P_{q+1}\) near the origin. The leading child velocity is transverse
to its own phase gradient and is therefore aligned with the leading
parent-density gradient. Both principal trigonometric factors are then close
to one, giving a favorable amplitude rate of size
\(\lambda_q^{\beta_q}\). The child itself vanishes at the precise activation
time because \(w_{q+1}=0\), but it is switched on immediately afterward with
an exponentially small coefficient. During the favorable part of the
evolution, the parent shear produces an exponentially large amplitude gain;
the normalization \(a_{q+1}(1)=1\) then gives a seed small enough to insert
the child at the next frequency.

The same parent shear transports the child phase gradient from the
perpendicular configuration toward the parent direction. By terminal time
the two phase gradients are nearly parallel, so the favorable amplitude
rate has faded. The remaining angular motion is slow. This is the precise
sense in which the leading origin geometry of the terminal child is almost
stationary: the phase is still transported, but
\(a_{q+1}(t)=1+o(1)\) and \(\nabla X_{q+1}(0,t)\) is
\(o(1)\)-close to its terminal value uniformly on the child's terminal
interval of time. The former child then has the geometry needed to become the dominant
newest part of the next parent.

After choosing \(\beta_{q+1}\in[\beta_q/2,\beta_q]\), a deterministic
frequency law selects \(\lambda_{q+1}\) exponentially large in
\(\lambda_q^{\beta_q/8}\), and the Fourier cutoff scale is
\(\mu_{q+1}=\lambda_{q+1}^{\beta_{q+1}/2}\). These choices separate the
child oscillation, its cutoff, and its terminal interval of time from the
corresponding parent scales.

The inductive step is closed by truncating the complete stage, not the parent and child
separately. Namely, the following level begins from
\(P_{q+2}=S_{\mu_{q+2}}\rho^{(q+1)}\). The nested cutoffs preserve all older
cutoff components and truncate only the newest child. The Fourier cutoff
supplies the higher spatial derivatives
needed at the following step through Bernstein estimates, thereby preventing
a loss of derivatives from propagating through the iteration. The cutoff
does introduce an error, but its exact commutator is again retained in the
force. This completes one parent--child cycle.

\paragraph{\textbf{The global induction and the scale exponent.}}

The iteration has two simultaneous goals. First, the new force increment
\(G_{q+1}:=\mathcal C_{\mu_{q+1}}(\rho^{(q)})
+E_{q+1}^{\mathrm{pert}}\) must decay at the child frequency. Once
\(\beta_{q+1}\leq1/40\), the estimates from the inductive step described
above give the following estimate uniformly in time for all integers
$m,k\geq0$:
\begin{equation}
  \|\partial_t^mG_{q+1}(t)\|_{C_x^k}
  \leq\lambda_{q+1}^{-1/4}
  \quad\text{whenever}\quad
  m+k\leq
  \left\lfloor\frac1{4\beta_{q+1}}\right\rfloor-10.
  \label{eq:intro-force-range}
\end{equation}
The second goal is that the terminal density gradient and velocity gradient
must keep growing through the inductive process.
Their principal scales are respectively
\(\lambda_q^{\beta_q}\) and
\(\lambda_q^{7\beta_q/8}\).

The scale exponent \(\beta_q\) therefore has two roles. It enters the size
and spatial concentration of the packets, and it controls the finite range
of mixed derivatives in \eqref{eq:intro-force-range}. Roughly
\(1/\beta_q\) derivatives of the force are controlled at level \(q\).
Smaller \(\beta_q\) gives a larger range, but it also increases both the
correction depth \(M_q\simeq\beta_q^{-2}\) and the number of parent
derivatives required to construct the next child. Although every finite
stage is smooth, the recursion records only a finite mixed-derivative range
for the density $\rho^{(q)}$ at level $q$, so \(\beta_q\) cannot simply be halved at every
step.

The recursion instead follows a retain-or-halve rule. It retains
\(\beta_q\) while the verified derivative range of the density $\rho^{(q)}$ at level $q$
catches up with the larger finite requirement at exponent \(\beta_q/2\).
Once that range and the associated finite frequency thresholds are
available, it sets \(\beta_{q+1}=\beta_q/2\). While \(\beta_q\) is retained,
frequency separation increases the available derivative reserve until the
halving criterion is satisfied; hence a positive exponent cannot be retained
forever. Consequently \(\beta_q\downarrow0\), and every
fixed pair \((m,k)\) therefore eventually lies in the range of
\eqref{eq:intro-force-range}.

The frequency growth dominates the decrease of the scale exponent: the scale
separation gives \(\beta_q\log\lambda_q\to\infty\). Thus
\(\lambda_q^{\beta_q}\) and
\(\lambda_q^{7\beta_q/8}\) still diverge. The same recursion can therefore
make the force increments small in successively more derivatives while
preserving the density-gradient and velocity-gradient amplification.

\paragraph{\textbf{The mixed space--time induction.}}

The force must be smooth jointly in space and time, whereas at a fixed stage
the decay estimate \eqref{eq:intro-force-range} covers only finitely many
mixed derivatives. Let
\(D=\partial_t+u_{\mathbb T}(P_{q+1})\cdot\nabla\) be the parent material
derivative. After the exact scalar amplitude factor is divided out, a
normalized Duhamel coefficient \(J\) satisfies
\((D+\widetilde c_P)J=\varphi\) with \(J(s_{q+1})=0\), where \(c_P\) is the
leading parent-interaction coefficient,
\(\widetilde c_P=c_P-c_P(0,t)\) is its centered form, and \(\varphi\) is the
already constructed Duhamel source. For \(p\geq1\),
differentiating this identity \(p-1\) times gives
\begin{equation}
  D^pJ
  =D^{p-1}\varphi
   -\sum_{r=0}^{p-1}\binom{p-1}{r}
      (D^r\widetilde c_P)D^{p-1-r}J.
  \label{eq:intro-duhamel-triangle}
\end{equation}
An order-\(p\) correction therefore requires positive material derivatives
of its Duhamel source only through order \(p-1\), at the cost of additional
spatial derivatives. This one-order gain is the mechanism by which the time
integral closes the induction.

The construction first closes the complete correction cascade at time order
zero. For each \(p\geq1\), it next proves the order-\(p\) estimate for the
seeded principal coefficient and then, successively, those for the
higher-correction coefficients, using only the fixed seed and already known
lower time orders of the parent inputs and preceding defects.
Only afterward does it estimate the order-\(p\) structured defects,
perturbation residual, cutoff commutator, and force increment. Thus no
order-\(p\) defect estimate is used to prove an order-\(p\) child-coefficient
estimate. Since
the phase satisfies \(DX_{q+1}=0\), conversion from material to ordinary
time derivatives uses only the additional spatial derivatives already built
into the finite range.

The activation cutoff and all its derivatives vanish at
\(s_{q+1}\), and this infinite-order vanishing propagates through the
Duhamel corrections.
Each child and its residual therefore extend smoothly by zero to earlier
times. Because \(\beta_q\to0\), every fixed mixed derivative eventually
satisfies \eqref{eq:intro-force-range}; the summability of
\(\lambda_q^{-1/4}\) then yields
\(F\in C^\infty([0,1]\times\mathbb T^2)\), including at the endpoint.

Finally, the activation times tend to one and the frequencies are lacunary.
On every compact interval \([0,T]\), \(T<1\), all sufficiently high-frequency
children vanish, so the finite stages stabilize there and give a classical
solution without a separate nonlinear compactness argument. The density
increments are summable in \(C^\eta\) for every \(\eta<1\), while the
terminal density gradients and velocity gradients diverge. Support separation
and Fourier-cutoff tail estimates ensure that later children do not interfere with the
dominant geometry on the preceding terminal intervals. This proves both
limits in \eqref{eq:main-blowup-limits} and the convergence in
\eqref{eq:endpoint-holder-convergence}.

Section~\ref{sec:velocity-and-jets} develops the localized oscillation
calculus using the periodic multiplier, cutoff, and kernel estimates proved
in the appendix. Section~\ref{sec:inductive-dynamics}
states the finite-stage invariant and proves the angle--amplitude mechanism.
Section~\ref{sec:one-layer} constructs one corrected child, and
Section~\ref{sec:global-iteration} gives the global recursion and proves
Theorem~\ref{thm:main}. The appendix also supplies the all-order mixed
estimates used to obtain joint smoothness of the force.
\newpage
\subsection{Relevant Literature}

\paragraph{Classical IPM and stable regimes.}
The behavior of the incompressible porous media equation (IPM) depends
strongly on the solution class and the background stratification.
C\'ordoba, Gancedo, and Orive established local existence and uniqueness
for classical solutions and several continuation criteria; they also
exhibited blow-up in a special infinite-energy class
\cite{CordobaGancedoOrive2007}. Around stable stratifications, Elgindi
proved nonlinear asymptotic stability \cite{Elgindi2017}, Castro,
C\'ordoba, and Lear obtained global quasi-stratified solutions in a
confined domain \cite{CastroCordobaLear2019}, and Bianchini, Crin-Barat,
and Paicu established asymptotic stability near a linearly stratified
state \cite{BianchiniCrinBaratPaicu2024}.

The stability theory coexists with mechanisms that create small scales.
Kiselev and Yao constructed solutions whose derivatives must grow without
bound if they remain smooth globally, and deduced nonlinear instability
for a class of stratified steady states \cite{KiselevYao2023}. At the
critical Sobolev index, Bianchini, C\'ordoba, and Mart\'inez-Zoroa proved
nonexistence and strong ill-posedness in $H^2$ for small perturbations of a
linearly stable profile \cite{BianchiniCordobaMartinezZoroa2025}. Park
recently sharpened stability analysis in a periodic channel by exploiting
the energy structure shared by IPM and Stokes transport \cite{Park2025}.
None of these stability or small-scale-growth results supplies a smooth,
finite-energy, unforced IPM singularity.

\paragraph{Weak nonuniqueness and mixing.}
At lower regularity, the picture changes sharply. C\'ordoba, Faraco, and
Gancedo constructed nontrivial bounded weak solutions compactly supported
in time, yielding nonuniqueness from zero initial data
\cite{CordobaFaracoGancedo2011}. Sz\'ekelyhidi then computed an explicit
relaxation of IPM and used convex integration to construct weak mixing
solutions for the unstable Muskat problem \cite{Szekelyhidi2012}. Isett
and Vicol proved the existence of H\"older-continuous weak solutions and
an $h$-principle for active-scalar multipliers that are not odd, a class
containing IPM \cite{IsettVicol2015}.

Subsequent work extended these constructions substantially. F\"orster and
Sz\'ekelyhidi treated broad classes of Muskat interfaces
\cite{ForsterSzekelyhidi2018}; Castro, C\'ordoba, and Faraco constructed
mixing solutions for fully unstable $H^5$ data
\cite{CastroCordobaFaraco2021}; Castro, Faraco, and Mengual quantified
degraded mixing \cite{CastroFaracoMengual2019} and later localized the
mixing zone for bubbles and turned interfaces
\cite{CastroFaracoMengual2022}; and Mengual incorporated viscosity jumps
\cite{Mengual2022}. Castro, Faraco, and Gebhard developed a macroscopic
entropy formulation intended to select a coarse-scale evolution within
this nonunique regime \cite{CastroFaracoGebhard2025}. These are rigorous
results about low-regularity weak solutions and must be kept logically
separate from loss of regularity along a classical solution.

\paragraph{Rigorous IPM singularity results.}
The known rigorous singularity results occupy several distinct settings.
Castro, C\'ordoba, Gancedo, and Orive identified an explicit blow-up
mechanism in a special infinite-energy reduction
\cite{CastroCordobaGancedoOrive2009}. Collot, Prange, and Tan recently
proved nonlinear stability of the resulting self-similar blow-up, again
within that infinite-energy reduction on an infinite periodic strip
\cite{CollotPrangeTan2026}. Kiselev and Sarsam proved finite-time blow-up
for a one-dimensional boundary-layer model of IPM, rather than the full
two-dimensional system \cite{KiselevSarsam2025}.

Two recent works are closer to the full equation but retain essential
qualifications. Dembski's preprint proves singularity formation for the
unforced IPM equation in a wedge, starting from Lipschitz data whose
density vanishes on the boundary \cite{Dembski2025}. C\'ordoba and
Mart\'inez-Zoroa construct smooth, finite-energy solutions of
two-dimensional IPM that become singular under a compactly supported
spatially smooth force \cite{CordobaMartinezZoroa2025}. The latter provides the
closest precedent for the present periodic construction. Its introductory
theorem states only \(F\in L_t^\infty C_x^\infty\), while Remark~1 treats
higher time regularity as requiring additional work.

\paragraph{The $C^{1,\alpha}$ Euler comparison.}
The $C^{1,\alpha}$ incompressible Euler literature supplies an important
comparison. Earlier strong-solution constructions of Elgindi and Jeong
used conical or corner domains \cite{ElgindiJeong2019}. In the whole
space, Elgindi constructed finite-time singularities for unforced,
axisymmetric Euler flows without swirl when $\alpha$ is sufficiently
small \cite{ElgindiEuler2021}. Elgindi, Ghoul, and Masmoudi subsequently
obtained finite-energy, locally self-similar examples and stability in the
same regularity class \cite{ElgindiGhoulMasmoudi2021}.

A separate boundary mechanism originated in the numerical scenario of
Luo and Hou \cite{LuoHouPNAS2014,LuoHouMMS2014}. For axisymmetric Euler
with swirl in a cylindrical domain, the boundary stagnation structure is
related to a two-dimensional Boussinesq system. Chen and Hou proved
blow-up for $C^{1,\alpha}$ data in this boundary setting
\cite{ChenHou2021} and later analyzed stability and instability of the
associated singular solutions \cite{ChenHou2024}.

Among boundary-free constructions, C\'ordoba, Mart\'inez-Zoroa, and Zheng
produced an infinitely layered, unforced solution that is globally
$C^{1,\alpha}$, smooth away from one point, and singular in finite time
\cite{CordobaMartinezZoroaZheng2025}. Elgindi and Pasqualotto announced a
different mechanism driven by Taylor--Couette instability
\cite{ElgindiPasqualotto2023}. Chen showed that the earlier Elgindi and
Chen--Hou constructions can be modified so that the velocity is nonsmooth
only at the singular point \cite{Chen2024Smoothness}. For every
\(\varepsilon>0\), C\'ordoba and Mart\'inez-Zoroa also constructed a forced,
nonaxisymmetric \(C^{1,\alpha}\) singularity with a force uniformly bounded
in \(C^{1,1/2-\varepsilon}\cap L^2\)
\cite{CordobaMartinezZoroa2023}.

Two 2026 developments reach the sharp range below the classical threshold.
For every $0<\alpha<1/3$, Shkoller announced Type-I blow-up for unforced,
finite-energy, axisymmetric no-swirl flow with $C^{1,\alpha}$ initial
velocity \cite{Shkoller2026}. Independently, Chen's companion preprints
construct exact $C^\alpha$ self-similar vorticity profiles and
asymptotically self-similar blow-up with locally $C^\alpha$ vorticity and
locally $C^{1,\alpha}$ velocity for the same range
\cite{Chen2026PartI,Chen2026PartII}. These results approach
$C^{1,1/3}$ from below; neither claims the endpoint.

\paragraph{Mollification as a methodological precedent.}
The use of an intermediate smoothing operator has a direct methodological
precedent in the admissible-Euler convex-integration construction of
Buckmaster, De Lellis, Sz\'ekelyhidi, and Vicol
\cite[Section~2.4]{BuckmasterDeLellisSzekelyhidiVicol2019}. In that paper,
an Euler--Reynolds pair is mollified at a scale $\ell$ before the gluing
and perturbation stages to control derivative loss typical of
convex-integration schemes. In the present IPM argument, nested Fourier
cutoffs play an analogous bookkeeping role, but the exact cutoff
commutators are retained in the force. The analogy is methodological
only; it does not invoke convex integration or a weak-nonuniqueness
conclusion.

\paragraph{PINN discovery of candidate blow-up profiles.}
Physics-informed neural networks have recently been used as targeted
discovery tools for self-similar singular profiles. Wang, Lai,
G\'omez-Serrano, and Buckmaster found numerical profiles for the
two-dimensional Boussinesq and three-dimensional axisymmetric Euler
equations, as well as an unstable profile for the CCF model
\cite{WangLaiGomezSerranoBuckmaster2023}. A subsequent large-scale search
reported stable and unstable self-similar profiles for IPM with boundary
and for Boussinesq, with the latter linked to axisymmetric Euler with
boundary \cite{WangEtAl2025Discovery}. Wang, L\'eger, Lai, and Buckmaster
later introduced a gradient-normalized residual to resolve sharp profiles
more accurately and numerically identified a fourth unstable IPM
candidate \cite{WangLegerLaiBuckmaster2025}.

\section{Periodic multipliers, nested cutoffs, and transported oscillations}
\label{sec:velocity-and-jets}

For \(\alpha\in\mathbb R\), set
\[
  e(\alpha)=(\cos\alpha,\sin\alpha),
  \qquad
  e^\perp(\alpha)=(-\sin\alpha,\cos\alpha).
\]
To avoid overloading the notation later, we write \(C_0=2\pi\).

\subsection{The periodic operator and parity}

The Fourier multiplier symbol in \eqref{eq:periodic-velocity-symbol} is real,
even, homogeneous
of degree zero away from the origin, and orthogonal to the frequency.
Consequently \(u_{\mathbb T}\) commutes with translations and spatial
derivatives and maps real scalars to real divergence-free vector fields.

Define the scalar operator \(V_{\mathbb T}\) by
\[
  \widehat{V_{\mathbb T}f}(k)
  =\frac{2\pi\mathrm{i} k_1}{|k|^2}\widehat f(k)
  \quad(k\ne0),
  \qquad
  \widehat{V_{\mathbb T}f}(0)=0.
\]
Writing
\(u_{\mathbb T}(\rho)
=(u_{\mathbb T,1}(\rho),u_{\mathbb T,2}(\rho))\)
for its two vector components, we have
\begin{equation}
  u_{\mathbb T,1}(\rho)=-V_{\mathbb T}(\partial_2\rho),
  \qquad
  u_{\mathbb T,2}(\rho)=V_{\mathbb T}(\partial_1\rho).
  \label{eq:periodic-velocity-from-V}
\end{equation}

Write \(A\Subset B\) when \(\overline A\) is a compact subset of \(B\).
Fix symmetric coordinate balls
\begin{equation}
  U_0\Subset U\Subset(-\pi,\pi)^2,
  \label{eq:periodic-coordinate-balls}
\end{equation}
centered at the origin. Every function supported in \(U_0\) below is
identified with its local lift; in Euclidean expressions, this lift is
extended by zero to \(\mathbb R^2\). Fix an
even function \(\eta_{\rm lin}\in C_c^\infty(U)\) that equals one on
a neighborhood of \(\overline{U_0}\).  For \(p\in\mathbb R^2\), let
\(\psi_p\in C^\infty(\mathbb T^2)\) be the periodic function obtained by
extending
\[
  x\longmapsto\eta_{\rm lin}(x)(p\cdot x)
\]
by zero from the chart.  Thus \(\psi_p\) is odd and agrees with the linear
phase \(p\cdot x\) near \(\overline{U_0}\). Extend the periodic Fourier
multiplier symbol homogeneously to
\begin{equation}
  m(\xi)
  =2\pi\left(\frac{\xi_1\xi_2}{|\xi|^2},
             -\frac{\xi_1^2}{|\xi|^2}\right),
  \qquad \xi\in\mathbb R^2\setminus\{0\}.
  \label{eq:homogeneous-velocity-symbol}
\end{equation}

For \(\lambda>0\), \(\alpha,\vartheta\in\mathbb R\), and a smooth amplitude
\(b=b(x)\) compactly supported in \(U_0\), set
\[
  z(x)=b(x)\sin(\lambda e(\alpha)\cdot x+\vartheta).
\]
Define the zeroth-order approximation to the velocity by
\begin{equation}
  u^{[0]}[z]
  =-C_0\cos\alpha\,e^\perp(\alpha)z.
  \label{eq:leading-wave-response}
\end{equation}
Indeed, evaluating the homogeneous Fourier multiplier symbol
\eqref{eq:homogeneous-velocity-symbol}
at \(e(\alpha)\)
gives
\[
  2\pi(\cos\alpha\sin\alpha,-\cos^2\alpha)
  =-C_0\cos\alpha\,e^\perp(\alpha).
\]
The high-order localized version of this calculation is proved below.

\subsection{Nested Fourier cutoffs}

Fix an even radial function \(\chi\in C_c^\infty(\mathbb R^2)\) such that
\[
  \chi(\xi)=1\quad(|\xi|\leq1),
  \qquad
  \chi(\xi)=0\quad(|\xi|\geq2).
\]
For \(\mu\geq1\), define
\begin{equation}
  \widehat{S_\mu f}(k)=\chi(k/\mu)\widehat f(k),
  \qquad k\in\mathbb Z^2.
  \label{eq:periodic-cutoff}
\end{equation}

For later use, define the cutoff commutator by
\begin{equation}
  \mathcal C_\mu(h)
  :=u_{\mathbb T}(S_\mu h)\cdot\nabla S_\mu h
    -S_\mu\bigl(u_{\mathbb T}(h)\cdot\nabla h\bigr).
  \label{eq:periodic-cutoff-commutator-definition}
\end{equation}

If \(f\) is odd, then
\(u_{\mathbb T}(f)\cdot\nabla f\) is odd and has zero mean.  The same is
true after applying any cutoff.  These observations account both for the
global zero mode and for the exact nonlinear Fourier support used later in the
force recursion. Proofs of the Bernstein-type estimates for $S_\mu$ and $\mathcal{C}_\mu$ used throughout this work may be found in the appendix.

\subsection{The local periodic kernel}

For \(f\in C^\infty(\mathbb T^2)\) supported in \(U_0\), define
\begin{equation}
  V_{\mathbb R}f(x)
  =\int_{\mathbb R^2}\frac{h_1}{|h|^2}f(x+h)\,dh,
  \label{eq:euclidean-V}
\end{equation}
acting on the extension by zero to \(\mathbb R^2\) of the local lift of \(f\).

\noindent\emph{Remark.} For \(f\) supported in \(U_0\), we decompose
\(V_{\mathbb T}f\) on \(U\) into \(V_{\mathbb R}f\) plus an integral of
\(f\) against a smooth kernel; this decomposition is proved in the appendix
in Lemma~\ref{lem:periodic-kernel-decomposition}.

\subsection{Affine high-frequency velocity expansion}

The results below isolate the oscillatory inputs used in the layer
construction.  An affine expansion identifies the local multiplier response;
transport keeps the nonlinear phase nearly affine on the perturbation support;
the periodic expansion separates finitely many correctable terms from a small
global remainder; and a common transported phase yields the anisotropic gain
in quadratic interactions.

Choose \(\kappa_0\in C_c^\infty([0,\infty))\), equal to one near zero, so
that \(h\mapsto\kappa_0(|h|)\) is smooth on \(\mathbb R^2\). For a unit
vector \(\omega=(\omega_1,\omega_2)\) and a multi-index
\(\gamma\), set
\[
  \mathcal K_\gamma(h)
  =\frac{h_1h^\gamma}{|h|^2\gamma!},
\]
and define
\begin{align}
  c_\gamma^s(\omega)
  &=\lim_{\sigma\downarrow0}
    \int_{\mathbb R^2}\mathcal K_\gamma(h)
      \cos(\omega\cdot h)\kappa_0(\sigma|h|)\,dh,
  \label{eq:periodic-profile-s}\\
  c_\gamma^c(\omega)
  &=\lim_{\sigma\downarrow0}
    \int_{\mathbb R^2}\mathcal K_\gamma(h)
      \sin(\omega\cdot h)\kappa_0(\sigma|h|)\,dh.
  \label{eq:periodic-profile-c}
\end{align}
For a scalar-valued profile \(f\), an affine phase
\(\Phi=\nu\omega\cdot x+\vartheta\), and an integer \(N\geq0\), put
\[
  \mathcal V_\Phi^N[f]
  =\sum_{|\gamma|\leq N}
    \nu^{-|\gamma|-1}\bigl(\partial^\gamma f\bigr)\cdot
    \bigl(c_\gamma^s(\omega)\sin\Phi
          +c_\gamma^c(\omega)\cos\Phi\bigr),
  \qquad
  \mathcal V_\Phi^{-1}=0.
\]
Define the Euclidean singular velocity by
\[
  u_{\mathbb R}(h)
  =\bigl(-V_{\mathbb R}(\partial_2h),
          V_{\mathbb R}(\partial_1h)\bigr).
\]
The subscript \(\mathrm{pw}\) below indicates that the expansion is built from
the affine plane wave \(f\sin\Phi\).
The order-\(N\) truncated high-frequency velocity expansion is
\begin{align}
  u_{\mathrm{pw},1}^{[N]}[f;\Phi]
  &=-\mathcal V_{\Phi+\pi/2}^N[\nu\omega_2 f]
    -\mathcal V_\Phi^{N-1}[\partial_2f],
  \label{eq:periodic-plane-wave-u1}\\
  u_{\mathrm{pw},2}^{[N]}[f;\Phi]
  &=\mathcal V_{\Phi+\pi/2}^N[\nu\omega_1 f]
    +\mathcal V_\Phi^{N-1}[\partial_1f].
  \label{eq:periodic-plane-wave-u2}
\end{align}

\begin{lemma}[Affine high-frequency velocity expansion]
\label{lem:affine-velocity-expansion}
The limits in \eqref{eq:periodic-profile-s}--
\eqref{eq:periodic-profile-c} exist uniformly in \(\omega\in\mathbb S^1\)
and define smooth functions on \(\mathbb S^1\).  They satisfy
\begin{equation}
  c_\gamma^s=0\quad(|\gamma|\text{ even}),
  \qquad
  c_\gamma^c=0\quad(|\gamma|\text{ odd}),
  \qquad
  c_0^s=0,\quad c_0^c(\omega)=2\pi\omega_1.
  \label{eq:periodic-profile-identities}
\end{equation}
Let \(N,J\geq0\), \(0<\varepsilon<1\), \(\nu\geq2\),
\(\omega\in\mathbb S^1\), \(\vartheta\in\mathbb R\), and a
scalar-valued profile \(f\in C_c^\infty(U_0)\). For
\(\Phi=\nu\omega\cdot x+\vartheta\),
\begin{align}
 \|V_{\mathbb R}(f\sin\Phi)-\mathcal V_\Phi^N[f]\|_{C^J(\mathbb R^2)}
 \leq C_{N,J,\varepsilon}\bigg(
  \sum_{d=0}^N\nu^{d\varepsilon+J-N}
      \|f\|_{C^{N-d+J}}
  +\nu^{J-(N+2)\varepsilon}\|f\|_{C^{N+J+1}}
 \bigg).
 \label{eq:periodic-scalar-master-error}
\end{align}
For \(N\geq0\), define
\[
  \mathcal R_\Phi^N[f]
  :=V_{\mathbb R}(f\sin\Phi)-\mathcal V_\Phi^N[f].
\]
If \(N\geq1\), then
\begin{align*}
 u_{\mathbb R}(f\sin\Phi)-u_{\mathrm{pw}}^{[N]}[f;\Phi]
 =\bigl(&-\mathcal R_{\Phi+\pi/2}^N[(\nu\omega_2)f]
         -\mathcal R_\Phi^{N-1}[\partial_2f],\\
        &\mathcal R_{\Phi+\pi/2}^N[(\nu\omega_1)f]
         +\mathcal R_\Phi^{N-1}[\partial_1f]\bigr).
\end{align*}
Thus \(u_{\mathrm{pw}}^{[N]}[f;\Phi]\) approximates the Euclidean singular
velocity of the localized oscillation \(f\sin\Phi\), with each scalar
remainder controlled by \eqref{eq:periodic-scalar-master-error} at order
\(N\) or \(N-1\).  It does not approximate the oscillatory scalar itself.
\end{lemma}

\begin{proof}
Subtract two radial cutoffs in
\eqref{eq:periodic-profile-s}--\eqref{eq:periodic-profile-c} and integrate in
a coordinate for which \(|\omega_a|\geq1/2\).  For every \(R>0\), more than
\(|\gamma|+2+R\) integrations by parts make the difference
\(O(\sigma^R)\), uniformly on each of two direction charts.  Differentiating
in \(\omega\) proves smooth uniform convergence.  Reflection
\(h\mapsto-h\) gives the first two identities in
\eqref{eq:periodic-profile-identities}.  The Fourier multiplier of
\eqref{eq:euclidean-V} is
\(2\pi\mathrm{i}\xi_1|\xi|^{-2}\); applying it to
\(\sin(\omega\cdot x)\) gives
\(2\pi\omega_1\cos(\omega\cdot x)\), proving the last two identities.

Set \(\kappa_\nu(h)=\kappa_0(\nu^\varepsilon|h|)\) and split
\eqref{eq:euclidean-V} smoothly into the terms with factors \(\kappa_\nu\)
and \(1-\kappa_\nu\). In the near-field term, Taylor-expand \(f(x+h)\) through
order \(N\). Scaling \(h\) by \(\nu\) gives
\(\mathcal V_\Phi^N[f]\); after \(J\) output derivatives, the Taylor
remainder is bounded by
\[
  C\nu^{J-(N+2)\varepsilon}\|f\|_{C^{N+J+1}}.
\]
In the far-field term, choose \(a\) with \(|\omega_a|\geq1/2\) and use, for its
complex form,
\[
 \int (1-\kappa_\nu)\frac{h_1}{|h|^2}f(x+h)
       e^{\mathrm{i}\nu\omega\cdot h}\,dh
 =\frac{(-1)^N}{(\mathrm{i}\nu\omega_a)^N}
   \int e^{\mathrm{i}\nu\omega\cdot h}
     \partial_{h_a}^N\!\left[(1-\kappa_\nu)
       \frac{h_1}{|h|^2}f(x+h)\right]dh.
\]
There are no boundary terms because the split is smooth and the profile is
compactly supported. When \(d\) derivatives hit the cutoff or kernel, the
cutoff cost is \(\nu^{d\varepsilon}\), and the contribution is bounded by
\[
  C\nu^{d\varepsilon+J-N}\|f\|_{C^{N-d+J}},
  \qquad 0\leq d\leq N.
\]
After scaling the near-field term, its cutoff is
\(\kappa_0(\nu^{\varepsilon-1}|h|)\); hence the finite-cutoff moments differ
from \eqref{eq:periodic-profile-s}--\eqref{eq:periodic-profile-c} by an
arbitrarily negative power of \(\nu\). This proves
\eqref{eq:periodic-scalar-master-error}.  Finally,
\[
  \partial_j(f\sin\Phi)
  =(\nu\omega_j)f\sin(\Phi+\pi/2)+(\partial_jf)\sin\Phi.
\]
Apply the scalar identity at orders \(N\) and \(N-1\) to obtain the displayed
velocity remainder.
\end{proof}

At a fixed time, let \(N\geq0\), \(\lambda>0\), \(\ell\in\mathbb N\),
\(\vartheta\in\mathbb R\), \(X\in C^\infty(U;\mathbb R)\) with
\(A:=\nabla X(0)\ne0\), and \(b\in C_c^\infty(U_0)\). Set
\[
  Y_\ell=\ell(X-A\cdot x),\qquad
  \nu_\ell=\lambda\ell|A|,\qquad
  \omega=A/|A|,
  \qquad
  \Phi_{\ell,\vartheta}=\nu_\ell\omega\cdot x+\vartheta.
\]
For the chosen sine-band presentation \((b,\ell,\vartheta)\), define
\begin{align}
  u^{[N]}\!\left[b\sin(\lambda\ell X+\vartheta)\right]
  ={}&u_{\mathrm{pw}}^{[N]}
       [b\cos(\lambda Y_\ell);\Phi_{\ell,\vartheta}]
       \notag\\
    &+u_{\mathrm{pw}}^{[N]}
       [b\sin(\lambda Y_\ell);
          \Phi_{\ell,\vartheta}+\pi/2].
  \label{eq:periodic-nonlinear-approximant}
\end{align}
Cosine bands are defined by shifting \(\vartheta\) by \(\pi/2\), and finite
presentations are treated band by band. Thus \(u^{[N]}\) here denotes a
quantity defined from the chosen presentation, rather than an operator
determined only by the resulting scalar function. Formula
\eqref{eq:periodic-nonlinear-approximant} is the exact sine-addition split;
no affine approximation of the density is made. When the data depend smoothly
on time, this definition is applied at each fixed time.

\subsection{Transported phase and oscillatory products}

For \(a\in\mathbb R\), write \(a_+=\max\{a,0\}\).

The next lemma assumes an odd parent with compact Fourier support and with scale-local derivative
bounds and an odd transported phase that is nearly affine on \(B_r\). It
controls the nonlinear phase remainder \(Y_\ell\) and the time dependence of
the induced frequency \(\nu_\ell\) and direction \(\alpha\); the following
corollary then converts these bounds into estimates for localized oscillatory
products.

\begin{lemma}[Transported-phase estimates]
\label{lem:transported-phase-estimates}
Let \(0<\beta\leq1/8\), \(\lambda\geq e\), and set
\[
  L=\log\lambda,\qquad
  \mu=\lambda^{\beta/2},\qquad
  M=\lceil\beta^{-2}\rceil,\qquad
  r=\lambda^{-2\beta}L^4.
\]
Let \(I=[s,1]\), where \(0\leq s\leq1\), and let \(B_r\) denote the
Euclidean ball of radius \(r\) centered at the origin. Assume
\(\overline{B_r}\subset U_0\). Fix integers \(T,K\geq0\), and define
\begin{equation}
  N_{\mathrm{prof}}=K+M+T+2,\qquad
  N_{\mathrm{phase}}=N_{\mathrm{prof}}+1.
  \label{eq:periodic-packet-orders}
\end{equation}
The parent estimates use derivatives through \(N_{\mathrm{prof}}\), whereas
the phase estimates use one additional derivative through
\(N_{\mathrm{phase}}\).  All constants in the hypotheses below are
independent of \(\lambda\) and the harmonic index \(\ell\).
Let \(P\in C^\infty(\mathbb T^2\times I)\) be odd in space and satisfy
\(\operatorname{supp}\widehat{P(\cdot,t)}\subset\{|k|\leq2\mu\}\) for
every \(t\in I\).  Put
\[
  v=u_{\mathbb T}(P),\qquad
  D=\partial_t+v\cdot\nabla,\qquad
  \mathcal G=\nabla P,
\]
and suppose on \(B_r\times I\) that
\begin{align}
  \|D^pv\|_{C^0(B_r)}
  &\leq C_pL^{C_p}r,
  \notag\\
  \|D^pv\|_{C^{j+1}(B_r)}
  &\leq C_{p,j}L^{C_{p,j}}\mu^j,
  \label{eq:periodic-packet-velocity-input}
\end{align}
for \(0\leq p\leq(T-1)_+\), \(0\leq j\leq N_{\mathrm{prof}}\), and
\begin{equation}
  \|D^p\mathcal G\|_{C^j(B_r)}
  \leq C_{p,j}L^{C_{p,j}}\mu^j
  \quad(0\leq p\leq T,\ p+j\leq N_{\mathrm{prof}}).
  \label{eq:periodic-packet-gradient-input}
\end{equation}
Let \(X\in C^\infty(U\times I;\mathbb R)\) be odd in space and satisfy on
\(B_r\times I\)
\[
  DX=0,\qquad
  L^{-C}\leq|\nabla X|\leq L^C,
\]
and
\begin{align}
  \|\nabla^jX\|_{C^0(B_r)}
  &\leq C_jL^{C_j}\mu^{(j-1)_+}
  &&(1\leq j\leq N_{\mathrm{phase}}),
  \notag\\
  \|X-\nabla X(0,t)\cdot x\|_{C^0(B_r)}
  &\leq L^C\lambda^{-4\beta},
  \notag\\
  \|\nabla X-\nabla X(0,t)\|_{C^0(B_r)}
  &\leq L^C\lambda^{-2\beta}.
  \label{eq:periodic-packet-phase-input}
\end{align}
The Fourier-support condition on \(P\) records how
\eqref{eq:periodic-packet-velocity-input}--
\eqref{eq:periodic-packet-gradient-input} are obtained later; the arguments
from this lemma through Corollary~\ref{cor:common-phase-interactions} use
those displayed bounds and oddness, not that support condition itself. The
last two lines of \eqref{eq:periodic-packet-phase-input} also follow from its
first line and oddness by Taylor's theorem, but are recorded in the form used
below.
Set
\[
  A(t)=\nabla X(0,t),\qquad
  Y_\ell=\ell(X-A(t)\cdot x),\qquad
  \nu_\ell=\lambda\ell|A(t)|.
\]
Assume that \(\nu_\ell\geq2\) on \(I\) for every integer
\(1\leq\ell\leq2^{M+1}\).

Choose the continuous angle \(\alpha\) by
\[
  \frac{\nabla X(0,t)}{|\nabla X(0,t)|}=e(\alpha(t)),
\]
with \(\alpha(1)\in[-\pi,\pi)\). One has
\begin{align}
  \max_{\mathcal D\in\{D,\partial_t\}}
  \|\mathcal D^pY_\ell\|_{C^0(B_r)}
  &\leq C_pL^{C_p}\lambda^{-4\beta},
  \notag\\
  \max_{\mathcal D\in\{D,\partial_t\}}
  \|\mathcal D^p\nabla Y_\ell\|_{C^j(B_r)}
  &\leq C_{p,j}L^{C_{p,j}}
     \lambda^{-2\beta+2j\beta},
  \label{eq:periodic-packet-phase-block}\\
  \nu_\ell^{-1}|\partial_t^p\nu_\ell|
  +|\partial_t^p\alpha|
  &\leq C_pL^{C_p}.
  \label{eq:periodic-packet-parameter-path}
\end{align}
Here \(0\leq p\leq T\), and \(p+j\leq N_{\mathrm{prof}}\) in the second
line.  Also
\(L^{-C}\lambda\leq\nu_\ell\leq C_\beta L^C\lambda\).
The constants depend only on \(\beta,T,K\), the fixed chart and cutoff data,
and the finitely many constants in the hypotheses.  They are independent of
\(\lambda\) and \(\ell\). In the iteration,
\(\beta\in\{2^{-3-k}:k\geq0\}\); for \(\beta\geq\beta_0>0\), all constants
and logarithmic exponents here and in the results below may therefore be
replaced by their maximum over this finite set.
\end{lemma}

\begin{proof}
Oddness of \(X\) gives zero constant and quadratic jets, and
\(Y_\ell=\ell(X-A\cdot x)\) has zero linear jet. The
spatial conclusions at \(p=0\) follow from
\eqref{eq:periodic-packet-phase-input}, Taylor's theorem, and
\(\mu=\lambda^{\beta/2}\).  Differentiating \(DX=0\) at the origin gives
\[
  \dot A=-D_xv(0,t)^{\mathsf T}A,
\]
and hence
\begin{equation}
  DY_\ell
  =-\ell A(t)\cdot\bigl(v(x,t)-D_xv(0,t)x\bigr).
  \label{eq:periodic-material-phase-identity}
\end{equation}
Every material derivative of the right side is odd with zero linear jet.
For the material estimates, induct on \(n=p+j\), where \(j\) is the
\(C^j\)-order of \(\nabla Y_\ell\). For \(p>0\), differentiate
\eqref{eq:periodic-material-phase-identity} \(p-1\) times and commute the
spatial derivatives using
\[
  [D,\partial_a]=-(\partial_av)\cdot\nabla
\]
together with \eqref{eq:periodic-packet-velocity-input}. Every parent factor
has material order at most \(p-1\). Every phase factor has material order at
most \(p-1\) and spatial order at most \(j\), hence total order at most
\(n-1\). Thus no output pair is used to prove itself. The largest parent
demand occurs when all spatial derivatives
land on the velocity remainder in
\eqref{eq:periodic-material-phase-identity}; it is
\(D^{p-1}v\) in \(C^{j+1}\), which lies in
\eqref{eq:periodic-packet-velocity-input}. At low spatial order,
\[
  \|D^pY_\ell\|_{C^0(B_r)}
  \leq L^C\mu^2r^3
  \leq L^C\lambda^{-4\beta},
  \qquad
  \|\nabla D^pY_\ell\|_{C^0(B_r)}
  \leq L^C\mu^2r^2
  \leq L^C\lambda^{-2\beta}.
\]
The Hessian is odd and therefore vanishes at the origin, so
\[
  \|\nabla^2D^pY_\ell\|_{C^0(B_r)}
  \leq L^C\mu^2r
  \leq L^C.
\]
For \(j\geq2\), \(\mu^j\leq\lambda^{-2\beta+2j\beta}\). At the top spatial
endpoint \((p,j)=(0,N_{\mathrm{prof}})\), all spatial derivatives may land on
the phase, requiring \(\nabla^{N_{\mathrm{prof}}+1}X\). This is exactly the
one additional derivative in
\(N_{\mathrm{phase}}=N_{\mathrm{prof}}+1\). At the opposite endpoint
with \(T\geq1\), the material identity uses only \(D^{T-1}v\); when
\(T=0\), the positive-material-order induction is vacuous.

To obtain ordinary-time estimates, iterate
\(\partial_t=D-v\cdot\nabla\). The largest phase demand occurs when all
transport factors land on the phase block, reaching spatial order \(j+p\).
Each accompanying undifferentiated velocity costs only
\(r\lambda^{2\beta}=L^4\); if a derivative instead lands on a velocity
factor, the phase order decreases. Hence the diagonal restriction
\(p+j\leq N_{\mathrm{prof}}\) is sufficient. The origin equation gives
\eqref{eq:periodic-packet-parameter-path}.  This proves
\eqref{eq:periodic-packet-phase-block}.
\end{proof}

\begin{corollary}[Oscillatory-product estimates]
\label{cor:oscillatory-product-estimates}
Assume the hypotheses of Lemma~\ref{lem:transported-phase-estimates}.  Let
\(\mathcal A>0\) and \(b\in C^\infty(U\times I)\), with
\[
  \bigcup_{t\in I}\operatorname{supp}b(\cdot,t)\Subset B_r,
\]
and suppose that
\begin{equation}
  \|D^pb\|_{C^j(B_r)}
  \leq C_{p,j}L^{C_{p,j}}\mathcal A\lambda^{2j\beta}
  \quad(0\leq p\leq T,\ p+j\leq N_{\mathrm{prof}}).
  \label{eq:periodic-packet-coefficient-input}
\end{equation}
The constants in this hypothesis are independent of
\(\lambda,\ell,\vartheta\), and \(\mathcal A\).  For
\(\mathfrak t\in\{\sin,\cos\}\), \(\vartheta\in\mathbb R\), and an integer
\(1\leq\ell\leq2^{M+1}\), let
\[
  z=b\,\mathfrak t(\lambda\ell X+\vartheta),
\]
extended by zero from the coordinate chart.  For \(0\leq p\leq T\) and
\(0\leq j\leq K\),
\begin{align}
  \|D^pz\|_{C^j(B_r)}
  &\leq C_{p,j}L^{C_{p,j}}\mathcal A\lambda^j,
  \notag\\
  \|\partial_t^pz\|_{C^j(B_r)}
  &\leq C_{p,j}L^{C_{p,j}}\mathcal A
    \lambda^{j+p(1-2\beta)}.
  \label{eq:periodic-packet-product}
\end{align}
The conclusion constants depend only on \(\beta,T,K\), the fixed chart and
cutoff data, and the finitely many constants in the hypotheses.  They are
independent of \(\lambda,\ell,\vartheta\), and \(\mathcal A\).
\end{corollary}

\begin{proof}
Leibniz' rule and \(DX=0\) give the material derivative estimate in
\eqref{eq:periodic-packet-product}.  Converting to ordinary time derivatives costs
\(r\lambda=L^4\lambda^{1-2\beta}\) whenever a derivative hits the
oscillatory factor.
\end{proof}

\subsection{Localized periodic velocity expansion}

For the data of Corollary~\ref{cor:oscillatory-product-estimates}, a
\emph{localized oscillatory packet} is a scalar
\[
  z=b\,\mathfrak t(\lambda\ell X+\vartheta),
  \qquad
  \mathfrak t\in\{\sin,\cos\},
  \qquad \ell\in\mathbb Z,\quad 1\leq\ell\leq2^{M+1},
\]
with \(b\) supported in \(U_0\) and satisfying the coefficient bounds in
that corollary.

\begin{proposition}[Localized periodic velocity expansion]
\label{prop:localized-periodic-velocity-expansion}
Assume the hypotheses of
Corollary~\ref{cor:oscillatory-product-estimates}, and use its notation.  For
a localized oscillatory packet \(z\), the truncated expansion in
\eqref{eq:periodic-nonlinear-approximant} has a finite-order decomposition
at the same phase. The index \(n\) is the net power of
\(\nu_\ell^{-1}\), namely
\(n=|\gamma|\) in the branch containing \(\nu_\ell\omega_a f\) and
\(n=|\gamma|+1\) in the branch containing \(\partial_af\).
\begin{equation}
  u^{[M]}[z]
  =\sum_{n=0}^M
   \bigl(B_n^s\sin(\lambda\ell X+\vartheta)
        +B_n^c\cos(\lambda\ell X+\vartheta)\bigr),
  \label{eq:periodic-packet-strata}
\end{equation}
where the coefficients are supported in \(\operatorname{supp}b\), extended
by zero from the coordinate chart.  All unqualified \(C^j\) norms of these
coefficients are periodic norms.  For
\(0\leq p\leq T\), \(0\leq j\leq K\), and \(0\leq n\leq M\),
\begin{equation}
  \max_{\mathcal D\in\{D,\partial_t\}}
  \bigl(\|\mathcal D^pB_n^s\|_{C^j}
       +\|\mathcal D^pB_n^c\|_{C^j}\bigr)
  \leq C_{p,j,n}L^{C_{p,j,n}}\mathcal A
    \lambda^{2(j-n)\beta}.
  \label{eq:periodic-packet-stratum-bound}
\end{equation}
For the same \(p,j\), the exact periodic velocity satisfies
\begin{equation}
  \left\|\partial_t^p
    (u_{\mathbb T}-u^{[M]})[z]\right\|_{C^j(\mathbb T^2)}
  \leq C_{p,j}L^{C_{p,j}}\mathcal A
    \lambda^{p+2(j+1)(1-\beta)-M\beta}.
  \label{eq:periodic-packet-leak}
\end{equation}
No material derivative estimate or local phase-harmonic presentation is asserted for
this grouped exact remainder.
The constants have the dependence stated in
Corollary~\ref{cor:oscillatory-product-estimates}.
\end{proposition}

\begin{proof}
Apply Lemma~\ref{lem:affine-velocity-expansion} with
\(\varepsilon=1-\beta\) to the exact sine-addition decomposition
\eqref{eq:periodic-nonlinear-approximant}. For each affine branch, choose the
coefficients recursively, with \(a\in\{1,2\}\) and \(e_a\) the corresponding
coordinate multi-index, by
\begin{align*}
  C_{\gamma+e_a}^s
  &=\partial_aC_\gamma^s
    -\lambda(\partial_aY_\ell)C_\gamma^c,\\
  C_{\gamma+e_a}^c
  &=\partial_aC_\gamma^c
    +\lambda(\partial_aY_\ell)C_\gamma^s,
\end{align*}
starting from \((C_0^s,C_0^c)=(b,0)\) for a sine input and \((0,b)\) for a
cosine input. Order the pairs
\((|\gamma|,j)\) lexicographically by
\((|\gamma|+j,|\gamma|)\). The derivative term at
\((|\gamma|+1,j)\) uses the preceding coefficient in \(C^{j+1}\), on the
same diagonal but with smaller \(|\gamma|\), so this recursion is noncircular.
The largest regularity demand occurs when these derivative terms put all
\(|\gamma|+j\) derivatives on \(b\); it is exactly the range in
\eqref{eq:periodic-packet-coefficient-input}. At derivative order \(d\), the
carrier term has the largest frequency exponent, since
\(\|\lambda\nabla Y_\ell\|_{C^d}
\leq L^C\lambda^{1-2\beta+2d\beta}\), whereas the derivative term costs
only \(\lambda^{2\beta}\leq\lambda^{1-2\beta}\). Leibniz' rule therefore
gives, for \(|\gamma|\leq M\) and
\(|\gamma|+j\leq N_{\mathrm{prof}}\),
\[
 \partial^\gamma[b\,\mathfrak t(\lambda Y_\ell)]
 =C_\gamma^s\sin(\lambda Y_\ell)
  +C_\gamma^c\cos(\lambda Y_\ell),
 \qquad
 \|C_\gamma^{s,c}\|_{C^j}
 \leq L^C\mathcal A
   \lambda^{|\gamma|(1-2\beta)+2j\beta}.
\]
An order-\(n\) term in the affine expansion contains \(\nu_\ell^{-n}\), so
it is at most
\(L^C\mathcal A\lambda^{2(j-n)\beta}\). To see the exact recombination,
put \(g=b e^{\mathrm{i}\lambda Y_\ell}\) and
\(\Phi=\Phi_{\ell,\vartheta}\). For each scalar branch,
\begin{equation}
 \mathcal V_\Phi^N[\operatorname{Re}g]
 +\mathcal V_{\Phi+\pi/2}^N[\operatorname{Im}g]
 =\sum_{|\gamma|\leq N}\nu_\ell^{-|\gamma|-1}
   \operatorname{Im}\!\left[
     (c_\gamma^s+\mathrm{i}c_\gamma^c)
     \partial^\gamma g\,e^{\mathrm{i}\Phi}\right].
 \label{eq:periodic-exact-recombination}
\end{equation}
Thus the terms with phase \(\lambda Y_\ell-\Phi\) cancel index by index,
and the remaining oscillatory factor is
\(\mathfrak t(\lambda\ell X+\vartheta)\). Since \(DX=0\), its material
derivative vanishes, so material derivatives fall only on the recombined
coefficients. Equivalently, these coefficients are finite linear combinations
of products of smooth functions of \(\nu_\ell\) and \(\alpha\) from the
affine moments with the real and imaginary parts of
\[
  e^{-\mathrm{i}\lambda Y_\ell}\partial^\gamma
    (b e^{\mathrm{i}\lambda Y_\ell}).
\]
When \(D\) hits this expression, the term from the first exponential cancels
the term in which \(D\) hits the second exponential without a spatial
derivative falling on \(DY_\ell\). Every remaining term contains a material
derivative of \(b\) or one of those moment factors, a commutator derivative
of \(v\), or a spatial derivative of \(DY_\ell\). The coefficient input,
\eqref{eq:periodic-material-phase-identity},
Lemma~\ref{lem:transported-phase-estimates}, and
\eqref{eq:periodic-packet-parameter-path} therefore close a triangular
induction in material order. At the endpoint it uses
\(p+|\gamma|+j\leq T+M+K=N_{\mathrm{prof}}-2\).
Writing \(\partial_t=D-v\cdot\nabla\), and using the
\(C^0\) bound for \(v\), the spatial coefficient bounds, and
\eqref{eq:periodic-packet-parameter-path}, gives the same estimate for
ordinary-time derivatives. This conversion uses only coefficient pairs on
the diagonal \(p'+j'\leq T+K\), which are supplied by the preceding
induction. Lemma~\ref{lem:transported-phase-estimates} and
Corollary~\ref{cor:oscillatory-product-estimates} therefore give
\eqref{eq:periodic-packet-strata} and
\eqref{eq:periodic-packet-stratum-bound} under material and ordinary
derivatives.

For \(p\geq1\), we differentiate the three-part proof of
\eqref{eq:periodic-scalar-master-error} term by term, rather than
differentiating its fixed-parameter statement. After \(p\) ordinary-time
and \(j\) spatial derivatives, every singular remainder is a far
integration-by-parts term, a near Taylor remainder, or a differentiated
moment tail. On \(\operatorname{supp}b\subset B_r\), a derivative of the
affine phase costs at most \(L^C\lambda r\), derivatives of its cutoff and
moments cost at most \(L^C\), and derivatives of its profile cost at most
\(L^C\lambda^{1-2\beta}\). The largest regularity demand is the
order-\(M\) Taylor remainder
in the differentiated velocity branch, with all \(p+j\) output derivatives
on that branch. It uses \(M+p+j+2\) coefficient derivatives, exactly
\(N_{\mathrm{prof}}\) at \((p,j)=(T,K)\), and has at
most \(p\) angular or radial parameter derivatives.  Each time derivative
costs at most \(L^C\lambda\). If \(E_{\rm far}\), \(E_{\rm Tay}\), and
\(E_{\rm mom}\) denote the three corresponding contributions to the
singular velocity remainder, substitution in
\eqref{eq:periodic-scalar-master-error} gives, for each
\(\star\in\{{\rm far},{\rm Tay},{\rm mom}\}\),
\[
  \|\partial_t^pE_\star\|_{C^j}
  \leq C_{p,j}L^C\mathcal A
    \lambda^{p+2(j+1)(1-\beta)-M\beta}.
\]

For the smooth periodic part, Lemma~\ref{lem:periodic-kernel-decomposition}
reduces the estimate to finitely many operators
\[
  \mathcal Hz(x,t)
  =\int_{U_0}H(x,y)b(y,t)e^{\mathrm{i}\lambda\ell X(y,t)}\,dy
\]
with smooth \(H\). On a neighborhood of
\(\operatorname{supp}b(\cdot,t)\Subset B_r\), the lower gradient bound in
Lemma~\ref{lem:transported-phase-estimates} makes
\[
  \mathcal L
  =\frac{\nabla_yX}
    {\mathrm{i}\lambda\ell|\nabla_yX|^2}\cdot\nabla_y
\]
well defined. Direct differentiation gives
\[
  \mathcal L(e^{\mathrm{i}\lambda\ell X})
  =\frac{\nabla_yX\cdot(\mathrm{i}\lambda\ell\nabla_yX)}
    {\mathrm{i}\lambda\ell|\nabla_yX|^2}
      e^{\mathrm{i}\lambda\ell X}
  =e^{\mathrm{i}\lambda\ell X}.
\]
The compact support of \(b\) removes the boundary term. Integrate exactly
\(M\) times with the formal transpose of \(\mathcal L\).
Every integration gains \(\lambda^{-1}\) and costs at most
\(L^C\lambda^{2\beta}\); every ordinary-time derivative costs at most
\(L^C\lambda^{1-2\beta}\).  Hence
\[
  \|\partial_t^p\mathcal Hz\|_{C^j(\mathbb T^2)}
  \leq C_{p,j}L^C\mathcal A
    \lambda^{(p-M)(1-2\beta)}.
\]
Only \(M+p\leq M+T\) coefficient derivatives are used.  Since
\[
 (p-M)(1-2\beta)
 \leq p+2(j+1)(1-\beta)-M\beta
\]
for \(\beta\leq1/8\), the smooth and singular contributions prove
\eqref{eq:periodic-packet-leak} within the block
\eqref{eq:periodic-packet-orders}.
\end{proof}

\begin{corollary}[Leading response at the origin]
\label{cor:leading-periodic-response}
Under the hypotheses of
Proposition~\ref{prop:localized-periodic-velocity-expansion}, let \(z\) be a
sine input. Then
\begin{equation}
  B_0^s=-C_0\cos\alpha\,e^\perp(\alpha)b,
  \qquad B_0^c=0,
  \label{eq:periodic-packet-leading}
\end{equation}
and
\begin{equation}
  D_xu^{[0]}[z](0,t)
  =-C_0\cos\alpha(t)e^\perp(\alpha(t))
    \bigl(\nabla z(0,t)\bigr)^{\mathsf T}.
  \label{eq:periodic-packet-origin-jet}
\end{equation}
\end{corollary}

\begin{proof}
At order zero, \eqref{eq:periodic-profile-identities} gives
\eqref{eq:periodic-packet-leading}; differentiating that response gives
\eqref{eq:periodic-packet-origin-jet}.
\end{proof}

\subsection{Common-phase interactions}

\begin{corollary}[Common-phase interaction bounds]
\label{cor:common-phase-interactions}
Assume the parent and phase hypotheses of
Lemma~\ref{lem:transported-phase-estimates} for a single shared
\(P,D,X,\lambda,\beta,T,K\) and the associated scales.  For
\(\iota=1,2\), let \(\mathcal A_\iota>0\) and
\(b_\iota\in C^\infty(U\times I)\), and set
\[
  z_\iota
  =b_\iota\mathfrak t_\iota
     (\lambda\ell_\iota X+\vartheta_\iota),
\]
where \(\ell_\iota\in\mathbb Z\), \(1\leq\ell_\iota\leq2^{M+1}\),
\(\vartheta_\iota\in\mathbb R\),
\(\mathfrak t_\iota\in\{\sin,\cos\}\),
\(\bigcup_{t\in I}\operatorname{supp}b_\iota(\cdot,t)\Subset B_r\), and
\(b_\iota\) satisfies \eqref{eq:periodic-packet-coefficient-input} with
size \(\mathcal A_\iota\). Thus both oscillations share the Fourier-truncated parent,
material derivative, transported phase, base oscillation parameter, and
scale exponent; their harmonic indices and phase shifts may differ.

Every coefficient obtained from
\(u^{[M]}[z_1]\cdot\nabla z_2\) by product-to-sum reduction satisfies
\begin{equation}
  \max_{\mathcal D\in\{D,\partial_t\}}
  \|\mathcal D^pc\|_{C^j}
  \leq C_{p,j}L^{C_{p,j}}\mathcal A_1\mathcal A_2
    \lambda^{1-2\beta+2j\beta}.
  \label{eq:periodic-packet-anisotropic}
\end{equation}
For either \(z=z_\iota\), every coefficient in
\((u^{[M]}-u^{[0]})[z]\cdot\mathcal G\) satisfies
\begin{equation}
  \max_{\mathcal D\in\{D,\partial_t\}}
  \|\mathcal D^pc\|_{C^j}
  \leq C_{p,j}L^{C_{p,j}}\mathcal A_\iota
    \lambda^{-2\beta+2j\beta}.
  \label{eq:periodic-packet-background}
\end{equation}
Both estimates hold for \(0\leq p\leq T\) and \(0\leq j\leq K\).
The constants have the dependence stated in
Corollary~\ref{cor:oscillatory-product-estimates}.
\end{corollary}

\begin{proof}
Put \(A=\nabla X(0,t)=|A|e(\alpha(t))\).  For every permitted harmonic
index \(\ell\), the definition of \(Y_\ell\) gives
\[
  \mathcal W_X
  :=e^\perp(\alpha)\cdot\nabla X
  =e^\perp(\alpha)\cdot(\nabla X-A)
  =\ell^{-1}e^\perp(\alpha)\cdot\nabla Y_\ell.
\]
This function and all its material or ordinary time derivatives vanish at
the origin. The mean-value theorem, together with
\eqref{eq:periodic-packet-phase-block} and
\eqref{eq:periodic-packet-parameter-path}, gives, for
\(\mathcal D\in\{D,\partial_t\}\),
\[
  \|\mathcal D^p\mathcal W_X\|_{C^0(B_r)}
  \leq r\|\nabla\mathcal D^p\mathcal W_X\|_{C^0(B_r)}
  \leq L^Cr
  \leq L^C\lambda^{-2\beta},
\]
and
\[
  \|\mathcal D^p\mathcal W_X\|_{C^j(B_r)}
  \leq L^C\lambda^{2(j-1)_+\beta}.
\]
Thus the leading oscillatory derivative costs only
\(\lambda^{1-2\beta}\), and every higher velocity term gains
\(\lambda^{-2\beta}\). This proves
\eqref{eq:periodic-packet-anisotropic}. Product-to-sum reduction produces
only indices \(|\ell_1\pm\ell_2|\leq2^{M+2}\); no velocity expansion is
applied to these products.

For ordinary-time derivatives of \(b_2\), \(\nabla b_2\), and
\(\mathcal G\), use \(\partial_t=D-v\cdot\nabla\) at one additional
spatial order. The respective transport costs are
\(r\lambda^{2\beta}=L^4\) and \(r\mu\leq L^4\), within the diagonal ranges
of \eqref{eq:periodic-packet-coefficient-input} and
\eqref{eq:periodic-packet-gradient-input}. Finally, every term in
\((u^{[M]}-u^{[0]})[z]\cdot\mathcal G\) has \(n\geq1\), so
\eqref{eq:periodic-packet-stratum-bound} and
\eqref{eq:periodic-packet-gradient-input} give
\eqref{eq:periodic-packet-background}.
\end{proof}

\section{Finite-stage geometry and amplification}
\label{sec:inductive-dynamics}

This section states the finite-stage invariant, selects the next spatial
oscillation parameter, and proves the phase-direction and amplitude mechanism
for one adjacent pair of perturbations.  An overdot denotes differentiation
in time. With \(C_0=2\pi\), define
\[
  \mathcal M(\alpha)
  :=-e^\perp(\alpha)e(\alpha)^{\mathsf T},
  \qquad \alpha\in\mathbb R.
\]
Fix
\[
  \widetilde C=\frac{16\pi^3}{C_0}=8\pi^2.
\]
For every layer, set
\begin{equation}
  \underline t_q
  =1-\widetilde C\lambda_q^{-3\beta_q/4},
  \qquad
  I_q=[\underline t_q,1].
  \label{eq:terminal-window}
\end{equation}

For parameters \(\lambda_i>2\) and \(\beta_i>0\), the associated scales are
defined by
\begin{equation}
  L_i=\log\lambda_i,
  \qquad
  \mu_i=\lambda_i^{\beta_i/2},
  \qquad
  M_i=\lceil\beta_i^{-2}\rceil,
  \qquad
  r_i=\lambda_i^{-2\beta_i}L_i^4.
  \label{eq:scale-dictionary}
\end{equation}
Thus \(\lambda_i\) is the base oscillation parameter, \(\mu_i\) is the
Fourier cutoff scale, \(M_i\) is both the velocity-expansion order and the
number of correction levels, and \(r_i\) is the support radius.

\subsection{The periodic inductive invariant}

The induction preserves four kinds of information: the exact stage equation,
separation of adjacent scales, support and derivative bounds, and terminal
geometry at the origin.

\begin{definition}[Admissible layered family]
\label{def:admissible-layered-family}
An admissible layered family through level \(q\) consists of functions that
are odd in the spatial variable,
\[
  P_i,\theta_i,\rho^{(i)},F^{(i)}
  \in C^\infty([0,1]\times\mathbb T^2),
  \qquad 0\leq i\leq q,
\]
and parameters
\[
  \lambda_i>2,\qquad 0<\beta_i\leq\frac18,
  \qquad k_i,K_i\in\mathbb N,
\]
with the following properties.

\begin{enumerate}[label=\textup{(\roman*)}]
\item \emph{Stage decomposition and equation.}

At level zero,
\[
  P_0=0,\qquad \rho^{(0)}=\theta_0.
\]
For \(1\leq i\leq q\),
\begin{equation}
  P_i=S_{\mu_i}\rho^{(i-1)},
  \qquad
  \rho^{(i)}=P_i+\theta_i,
  \qquad
  \operatorname{supp}\widehat P_i
  \subset\{k:|k|\leq2\mu_i\}.
  \label{eq:admissible-stage-decomposition}
\end{equation}
Every finite stage solves
\begin{equation}
  \partial_t\rho^{(i)}
  +u_{\mathbb T}(\rho^{(i)})\cdot\nabla\rho^{(i)}
  =F^{(i)}
  \quad\text{on }[0,1]\times\mathbb T^2.
  \label{eq:admissible-stage-equation}
\end{equation}
Oddness makes every displayed scalar mean zero.  Here \(F^{(i)}\) denotes
the cumulative finite-stage force.

\item \emph{Scale ranges and separation.}

The integer \(k_i\) specifies the finite range \(j+2m\leq k_i\) for the
perturbation \(\theta_i\), whereas \(K_i\) specifies the propagated range
\(1\leq j+2m\leq K_i\) for the complete stage \(\rho^{(i)}\). They satisfy
\begin{equation}
  k_i,K_i\geq
  d(\beta_i):=\left\lceil\frac4{\beta_i^2}\right\rceil+1.
  \label{eq:admissible-parameter-domain}
\end{equation}
The terminal window is required to lie in the time domain:
\begin{equation}
  \widetilde C\lambda_i^{-3\beta_i/4}\leq1.
  \label{eq:admissible-terminal-window-domain}
\end{equation}
Thus \(0\leq\underline t_i<1\) and \(I_i\subset[0,1]\).
In addition,
\begin{equation}
  K_i\geq d(\beta_{i+1})
  \qquad(0\leq i<q).
  \label{eq:admissible-child-authorization}
\end{equation}
The scales \(L_i,\mu_i,M_i,r_i\) are those in
\eqref{eq:scale-dictionary}.  When both indices exist,
\begin{align}
  \frac{\beta_i}{2}
  &\leq\beta_{i+1}\leq\beta_i,
  \label{eq:admissible-beta-step}\\
  \lambda_{i+1}&\geq4\lambda_i,
  \qquad
  \lambda_{i+1}^{\beta_{i+1}/8}\geq4\lambda_i,
  \label{eq:admissible-frequency-separation}\\
  2\lambda_{i+1}^{-3\beta_{i+1}/4}
  &\leq\lambda_i^{-3\beta_i/4},
  \label{eq:admissible-time-separation}\\
  \mu_{i+1}&\geq\lambda_i^4,
  \qquad
  \mu_{i+1}\geq4\mu_i.
  \label{eq:admissible-bandwidth-separation}
\end{align}

\item \emph{Support, activation, and derivative estimates.}

Each perturbation is supported in the fixed coordinate chart:
\begin{equation}
  \operatorname{supp}\theta_i(t)\subset B_{r_i}\Subset U_0.
  \label{eq:admissible-perturbation-support}
\end{equation}
For \(i\geq1\), there is an activation time \(s_i\in I_{i-1}\) such that
\(\theta_i(t)=0\) for \(t\leq s_i\) (and hence vanishes to infinite order
in time at \(t=s_i\)), and
\begin{equation}
  s_i+L_i^{-8}<\underline t_i.
  \label{eq:admissible-activation-before-terminal-window}
\end{equation}
The finite spatial and mixed-derivative estimates are
\begin{align}
  \|\partial_t^m\theta_i(t)\|_{C^j}
  &\leq\lambda_i^{j+2m}
  &&(j+2m\leq k_i),
  \label{eq:admissible-perturbation-mixed-bound}\\
  \|\partial_t^m\rho^{(i)}(t)\|_{C^j}
  &\leq2\lambda_i^{j+2m}
  &&(1\leq j+2m\leq K_i).
  \label{eq:admissible-stage-mixed-bound}
\end{align}
For every fixed \((j,m)\in\mathbb N_0^2\), the perturbation also satisfies
\begin{equation}
  \|\partial_t^m\theta_i(t)\|_{C^j}
  \leq C_{j,m}(\beta_i)L_i^{N_{j,m}(\beta_i)}
    \lambda_i^{j+m(1-2\beta_i)-1+\beta_i}.
  \label{eq:admissible-perturbation-sharp-bound}
\end{equation}
The families \(C_{j,m}\) and \(N_{j,m}\) are fixed before the stage
induction.  Induct on \(m\): once every spatial order at time orders below
\(m\) has been fixed, choose the order-\((j,m)\) values simultaneously for
all \(j\).  In the order-\(m\) density pass, obtain every
threshold-relevant cutoff-parent input from the
constant-free \((j,m)=(2,0)\) case of
\eqref{eq:admissible-stage-mixed-bound} and the Fourier-cutoff Bernstein
inequality.  The remaining sharp inputs have time order strictly below
\(m\), so their constants have already been fixed.  Same-time parent
estimates are used only in the subsequent defect and force estimates and do
not enter this choice.  For
\(\beta\leq b\leq\min\{2\beta,1/8\}\), choose finite
\(\widehat C_{j,m}(b,\beta)\geq1\) and
\(\widehat N_{j,m}(b,\beta)\in\mathbb N_0\) that dominate the resulting
perturbation estimate.  At \((b,\beta)=(1/8,1/8)\), include the level-zero
estimate.  Finally, set
\begin{equation}
  C_{j,m}(\beta)
  :=\sup_{\substack{\beta\leq\gamma\leq1/8\\
                     \gamma\leq b\leq\min\{2\gamma,1/8\}}}
       \widehat C_{j,m}(b,\gamma),
  \qquad
  N_{j,m}(\beta)
  :=\left\lceil
    \sup_{\substack{\beta\leq\gamma\leq1/8\\
                     \gamma\leq b\leq\min\{2\gamma,1/8\}}}
       \widehat N_{j,m}(b,\gamma)
  \right\rceil.
  \label{eq:sharp-constant-envelope}
\end{equation}
Local uniformity of the estimates makes these suprema finite.  The resulting
families are independent of the stage index and both base oscillation
parameters, and \eqref{eq:sharp-constant-envelope} is the envelope used when
the earlier exponents lie in \([\beta,1/8]\).

\item \emph{Terminal origin geometry and angle evolution.}

On \(I_i\), the inherited cutoff density is small relative to the newest
perturbation:
\begin{equation}
  \|P_i(t)\|_{C^1}
  +\|u_{\mathbb T}(P_i)(t)\|_{C^1}
  \leq\lambda_i^{\beta_i/8}.
  \label{eq:admissible-cutoff-parent-bound}
\end{equation}
The two amplitudes have different roles: \(A_i\) measures the perturbation
density gradient, while \(B_i\) is the coefficient in its velocity-Jacobian
expansion. There are positive functions
\[
  A_i,B_i\in C\!\left(I_i;
    [\lambda_i^{\beta_i}/2,2\lambda_i^{\beta_i}]\right),
\]
an angle \(\alpha_i\in C^1(I_i)\), and remainders
\[
  E_i^\rho\in C(I_i;\mathbb R^2),
  \qquad
  E_i^u\in C(I_i;\mathbb R^{2\times2})
\]
such that
\begin{align}
  \nabla\theta_i(0,t)
  &=A_i(t)e(\alpha_i(t))+E_i^\rho(t),
  &
  |E_i^\rho(t)|
  &\leq\sqrt2\,\lambda_i^{\beta_i/8},
  \label{eq:admissible-gradient-geometry}\\
  D_xu_{\mathbb T}(\theta_i)(0,t)
  &=C_0B_i(t)\cos\alpha_i(t)\mathcal M(\alpha_i(t))
    +E_i^u(t),
  &
  |(E_i^u)_{ab}(t)|
  &\leq\lambda_i^{\beta_i/8}.
  \label{eq:admissible-jacobian-geometry}
\end{align}
The reference angle is transported by \(P_i\):
\begin{equation}
  \dot\alpha_i
  =-\left\langle
    D_xu_{\mathbb T}(P_i)(0,t)^{\mathsf T}e(\alpha_i),
    e^\perp(\alpha_i)\right\rangle,
  \qquad
  \alpha_i(1)=\frac\pi2+\lambda_i^{-\beta_i/8}.
  \label{eq:admissible-angle-law}
\end{equation}
For \(i=0\), the right side is zero.
\end{enumerate}
\end{definition}

The stage decomposition and cutoff nesting give an exact finite-stage formula
before any global parameter recursion is imposed.  For every
\(1\leq Q\leq q\),
\begin{equation}
  \rho^{(Q)}
  =S_{\mu_1}\theta_0
   +\sum_{i=1}^{Q-1}S_{\mu_{i+1}}\theta_i
   +\theta_Q.
  \label{eq:iterated-density-formula}
\end{equation}
Indeed, \(S_{\mu_j}S_{\mu_i}=S_{\mu_i}\) for \(j>i\), by
\eqref{eq:admissible-bandwidth-separation} and
Lemma~\ref{lem:periodic-cutoff-calculus}; induction in \(Q\) proves the
formula.

\subsection{Selection of the next scales}

Set
\[
  k_-:=\frac1{192\pi}.
\]

\begin{proposition}[Deterministic scale selection]
\label{prop:deterministic-frequency-selection}
Let an admissible family be given through level \(q\), and choose
\(\beta_{q+1}\) subject to
\begin{equation}
  \frac{\beta_q}{2}\leq\beta_{q+1}\leq\beta_q,
  \qquad
  K_q\geq d(\beta_{q+1}).
  \label{eq:child-beta-range}
\end{equation}
Define
\begin{equation}
  \boxed{\;
  \lambda_{q+1}
  =\exp\left(
    \frac{\beta_{q+1}}{192\pi}
    \lambda_q^{\beta_q/8}\right).
  \;}
  \label{eq:deterministic-frequency-law}
\end{equation}
There is a finite threshold
\[
  \Lambda_{\rm sep}=\Lambda_{\rm sep}(\beta_q)
\]
such that, if \(\lambda_q\geq\Lambda_{\rm sep}\), the adjacent frequency,
terminal-time, and frequency separations
\eqref{eq:admissible-frequency-separation}--
\eqref{eq:admissible-bandwidth-separation} hold for the pair \((q,q+1)\).
In particular,
\begin{equation}
  \lambda_{q+1}^{\beta_{q+1}/8}\geq4\lambda_q,
  \qquad
  \lambda_{q+1}^{-\beta_{q+1}/8}\leq\lambda_q^{-1},
  \qquad
  \mu_{q+1}\geq\lambda_q^4,
  \qquad
  \mu_{q+1}\geq4\mu_q.
  \label{eq:selected-bandwidth-separation}
\end{equation}
The exact transfer identity is
\begin{equation}
  \lambda_q^A
  =\left(
     \frac{192\pi L_{q+1}}{\beta_{q+1}}
    \right)^{8A/\beta_q}
  \qquad(A\in\mathbb R).
  \label{eq:frequency-transfer}
\end{equation}
\end{proposition}

\begin{proof}
Put \(b=\beta_q\), \(\beta=\beta_{q+1}\), and
\(\Lambda=\lambda_q\).  Then
\begin{equation}
  L_{q+1}=k_-\beta\Lambda^{b/8}.
  \label{eq:next-log-frequency}
\end{equation}
Since \(\beta\in[b/2,b]\), the quantities
\[
  \frac{\beta L_{q+1}}8-\log(4\Lambda),
  \qquad
  \frac{3\beta L_{q+1}}4-\frac{3b}4\log\Lambda,
  \qquad
  \frac{\beta L_{q+1}}2-4\log\Lambda
\]
tend to infinity with \(\Lambda\), uniformly for \(\beta\in[b/2,b]\).
These three limits give respectively
\eqref{eq:admissible-frequency-separation},
\eqref{eq:admissible-time-separation}, and
\(\mu_{q+1}\geq\lambda_q^4\).  They also imply
\(\lambda_{q+1}\geq4\lambda_q\),
and \(\lambda_{q+1}^{-\beta/8}\leq\Lambda^{-1}\), above one threshold
depending only on \(b\).  Moreover,
\(\mu_{q+1}\geq\Lambda^4\geq4\Lambda^{b/2}=4\mu_q\) for
\(\Lambda\geq2\) and \(b\leq1/8\), so this separation requires no further
threshold.  Solving \eqref{eq:next-log-frequency} for \(\Lambda^A\)
proves \eqref{eq:frequency-transfer}.
\end{proof}

\subsection{Cutoff-parent geometry, flow, and transported phase}

\begin{definition}[Adjacent-layer setup]
\label{def:adjacent-layer-setup}
Fix an admissible family through level \(q\), and set the parent aliases
\[
  \Lambda:=\lambda_q,\qquad b:=\beta_q.
\]
Choose the child exponent \(\beta_{q+1}\) subject to
\[
  \beta:=\beta_{q+1}\in[b/2,b],
  \qquad K_q\geq d(\beta).
\]
Define the child frequency \(\lambda_{q+1}\) from
\((\Lambda,b,\beta)\) by \eqref{eq:deterministic-frequency-law}, write
\(\lambda:=\lambda_{q+1}\), and assume
\[
  \Lambda\geq\Lambda_{\rm sep}(b).
\]
Proposition~\ref{prop:deterministic-frequency-selection} supplies all
adjacent-scale separations.  Set
\[
  \begin{aligned}
    L&:=L_{q+1}=\log\lambda,
    &M&:=M_{q+1}=\lceil\beta^{-2}\rceil,
    &\mu&:=\mu_{q+1}=\lambda^{\beta/2},\\
    R&:=\lambda^{-2\beta},
    &r&:=r_{q+1}=RL^4,
  \end{aligned}
\]
where \(B_R\) is the initial child core and \(B_r\) is its buffered support
region. Define the cutoff parent
\begin{equation}
  P_{q+1}=S_\mu\rho^{(q)},
  \qquad
  v_q=u_{\mathbb T}(P_{q+1}).
  \label{eq:actual-next-parent}
\end{equation}
\end{definition}

Constants in the remainder of this section
may depend on the adjacent exponents, displayed finite derivative orders, the
cutoff \(\chi\), the constants in
Definition~\ref{def:admissible-layered-family}, but not on \(\Lambda\) or
\(\lambda\).

Cutting off the complete stage at the child Fourier scale leaves all older cutoff
multipliers unchanged and removes only a rapidly decaying tail of \(\theta_q\);
the next lemma transfers the terminal origin geometry to this cutoff parent.

\begin{lemma}[Cutoff preservation of the newest geometry]
\label{lem:cutoff-parent-geometry}
Assume Definition~\ref{def:adjacent-layer-setup}.  There is a finite threshold
\[
  \Lambda_{\rm geom}=\Lambda_{\rm geom}(b,\beta)
\]
such that, if \(\Lambda\geq\Lambda_{\rm geom}\), the following conclusions
hold for \(t\in I_q\):
\begin{align}
  |\alpha_q(t)-\alpha_q(1)|
  &\leq\Lambda^{-b/2},
  \label{eq:parent-angle-freeze}\\
  -2\Lambda^{-b/8}
  \leq\cos\alpha_q(t)
  &\leq-\tfrac12\Lambda^{-b/8},
  \label{eq:parent-cosine-window}\\
  |\cos\alpha_q(t)|
  &\leq(1+\Lambda^{-3b/8})\Lambda^{-b/8}.
  \label{eq:sharp-parent-cosine}
\end{align}
The nested Fourier cutoffs give, for \(t\in[0,1]\), the exact identity
\begin{equation}
  P_{q+1}=P_q+S_\mu\theta_q.
  \label{eq:next-parent-nested-decomposition}
\end{equation}
For \(\tau_q=(I-S_\mu)\theta_q\),
\begin{equation}
  \|\tau_q(t)\|_{C^2}
  +\|D_xu_{\mathbb T}(\tau_q)(t)\|_{L^\infty}
  \leq C\Lambda^{-4},
  \label{eq:newest-perturbation-cutoff-tail}
\end{equation}
for every \(t\in[0,1]\), where \(C\) is absolute and in particular
independent of \(b,\beta\).  Consequently, for \(t\in I_q\),
\begin{align}
  \nabla S_\mu\theta_q(0,t)
  &=A_qe(\alpha_q)+\widetilde E_q^\rho,
  &
  |\widetilde E_q^\rho|
  &\leq4\Lambda^{b/8},
  \label{eq:cutoff-parent-gradient-geometry}\\
  D_xu_{\mathbb T}(S_\mu\theta_q)(0,t)
  &=C_0B_q\cos\alpha_q\,\mathcal M(\alpha_q)
    +\widetilde E_q^u,
  &
  |(\widetilde E_q^u)_{ab}|
  &\leq4\Lambda^{b/8}.
  \label{eq:cutoff-parent-jacobian-geometry}
\end{align}
\end{lemma}

\begin{proof}
The exact angle law and \eqref{eq:admissible-cutoff-parent-bound} give
\[
  |\dot\alpha_q|\leq C\Lambda^{b/8}.
\]
Integration over \(I_q\) proves
\eqref{eq:parent-angle-freeze}.  Write
\(\alpha_q=\pi/2+\Lambda^{-b/8}+h\), where
\(|h|\leq\Lambda^{-b/2}\).  The bounds
\(2|z|/\pi\leq|\sin z|\leq|z|\) for \(|z|\leq\pi/2\) give
\eqref{eq:parent-cosine-window} and \eqref{eq:sharp-parent-cosine}.

The Fourier support of \(P_q\) lies in \(B_{2\mu_q}\), and
\(\mu\geq2\mu_q\).  Lemma~\ref{lem:periodic-cutoff-calculus} gives
\(S_\mu P_q=P_q\), proving
\eqref{eq:next-parent-nested-decomposition}.  Taking \(j=2,s=3\) in
\eqref{eq:periodic-spectral-tail}, using
\(\|\theta_q\|_{C^8}\leq\Lambda^8\) and \(\mu\geq\Lambda^4\), yields
\[
  \|\tau_q\|_{C^2}\leq C\mu^{-3}\Lambda^8
  \leq C\Lambda^{-4}.
\]
The Jacobian estimate in Lemma~\ref{lem:periodic-parity-schauder} proves
the second part of \eqref{eq:newest-perturbation-cutoff-tail}.  Combining this tail with
\eqref{eq:admissible-gradient-geometry} and
\eqref{eq:admissible-jacobian-geometry} proves the last two conclusions.
\end{proof}

The next three results prepare the hypotheses of the transported-phase
calculus in Section~\ref{sec:velocity-and-jets}. The cutoff-parent flow first
keeps the child core inside the coordinate chart and controls its derivatives.
Pulling a terminal linear phase back by this flow then constructs the child
phase; the final corollary records the resulting finite-order inputs.

Fix once and for all a radius \(\rho_{\rm ch}>0\) such that
\begin{equation}
  \overline{B_{2\rho_{\rm ch}}}\subset U_0.
  \label{eq:fixed-flow-chart-buffer}
\end{equation}
The subscript {\textrm{ch}} is shorthand for child. 
\begin{lemma}[Cutoff-parent flow]
\label{lem:cutoff-parent-flow}
Assume the hypotheses of Lemma~\ref{lem:cutoff-parent-geometry}, and fix an
integer \(J\geq1\).  There is a finite threshold
\[
  \Lambda_{\rm flow}
  =\Lambda_{\rm flow}(b,\beta)
\]
such that the following conclusions hold when
\(\Lambda\geq\Lambda_{\rm flow}\), where the threshold is chosen at least
as large as \(\Lambda_{\rm geom}\).

Let \(\Gamma_q(t;\sigma)\) solve
\[
  \partial_t\Gamma_q(t;\sigma)
  =D_xv_q(0,t)\Gamma_q(t;\sigma),
  \qquad
  \Gamma_q(\sigma;\sigma)=\operatorname{Id}.
\]
All the following conclusions hold for arbitrary
\(\sigma,t\in I_q\), in either time order.  First,
\begin{equation}
  \|\Gamma_q(t;\sigma)\|_{\mathrm{op}}
  +\|\Gamma_q(\sigma;t)\|_{\mathrm{op}}
  \leq C\Lambda^{b/4}.
  \label{eq:cutoff-parent-linear-flow-distortion}
\end{equation}
Let \(\Phi_q(t;\sigma,\cdot)\) be the global flow on \(\mathbb T^2\) of
\(v_q\):
\[
  \partial_t\Phi_q(t;\sigma,x)
  =v_q(\Phi_q(t;\sigma,x),t),
  \qquad
  \Phi_q(\sigma;\sigma,x)=x.
\]
For \(x\in B_{2r}\), let
\(\widetilde\Phi_q(t;\sigma,x)\) denote the lift of this trajectory whose
value at \(t=\sigma\) is the coordinate representative \(x\). Then
\begin{align}
  \sup_{x\in B_{2r}}
    |\widetilde\Phi_q(t;\sigma,x)|
  &\leq C_{b,\beta}L^2r<\rho_{\rm ch},
  \label{eq:actual-parent-buffered-flow}\\
  \Phi_q(t;\sigma,\overline{B_R})&\subset B_{3r/4}\Subset B_r,
  \label{eq:actual-parent-core-flow}\\
  \|D_x^j\Phi_q(t;\sigma,\cdot)\|_{C^0(B_r)}
  &\leq C_jL^{C_j}\mu^{(j-1)_+}
  \qquad(1\leq j\leq J).
  \label{eq:actual-parent-flow-derivatives}
\end{align}
Moreover, for \(x\in B_r\),
\begin{equation}
  \det D_x\Phi_q(t;\sigma,x)=1,
  \qquad
  \left\|\bigl(D_x\Phi_q(t;\sigma,x)\bigr)^{-1}\right\|_{\mathrm{op}}
  \leq CL^C.
  \label{eq:actual-parent-flow-inverse}
\end{equation}
Thus every trajectory used below to define the transported phase stays in
the fixed coordinate chart.
\end{lemma}

\begin{proof}
We estimate the inverse transpose in the frame frozen at \(\alpha_q(1)\),
where the leading system is triangular. We first control its second component,
then its first; incompressibility transfers the resulting bounds to
\(\Gamma_q\) and its inverse. Let \(Z=\Gamma_q^{-\mathsf T}\). By
\eqref{eq:next-parent-nested-decomposition},
\eqref{eq:cutoff-parent-jacobian-geometry}, and
\eqref{eq:admissible-cutoff-parent-bound}, each column
\(z=(z_1,z_2)\) of \(Z\) satisfies
\begin{align*}
  \dot z_1
  &=C_0B_q\cos\alpha_q\,z_2
    +\varepsilon_{11}z_1+\varepsilon_{12}z_2,\\
  \dot z_2&=\varepsilon_{21}z_1+\varepsilon_{22}z_2,
\end{align*}
where
\[
  \sum_{a,b}|\varepsilon_{ab}|
  \leq C\Lambda^{b/2}.
\]
The last bound includes the angular change of the principal matrix:
\(\Lambda^{7b/8}\Lambda^{-b/2}=\Lambda^{3b/8}
\leq\Lambda^{b/2}\).
It suffices to consider forward time; the same estimates apply after reversing
time.  For \(\sigma\leq u\leq t\), set
\[
  m_a(u)=\sup_{\sigma\leq\tau\leq u}|z_a(\tau)|,
  \qquad a\in\{1,2\}.
\]
Integration of the second equation and absorption of
\(C\Lambda^{-b/4}m_2(u)\) give
\[
  m_2(u)
  \leq C
    (1+\Lambda^{-b/4}m_1(u)).
\]
Since \(|C_0B_q\cos\alpha_q|\leq C\Lambda^{7b/8}\), the coefficient of
\(m_1\) after this substitution is bounded by
\(C\Lambda^{7b/8-b/4}=C\Lambda^{5b/8}\), while the source is bounded by
\(C\Lambda^{7b/8}\).  Hence
\[
  1+m_1(u)
  \leq C
   +\int_\sigma^u
    \left(C\Lambda^{5b/8}m_1(\tau)
          +C\Lambda^{7b/8}\right)\,d\tau.
\]
The interval has length at most
\(\widetilde C\Lambda^{-3b/4}\).  Integral Gr\"onwall therefore gives
\(m_1(u)\leq C\Lambda^{b/8}\), and hence
\(\|Z\|\leq C\Lambda^{b/4}\). The same argument with
reversed time controls the inverse. Since \(v_q\) is divergence free,
\(\det\Gamma_q=1\), and in two dimensions these bounds give
\eqref{eq:cutoff-parent-linear-flow-distortion} for both \(\Gamma_q\) and its
inverse.

Cutoff boundedness, the \(j=2,m=0\) case of
\eqref{eq:admissible-stage-mixed-bound}, the Fourier support of \(P_{q+1}\),
Bernstein's inequality, the periodic Schauder estimate, and
\eqref{eq:frequency-transfer} give, for every \(k\geq1\), constants
\(C_k,N_k\), depending only on \(b,\beta,k\) and the fixed cutoff, such that
\begin{equation}
  \|\nabla^kv_q\|_{L^\infty(\mathbb T^2)}
  \leq C_kL^{N_k}\mu^{k-1}.
  \label{eq:flow-parent-spatial-input}
\end{equation}
In particular, fix
\(C_*=\max\{C_1,C_2\}\) and
\(N_*=\max\{N_1,N_2\}\).  These two constants, and not the requested
order \(J\), determine the threshold below.
Put
\[
  \mathsf A(t)=D_xv_q(0,t),
  \qquad
  \mathcal N(x,t)=v_q(x,t)-\mathsf A(t)x.
\]
Oddness gives \(v_q(0,t)=0\). Taylor's theorem and
\eqref{eq:flow-parent-spatial-input} therefore imply, for
\(x\in B_{\rho_{\rm ch}}\),
\begin{equation}
  |\mathcal N(x,t)|
  \leq C_*L^{N_*}\mu|x|^2,
  \qquad
  \|D_x\mathcal N(x,t)\|
  \leq C_*L^{N_*}\mu|x|.
  \label{eq:flow-nonlinear-remainder}
\end{equation}

Let
\[
  H=\widetilde C\Lambda^{-3b/4},
  \qquad
  G=\max\{1,C_{b,\beta}\Lambda^{b/4}\},
\]
where \(G\) dominates both linear-flow norms in
\eqref{eq:cutoff-parent-linear-flow-distortion}. The frequency law gives
\begin{equation}
  H=\widetilde C(k_-\beta)^6L^{-6},
  \qquad
  G\leq C_{b,\beta}L^2,
  \qquad
  \mu r=\lambda^{-3\beta/2}L^4.
  \label{eq:flow-small-parameters}
\end{equation}
After increasing \(\Lambda_{\rm flow}\), we may require simultaneously
\begin{equation}
  4Gr<\rho_{\rm ch},\qquad
  2GR<r,\qquad
  8C_*G^2HL^{N_*}\mu r\leq\frac12.
  \label{eq:flow-bootstrap-thresholds}
\end{equation}
These conditions are compatible because
\(\Lambda^{b/4}=(L/(k_-\beta))^2\),
\(GR/r=G/L^4\to0\), and
\(L^N\mu r\to0\) for every fixed \(N\).

Fix \(\sigma,t\in I_q\), put
\(h=|t-\sigma|\leq H\), and write
\[
  y(u)=\widetilde\Phi_q(u;\sigma,x).
\]
On the interval between \(\sigma\) and \(t\), variation of constants gives
\begin{equation}
  y(u)=\Gamma_q(u;\sigma)x+
  \int_\sigma^u
    \Gamma_q(u;\tau)\mathcal N(y(\tau),\tau)\,d\tau.
  \label{eq:nonlinear-flow-variation}
\end{equation}
For \(x\in B_{2r}\), bootstrap \(|y|\leq4Gr\). Equations
\eqref{eq:flow-nonlinear-remainder} and
\eqref{eq:flow-bootstrap-thresholds} give
\[
  |y(u)|
  \leq2Gr+16C_*G^3hL^{N_*}\mu r^2
  <4Gr.
\]
The inequality is strict, so the continuation argument proves
\eqref{eq:actual-parent-buffered-flow}, in both time directions.

For \(x\in\overline{B_R}\), bootstrap instead \(|y|\leq r\). Then
\[
  |y(u)|
  \leq GR+C_*GhL^{N_*}\mu r^2
  <\frac r2+\frac r{16}<\frac{3r}{4}
\]
by \eqref{eq:flow-bootstrap-thresholds}. This proves
\eqref{eq:actual-parent-core-flow}.

For \(x\in B_r\), set
\[
  J_1(u)=D_x\widetilde\Phi_q(u;\sigma,x),
  \qquad
  E(u)=D_xv_q(y(u),u)-\mathsf A(u).
\]
The buffered-flow estimate and
\eqref{eq:flow-nonlinear-remainder} give
\[
  \|E(u)\|\leq4C_*GL^{N_*}\mu r.
\]
Variation of constants around \(\Gamma_q\) yields
\[
  J_1(u)=\Gamma_q(u;\sigma)
  +\int_\sigma^u\Gamma_q(u;\tau)E(\tau)J_1(\tau)\,d\tau.
\]
Hence
\[
  \sup_u\|J_1(u)\|
  \leq G+4C_*G^2HL^{N_*}\mu r\sup_u\|J_1(u)\|
  \leq G+\frac14\sup_u\|J_1(u)\|,
\]
and therefore \(\sup_u\|J_1(u)\|\leq2G\leq CL^2\).

For \(m\geq2\), let \(J_m=D_x^m\widetilde\Phi_q\). Differentiating the
flow equation gives
\[
  \partial_uJ_m
  =\mathsf A J_m+EJ_m+\mathcal F_m,
  \qquad
  J_m(\sigma)=0,
\]
where \(\mathcal F_m\) is the sum over partitions
\(m_1+\cdots+m_a=m\) with \(a\geq2\). Assuming the result below orders
\(m\), each such term satisfies
\[
  \|\nabla^av_q\|
  \prod_{\nu=1}^a\|J_{m_\nu}\|
  \leq C_mL^{C_m}
    \mu^{a-1}\prod_{\nu=1}^a\mu^{m_\nu-1}
  =C_mL^{C_m}\mu^{m-1}.
\]
Variation of constants, followed by the same absorption used for \(J_1\),
proves
\[
  \|J_m\|_{C^0(B_r)}
  \leq C_mL^{C_m}\mu^{m-1}
  \qquad(2\leq m\leq J).
\]
This proves \eqref{eq:actual-parent-flow-derivatives}.

Finally, incompressibility gives
\[
  \partial_u\det J_1(u)
  =(\operatorname{div}v_q)(y(u),u)\det J_1(u)=0.
\]
Thus \(\det J_1=1\). In two dimensions a determinant-one matrix and its
inverse have the same Euclidean operator norm. The \(m=1\) estimate proves
\eqref{eq:actual-parent-flow-inverse}.
\end{proof}

\begin{lemma}[Transported child phase]
\label{lem:transported-child-phase}
Assume Definition~\ref{def:adjacent-layer-setup}, fix an integer \(J\geq3\),
and take the
flows from Lemma~\ref{lem:cutoff-parent-flow} with
\(\Lambda\geq\Lambda_{\rm flow}\). Put
\[
  p_{\rm term}=e\!\left(\frac\pi2+\lambda^{-\beta/8}\right).
\]
Define globally on \(\mathbb T^2\times I_q\)
\begin{equation}
  X_{q+1}(x,t)
  =\psi_{p_{\rm term}}\bigl(\Phi_q(1;t,x)\bigr).
  \label{eq:actual-parent-phase}
\end{equation}
Then \(X_{q+1}\) is odd,
\[
  (\partial_t+v_q\cdot\nabla)X_{q+1}=0,
  \qquad
  \nabla X_{q+1}(0,t)=\Gamma_q(1;t)^{\mathsf T}p_{\rm term},
\]
and the following estimates hold on \(B_r\times I_q\):
\begin{align}
  L^{-C}\leq|\nabla X_{q+1}|
  &\leq L^C,
  \notag\\
  \|\nabla^jX_{q+1}\|_{C^0(B_r)}
  &\leq C_jL^{C_j}\mu^{(j-1)_+},
  \qquad 1\leq j\leq J,
  \label{eq:actual-parent-phase-derivatives}\\
  \|X_{q+1}-\nabla X_{q+1}(0,t)\cdot x\|_{C^0(B_r)}
  &\leq L^C\lambda^{-4\beta},
  \notag\\
  \|\nabla X_{q+1}-\nabla X_{q+1}(0,t)\|_{C^0(B_r)}
  &\leq L^C\lambda^{-2\beta}.
  \label{eq:actual-parent-phase-affine}
\end{align}
The continuous lift \(\alpha_{q+1}\) determined by
\begin{equation}
  e(\alpha_{q+1}(t))
  =\frac{\nabla X_{q+1}(0,t)}
    {|\nabla X_{q+1}(0,t)|},
  \qquad
  \alpha_{q+1}(1)=\frac\pi2+\lambda^{-\beta/8},
  \label{eq:child-angle-phase-gradient}
\end{equation}
satisfies
\begin{equation}
  \dot\alpha_{q+1}
  =-\left\langle
    D_xu_{\mathbb T}(P_{q+1})(0,t)^{\mathsf T}e(\alpha_{q+1}),
    e^\perp(\alpha_{q+1})\right\rangle.
  \label{eq:new-angle-transport}
\end{equation}
\end{lemma}

\begin{proof}
Write \(X:=X_{q+1}\) in this fixed-step proof.
The flow property gives
\[
  (\partial_t+v_q\cdot\nabla)X=0
  \quad\text{on }\mathbb T^2\times I_q.
\]
Both \(v_q\) and \(\psi_{p_{\rm term}}\) are odd, so uniqueness for the
flow gives that \(X\) is odd. By
\eqref{eq:actual-parent-buffered-flow}, if \(x\in B_r\), then
\(\widetilde\Phi_q(1;t,x)\in B_{\rho_{\rm ch}}\subset U_0\).
Consequently, on \(B_r\),
\[
  X(x,t)
  =p_{\rm term}\cdot\widetilde\Phi_q(1;t,x),
  \qquad
  \nabla X(0,t)
  =\Gamma_q(1;t)^{\mathsf T}p_{\rm term}.
\]
Equations~\eqref{eq:actual-parent-flow-derivatives} and
\eqref{eq:actual-parent-flow-inverse} give
\[
  L^{-C}\leq|\nabla X(x,t)|\leq L^C
  \qquad(x\in B_r),
\]
together with all the asserted higher spatial derivative bounds.

Since \(X\) is odd, its quadratic jet vanishes at the origin. Taylor's
theorem and \(\|D_x^3X\|\leq L^C\mu^2\) give
\[
  \|X-\nabla X(0,t)\cdot x\|_{C^0(B_r)}
  \leq L^C\mu^2r^3,
  \qquad
  \|\nabla X-\nabla X(0,t)\|_{C^0(B_r)}
  \leq L^C\mu^2r^2.
\]
Substitution of \(\mu=\lambda^{\beta/2}\) and
\(r=\lambda^{-2\beta}L^4\) proves
\eqref{eq:actual-parent-phase-affine}. Differentiating the transport equation
at the origin gives
\[
  \partial_t\nabla X(0,t)
  =-D_xv_q(0,t)^{\mathsf T}\nabla X(0,t).
\]
Taking its angular component proves \eqref{eq:new-angle-transport}.
\end{proof}

The following corollary packages the hypotheses of
Lemma~\ref{lem:transported-phase-estimates} in the derivative ranges
authorized by \(K_q\).

\begin{corollary}[Finite-order transported-phase inputs]
\label{cor:finite-packet-inputs}
Assume Definition~\ref{def:adjacent-layer-setup}.  Fix integers \(T,K\geq0\),
let \(N_{\mathrm{prof}},N_{\mathrm{phase}}\) be defined by
\eqref{eq:periodic-packet-orders}, and suppose
\begin{equation}
  N_{\mathrm{phase}}+2T+2\leq K_q.
  \label{eq:finite-packet-authorization}
\end{equation}
Set \(I=I_q\), let \(X_{q+1}\) be the child phase supplied by
Lemma~\ref{lem:transported-child-phase}, and set
\[
  D_q=\partial_t+v_q\cdot\nabla.
\]
There is a finite threshold
\[
  \Lambda_{\rm pkt}=\Lambda_{\rm pkt}(b,\beta)
\]
such that, when \(\Lambda\geq\Lambda_{\rm pkt}\), the fields
\(P_{q+1},v_q,D_q,X_{q+1}\) satisfy, on \(B_r\times I_q\),
\eqref{eq:periodic-packet-velocity-input}--
\eqref{eq:periodic-packet-phase-input} with
\(P=P_{q+1}\), \(v=v_q\), \(D=D_q\), and
\(X=X_{q+1}\), \(\mathcal G=\nabla P_{q+1}\). More precisely, the velocity
input holds for
\[
  0\leq p\leq(T-1)_+,
  \qquad 0\leq j\leq N_{\mathrm{prof}},
\]
the density-gradient input holds for
\[
  0\leq p\leq T,
  \qquad p+j\leq N_{\mathrm{prof}},
\]
and the spatial phase estimates hold through order \(N_{\mathrm{phase}}\).
Moreover,
\[
  \lambda\ell|\nabla X_{q+1}(0,t)|\geq2
  \qquad(\ell\in\mathbb Z,\ 1\leq\ell\leq2^{M+1}).
\]
\end{corollary}

\begin{proof}[Proof of Corollary~\ref{cor:finite-packet-inputs}]
Apply Lemma~\ref{lem:cutoff-parent-flow} and
Lemma~\ref{lem:transported-child-phase} with \(J=N_{\mathrm{phase}}\).
Their common threshold is independent of this requested order; only the
constants in the resulting derivative estimates depend on it.
Fourier support limitation, the mixed stage estimate
\eqref{eq:admissible-stage-mixed-bound}, and
\eqref{eq:next-log-frequency} give, in the displayed finite ranges,
\[
  \|D_q^pv_q\|_{C^{j+1}}
  +\|D_q^p\nabla P_{q+1}\|_{C^j}
  \leq C_{p,j}L^{C_{p,j}}\mu^j.
\]
Indeed, expansion of \(D_q^p\) uses ordinary stage derivatives of weighted
order at most \(j+2p+2\), all of which are supplied by
\eqref{eq:finite-packet-authorization}.  Oddness improves the
\(C^0(B_r)\) velocity bound to \(C_pL^{C_p}r\).  The phase conclusions are
those of Lemma~\ref{lem:transported-child-phase}.  Its lower gradient bound
also gives the stated lower frequency bound after increasing a threshold
depending only on \(b,\beta\), since that bound uses only the first flow
derivative.
\end{proof}

The identity \(S_\mu P_q=P_q\) has a second consequence.  The angle
\(\alpha_q\) in \eqref{eq:admissible-angle-law} continues to be driven by
exactly \(P_q\) when the cutoff parent is formed.  Thus there is no cutoff
error in the reference-angle equation; only the spectral tail created when
\(\theta_q\) is cut off, estimated in
\eqref{eq:newest-perturbation-cutoff-tail}, enters the child equation.

\subsection{Perpendicularity time}

The child is activated when its phase-gradient direction is perpendicular to
the newest parent direction. The parent shear makes the relative angle satisfy
a Riccati comparison, which yields one such time and locates it inside \(I_q\).

Define the positive strain coefficient of the newest perturbation by
\begin{equation}
  \kappa_q(t)=-C_0B_q(t)\cos\alpha_q(t).
  \label{eq:parent-strain-coefficient}
\end{equation}
When \(\Lambda\geq
\Lambda_{\rm geom}(b,\beta)\),
Lemma~\ref{lem:cutoff-parent-geometry} gives on \(I_q\)
\begin{equation}
  \frac{C_0}{4}\Lambda^{7b/8}
  \leq\kappa_q(t)
  \leq2C_0(1+\Lambda^{-3b/8})\Lambda^{7b/8}.
  \label{eq:parent-strain-bounds}
\end{equation}

\begin{lemma}[Perpendicularity crossing]
\label{lem:perpendicularity-crossing}
Assume Definition~\ref{def:adjacent-layer-setup}, and take
the child phase \(X_{q+1}\) and angle \(\alpha_{q+1}\) from
Lemma~\ref{lem:transported-child-phase} with \(J=3\). Thus, on \(I_q\),
\begin{gather*}
  e(\alpha_{q+1})
  =\frac{\nabla X_{q+1}(0,t)}
    {|\nabla X_{q+1}(0,t)|},
  \qquad
  \alpha_{q+1}(1)=\frac\pi2+\lambda^{-\beta/8},
  \\
  \dot\alpha_{q+1}
  =-\left\langle
    D_xu_{\mathbb T}(P_{q+1})(0,t)^{\mathsf T}e(\alpha_{q+1}),
    e^\perp(\alpha_{q+1})\right\rangle.
\end{gather*}
Choose a finite threshold
\[
  \Lambda_{\perp}=\Lambda_{\perp}(b,\beta)
\]
with
\[
  \Lambda_\perp\geq
  \max\left\{
    \Lambda_{\rm geom}(b,\beta),
    \Lambda_{\rm flow}(b,\beta)
  \right\}.
\]
If \(\Lambda\geq\Lambda_{\perp}\), there is a unique
\(s_{q+1}\in I_q\) such that
\begin{equation}
  \alpha_{q+1}(s_{q+1})-\alpha_q(s_{q+1})
  =-\frac\pi2.
  \label{eq:perpendicularity-crossing-definition}
\end{equation}
Writing \(\delta=\alpha_{q+1}-\alpha_q\), one has on
\([s_{q+1},1]\)
\begin{align}
  -\frac\pi2\leq\delta(t)
  &\leq-\frac34\Lambda^{-b/8},
  \label{eq:angle-gap-range}\\
  \frac{C_0}{8\pi^3}\Lambda^{7b/8}\delta(t)^2
  \leq\dot\delta(t)
  &\leq3C_0\Lambda^{7b/8}\delta(t)^2.
  \label{eq:riccati-comparison}
\end{align}
Moreover,
\begin{equation}
  \frac1{6C_0}\Lambda^{-3b/4}
  \leq1-s_{q+1}
  \leq\frac{32\pi^3}{3C_0}\Lambda^{-3b/4}
  <\widetilde C\Lambda^{-3b/4},
  \label{eq:perpendicularity-window}
\end{equation}
and the crossing is transverse:
\begin{equation}
  \dot\delta(s_{q+1})
  \geq\frac{C_0}{32\pi}\Lambda^{7b/8}.
  \label{eq:perpendicularity-transversality}
\end{equation}
\end{lemma}

\begin{proof}
Subtract \eqref{eq:admissible-angle-law} from
\eqref{eq:new-angle-transport}. For a matrix \(B\), set
\[
  \mathfrak a_B(\vartheta)
  :=-\langle B^{\mathsf T}e(\vartheta),e^\perp(\vartheta)\rangle.
\]
Then
\[
  \mathfrak a_B'(\vartheta)
  =\langle B^{\mathsf T}e(\vartheta),e(\vartheta)\rangle
   -\langle B^{\mathsf T}e^\perp(\vartheta),e^\perp(\vartheta)\rangle,
\]
so \(|\mathfrak a_B'|\leq2\|B\|_{\mathrm{op}}\). By
\eqref{eq:next-parent-nested-decomposition}, the contribution of \(P_q\) is
\(\mathfrak a_{D_xu_{\mathbb T}(P_q)(0,t)}(\alpha_{q+1})
-\mathfrak a_{D_xu_{\mathbb T}(P_q)(0,t)}(\alpha_q)\). The mean-value theorem
and \eqref{eq:admissible-cutoff-parent-bound} bound it by
\(C\Lambda^{b/8}|\delta|\). The principal part of
\eqref{eq:cutoff-parent-jacobian-geometry} gives exactly
\(\kappa_q\sin^2\delta\).  Hence
\begin{equation}
  \dot\delta=\kappa_q\sin^2\delta+\mathcal E_q^\alpha,
  \qquad
  |\mathcal E_q^\alpha|
  \leq16\Lambda^{b/8}|\delta|+8\Lambda^{b/8}.
  \label{eq:defective-parent-angle-ode}
\end{equation}
There is no term involving \((I-S_\mu)P_q\), because that function is
identically zero.

At \(t=1\),
\[
  -\Lambda^{-b/8}
  \leq\delta(1)
  \leq-\frac34\Lambda^{-b/8}.
\]
On the strip
\[
  -\frac\pi2\leq\delta
  \leq-\frac14\Lambda^{-b/8},
\]
the remainder in \eqref{eq:defective-parent-angle-ode} is absorbed by the
lower bound in \eqref{eq:parent-strain-bounds}.  Using
\(2|z|/\pi\leq|\sin z|\leq|z|\) for \(|z|\leq\pi/2\), and the sharp upper
bound in \eqref{eq:parent-strain-bounds}, gives
\eqref{eq:riccati-comparison}.  In particular, the gap is strictly
increasing in forward time throughout the strip.

At every \(t\in I_q\) for which \(\delta(t)=-\pi/2\),
\eqref{eq:defective-parent-angle-ode} and
\eqref{eq:parent-strain-bounds} give
\[
  \dot\delta(t)
  \geq\frac{C_0}{4}\Lambda^{7b/8}
      -(8\pi+8)\Lambda^{b/8}>0.
\]
Thus every crossing of \(-\pi/2\) is upward, so there can be at most one
such crossing in \(I_q\).

Integrating the inequalities for \(-1/\delta\) from the terminal gap
backward to \(-\pi/2\) gives
\eqref{eq:perpendicularity-window}.  The upper endpoint lies strictly inside
\(I_q\) because \(32/3<16\), which gives existence.  The preceding estimate
at an arbitrary crossing gives uniqueness on all of \(I_q\).
Evaluating the lower inequality in \eqref{eq:riccati-comparison} at
\(\delta=-\pi/2\) proves
\eqref{eq:perpendicularity-transversality}.
\end{proof}

\subsection{Amplitude gain}

Set
\[
  k_{\mathrm g}:=\frac1{96\pi}.
\]

The cutoff-parent gradient decomposes on \(I_q\) as
\begin{equation}
  \nabla P_{q+1}(0,t)
  =A_q(t)e(\alpha_q(t))+E_q^{\mathrm{par}}(t),
  \qquad
  |E_q^{\mathrm{par}}(t)|
  \leq C\Lambda^{b/8}.
  \label{eq:actual-parent-gradient-geometry}
\end{equation}
Indeed, combine \eqref{eq:next-parent-nested-decomposition},
\eqref{eq:admissible-cutoff-parent-bound}, and
\eqref{eq:cutoff-parent-gradient-geometry}.

The principal perturbation constructed in the next section, with amplitude
\(a_{q+1}\) from Lemma~\ref{lem:amplitude-gain} and activation cutoff \(w\) from
\eqref{eq:activation-and-seed}, has origin coefficient
\(w(t)a_{q+1}(t)\lambda^{\beta-1}\). Since \(X_{q+1}(0,t)=0\), its
origin gradient is
\[
  w(t)a_{q+1}(t)\lambda^\beta\nabla X_{q+1}(0,t).
\]
For \(t\geq s_{q+1}+L_{q+1}^{-8}\), the activation cutoff satisfies
\(w(t)=1\). Its leading velocity response is
multiplication by
\[
  -C_0\cos\alpha_{q+1}e^\perp(\alpha_{q+1}).
\]
The following equation cancels the resulting leading interaction with
\(\nabla P_{q+1}\).

\begin{lemma}[Amplitude gain]
\label{lem:amplitude-gain}
Under the hypotheses and notation of
Lemma~\ref{lem:perpendicularity-crossing}, let \(s_{q+1}\) be the unique
crossing furnished there. Choose a finite threshold
\[
  \Lambda_{\rm gain}=\Lambda_{\rm gain}(b,\beta)
  \geq\Lambda_\perp(b,\beta).
\]
If \(\Lambda\geq\Lambda_{\rm gain}\), let \(a_{q+1}\) be the unique positive
solution on \(I_q\) of
\begin{equation}
  \frac{\dot a_{q+1}}{a_{q+1}}
  =C_0\cos\alpha_{q+1}\,
    e^\perp(\alpha_{q+1})\cdot\nabla P_{q+1}(0,t),
  \qquad
  a_{q+1}(1)=1.
  \label{eq:new-amplitude-equation}
\end{equation}
Then
\begin{equation}
  \frac{a_{q+1}(1)}{a_{q+1}(s_{q+1})}
  \geq\exp\left(k_{\mathrm g}\Lambda^{b/8}\right).
  \label{eq:amplitude-gain-bounds}
\end{equation}
On every subinterval,
\begin{equation}
  \frac{a_{q+1}(t_2)}{a_{q+1}(t_1)}
  \geq e^{-1}
  \qquad
  (s_{q+1}\leq t_1\leq t_2\leq1).
  \label{eq:amplitude-no-collapse}
\end{equation}
\end{lemma}

\begin{proof}
Write \(a:=a_{q+1}\).
We isolate a short interval of positive gain and then bound every later
negative contribution.
By \eqref{eq:actual-parent-gradient-geometry}, the amplitude equation is
equivalently
\begin{align}
  \frac{\dot a}{a}
  ={}&C_0A_q\cos\alpha_{q+1}
      \sin(\alpha_q-\alpha_{q+1})+g_q,
  \notag\\
  g_q
  ={}&C_0\cos\alpha_{q+1}\,
      e^\perp(\alpha_{q+1})\cdot E_q^{\mathrm{par}},
  \qquad
  |g_q|\leq C\Lambda^{b/8}.
  \label{eq:amplitude-error}
\end{align}

For the lower bound, let \(t_-\) be the unique time at which
\(\delta(t_-)=-2\pi/5\).  The upper Riccati inequality gives
\begin{equation}
  t_--s_{q+1}
  \geq\frac1{6\pi C_0}\Lambda^{-7b/8}.
  \label{eq:positive-gain-subinterval}
\end{equation}
Throughout this interval, by \eqref{eq:parent-angle-freeze},
\[
  \alpha_q(t)
  =\frac\pi2+\Lambda^{-b/8}+O(\Lambda^{-b/2}),
  \qquad
  -\frac\pi2\leq\delta(t)\leq-\frac{2\pi}{5}.
\]
Consequently, after increasing the threshold,
\(0<\alpha_{q+1}(t)<\pi/6\) and
\(2\pi/5\leq\alpha_q-\alpha_{q+1}\leq\pi/2\). Thus
\(\cos\alpha_{q+1}\geq1/2\) and
\(\sin(\alpha_q-\alpha_{q+1})\geq1/2\).  Since
\(A_q\geq\Lambda^b/2\), the principal contribution to
\(\log(a(t_-)/a(s_{q+1}))\) is at least
\[
  \frac1{48\pi}\Lambda^{b/8}.
\]

The angle \(\alpha_{q+1}\) is strictly increasing on
\([s_{q+1},1]\): the lower Riccati bound dominates
\(|\dot\alpha_q|\) even at the terminal gap.  Let \(t_0\) be its unique
crossing of \(\pi/2\).  The principal amplitude term is nonnegative on
\([t_-,t_0]\). By \eqref{eq:selected-bandwidth-separation}, on \([t_0,1]\),
\[
  0\leq\alpha_{q+1}-\frac\pi2
  \leq\lambda^{-\beta/8}\leq\Lambda^{-1},
\]
and \eqref{eq:perpendicularity-window} gives
\begin{align*}
  \left|\int_{t_0}^1
    C_0A_q\cos\alpha_{q+1}
      \sin(\alpha_q-\alpha_{q+1})\,dt\right|
  &\leq C\Lambda^{b/4-1},\\
  \int_{s_{q+1}}^1|g_q|\,dt
  &\leq C\Lambda^{-5b/8}.
\end{align*}
Both powers tend to zero because \(0<b\leq1/8\). For a sufficiently large
threshold,
the total logarithmic gain is therefore at least
\((96\pi)^{-1}\Lambda^{b/8}\).

The same displayed bounds on an arbitrary subinterval show, after increasing
the threshold, that its negative principal contribution and its
\(g_q\)-contribution have total absolute value at most \(1\). This proves
\eqref{eq:amplitude-no-collapse}.
\end{proof}

\begin{corollary}[Deterministic seed smallness]
\label{cor:deterministic-seed-smallness}
Assume the hypotheses and notation of Lemma~\ref{lem:amplitude-gain}. There is
a finite threshold
\[
  \Lambda_{\rm seed}=\Lambda_{\rm seed}(b,\beta)
  \geq\Lambda_{\rm gain}(b,\beta)
\]
such that, if \(\Lambda\geq\Lambda_{\rm seed}\), then
\begin{equation}
  a_{q+1}(s_{q+1})
  \leq\lambda^{-2/\beta}
  \leq\lambda^{-1/\beta}.
  \label{eq:deterministic-seed-bound}
\end{equation}
\end{corollary}

\begin{proof}
By \eqref{eq:next-log-frequency}, \(k_{\mathrm g}=2k_-\), and
\(a_{q+1}(1)=1\),
the lower bound in \eqref{eq:amplitude-gain-bounds} gives
\[
  a_{q+1}(s_{q+1})
  \leq e^{-2k_-\Lambda^{b/8}}
  =\lambda^{-2/\beta}
  \leq\lambda^{-1/\beta}.
\]
\end{proof}

\section{Construction of one new layer}
\label{sec:one-layer}

We first form the cutoff parent and its leading transported oscillation. The
leading linearized operator can be inverted along the parent flow; repeating
this operation finitely many times cancels the successive nonzero
phase-harmonic errors. The final uncancelled term, together with the exact
velocity remainders and the cutoff commutator, is placed in the force.

Fix the adjacent-layer setup of
Definition~\ref{def:adjacent-layer-setup}. We retain the one-step aliases
\[
  \Lambda=\lambda_q,\qquad b=\beta_q,\qquad
  \beta=\beta_{q+1},\qquad \lambda=\lambda_{q+1},
  \qquad
  L=L_{q+1},\quad \mu=\mu_{q+1},\quad r=r_{q+1},
\]
together with its local core radius \(R\), and set
\[
  M:=M_{q+1}=\lceil\beta^{-2}\rceil.
\]
Proposition~\ref{prop:deterministic-frequency-selection} supplies all scale
separations used below.  The adjacent-layer setup includes
\(K_q\geq d(\beta)\) and defines \(\lambda\) by
\eqref{eq:deterministic-frequency-law}.

All constants in this section may depend on the two adjacent exponents, the
displayed derivative orders, the fixed perturbation and activation profiles,
and the cutoff in \eqref{eq:periodic-cutoff}, but not on either base
oscillation parameter. They may also depend on the finite subset of the
a priori family \(C_{j,m}(\eta),N_{j,m}(\eta)\) fixed in
Definition~\ref{def:admissible-layered-family} that is explicitly used.

Mixed time derivatives close in two passes. Within one child, at each order
\(p\), all its perturbation coefficients are estimated first,
using structured-defect estimates only below order \(p\). Only then are the
order-\(p\) residual and force-increment terms formed. The closure order is
stated and
explained after the correction hierarchy. Appendix~\ref{sec:smooth-force}
proves Lemmas~\ref{lem:mixed-parent-inputs},
\ref{lem:normalized-weighted-propagator}, and
\ref{lem:weighted-duhamel-estimate}, and supplies the technical endpoint and
derivative-range details for Lemma~\ref{lem:density-first-mixed-closure}.

\subsection{The cutoff parent and its weighted propagator}

The parent for the new perturbation is the cutoff of the complete preceding
stage.  Write
\begin{equation}
  P:=P_{q+1}=S_\mu\rho^{(q)}.
  \label{eq:actual-parent}
\end{equation}
Lemma~\ref{lem:cutoff-parent-geometry} gives
\begin{equation}
  S_\mu P_q=P_q,
  \qquad
  P=P_q+S_\mu\theta_q,
  \label{eq:nested-parent-identity}
\end{equation}
and
\begin{equation}
  \|(I-S_\mu)\theta_q\|_{C^2}
  +\|D_xu_{\mathbb T}((I-S_\mu)\theta_q)\|_{L^\infty}
  \leq C\Lambda^{-4}.
  \label{eq:parent-newest-spectral-tail}
\end{equation}
Thus the previously filtered terms are unchanged and only the newest
perturbation is truncated.

Put
\[
  v:=u_{\mathbb T}(P),
  \qquad
  D:=\partial_t+v\cdot\nabla.
\]
\begin{lemma}[All-orders cutoff-parent inputs]
\label{lem:mixed-parent-inputs}
Fix \(T,S\in\mathbb N_0\), and assume the adjacent-layer setup above. Use
the constant-free \((j,m)=(2,0)\) case of
\eqref{eq:admissible-stage-mixed-bound} at time order zero. For
\(1\leq p\leq T\), use the sharp perturbation estimate
\eqref{eq:admissible-perturbation-sharp-bound} only at spatial order two,
with the a priori constants fixed in
Definition~\ref{def:admissible-layered-family}. Assume also
\[
  \lambda
  =\exp\left(\frac{\beta}{192\pi}\Lambda^{b/8}\right),
\]
and hence the exact transfer identity \eqref{eq:frequency-transfer}.  Then
there are constants depending on \(b,\beta,T,S\) and the fixed profile and
cutoff data, but not on the base oscillation parameters, such that
\begin{equation}
  \max_{0\leq p\leq T}
  \left(
    \|\partial_t^p\rho^{(q)}\|_{C^2}
    +\|\partial_t^pu_{\mathbb T}(\rho^{(q)})\|_{C^1}
  \right)
  \leq C_TL^{C_T}.
  \label{eq:mixed-parent-low-norms}
\end{equation}
Consequently,
\begin{align}
  \|\partial_t^pv\|_{C^{j+1}}
  &\leq C_{p,j}L^{C_{p,j}}\mu^j,
  \label{eq:mixed-cutoff-velocity}\\
  \|\partial_t^p\nabla P\|_{C^j}
  &\leq C_{p,j}L^{C_{p,j}}\mu^j,
  \label{eq:mixed-cutoff-gradient}
\end{align}
for \(0\leq p\leq T\) and \(0\leq j\leq S\).  On \(B_r\),
\begin{equation}
  \|\partial_t^pv\|_{C^0(B_r)}
  \leq C_pL^{C_p}r.
  \label{eq:mixed-cutoff-velocity-core}
\end{equation}
For \(0\leq p\leq T\) and \(0\leq j\leq S\), the local material estimates are
\begin{align}
  \|D^pv\|_{C^{j+1}(B_r)}
  +\|D^p\nabla P\|_{C^j(B_r)}
  &\leq C_{p,j}L^{C_{p,j}}\mu^j,
  \notag\\
  \|D^pv\|_{C^0(B_r)}
  &\leq C_pL^{C_p}r.
  \label{eq:mixed-cutoff-local-material-estimates}
\end{align}
At material order \(p\), only ordinary parent time orders at most \(p\)
occur. The high spatial orders in
\eqref{eq:mixed-cutoff-velocity}--\eqref{eq:mixed-cutoff-local-material-estimates}
follow from the Fourier cutoff and Bernstein's inequality; they consume no
higher-order member of the a priori family. In the local material estimates,
the extra derivative from a transport factor is absorbed by
\(r\mu=L^4\lambda^{-3\beta/2}\leq1\).
\end{lemma}

Hereafter, the cutoff-parent inputs mean the local estimates
\eqref{eq:mixed-cutoff-local-material-estimates} on \(B_r\), supplied by
Lemma~\ref{lem:mixed-parent-inputs} under admissibility through level \(q\).
At coefficient time order zero only the constant-free stage bound is used.
For \(p\geq1\), the coefficient pass invokes this lemma only through parent
time order \(p-1\); the sharp order-\(p\) parent estimate is used afterwards
in the defect and force pass.

Let \(\alpha=\alpha_{q+1}\) be the phase angle transported by
\(u_{\mathbb T}(P)\), with terminal value
\[
  \alpha(1)=\frac\pi2+\lambda^{-\beta/8}.
\]
The admissible geometry and scale separation permit the application of
Lemmas~\ref{lem:cutoff-parent-flow} and
\ref{lem:transported-child-phase}.  Lemma
\ref{lem:perpendicularity-crossing} gives the unique crossing time
\(s=s_{q+1}\). Lemma~\ref{lem:amplitude-gain} supplies the positive
amplitude \(a=a_{q+1}\), normalized by \(a(1)=1\), and
Corollary~\ref{cor:deterministic-seed-smallness} gives
\begin{equation}
  a(s)\leq\lambda^{-2/\beta}.
  \label{eq:seed-amplitude-smallness}
\end{equation}

Let \(\Phi:=\Phi_q\), \(X:=X_{q+1}\), and
\(p_{\rm term}=e(\pi/2+\lambda^{-\beta/8})\) be the flow, transported phase, and
terminal direction supplied by Lemmas~\ref{lem:cutoff-parent-flow} and
\ref{lem:transported-child-phase}.  Thus \(DX=0\),
\(X(\cdot,1)=p_{\rm term}\cdot x\) on \(B_r\), \(X\) is odd, and
\begin{equation}
  \Phi(t;\sigma,\overline{B_R})
  \subset B_{3r/4}\Subset B_r\Subset U_0.
  \label{eq:one-step-core-flow}
\end{equation}

The leading periodic packet response in
Corollary~\ref{cor:leading-periodic-response} is
\[
  u^{[0]}[z]=m_0(t)z,
  \qquad
  m_0(t):=-C_0\cos\alpha(t)e^\perp(\alpha(t)).
\]
Put \(c_P(x,t):=m_0(t)\cdot\nabla P(x,t)\). On a shared phase, the leading
linearized operator acting on the coefficient is \(D+c_P\). For a coefficient
\(\varphi\) prescribed at time \(\sigma\), its characteristic solution
operator is
\begin{equation}
  H_{\sigma,t}[\varphi](x)
  :=\varphi(\Phi(\sigma;t,x))
  \exp\!\left(-\int_\sigma^t
    c_P(\Phi(\tau;t,x),\tau)\,d\tau\right).
  \label{eq:weighted-coefficient-propagator}
\end{equation}
It satisfies \((D+c_P)H_{\sigma,t}[\varphi]=0\) and
\(H_{\sigma,\sigma}[\varphi]=\varphi\).  Write
\begin{equation}
  \mathcal S_{\sigma,t}[\varphi;\ell,\vartheta]
  :=H_{\sigma,t}[\varphi]
    \sin(\lambda\ell X+\vartheta).
  \label{eq:packet-propagator}
\end{equation}
The cosine slot is \(\vartheta=\pi/2\).
Since the phase is transported and the leading velocity is multiplication by
\(m_0\),
\begin{equation}
  D\mathcal S_{\sigma,t}[\varphi;\ell,\vartheta]
  +u^{[0]}[\mathcal S_{\sigma,t}[\varphi;\ell,\vartheta]]
    \cdot\nabla P=0.
  \label{eq:packet-propagator-identity}
\end{equation}
Oddness fixes the origin. The amplitude equation is
\(\dot a/a=-c_P(0,t)\); hence
\begin{equation}
  H_{\sigma,t}[\varphi](0)
  =\frac{a(t)}{a(\sigma)}\varphi(0).
  \label{eq:origin-weight}
\end{equation}

The normalization below factors out the exact origin growth
\(a(t)/a(\sigma)\). Define
\begin{equation}
  \widetilde H_{\sigma,t}[\varphi]
  :=\frac{a(\sigma)}{a(t)}H_{\sigma,t}[\varphi].
  \label{eq:normalized-weighted-propagator}
\end{equation}
It solves
\[
  \left(D+c_P-c_P(0,t)\right)
  \widetilde H_{\sigma,t}[\varphi]=0,
  \qquad
  \widetilde H_{\sigma,\sigma}[\varphi]=\varphi.
\]
Put
\[
  \Omega_t:=\Phi(t;s,B_R)\Subset B_r.
\]

\begin{lemma}[Normalized weighted-propagator estimates]
\label{lem:normalized-weighted-propagator}
Fix \(T,S\in\mathbb N_0\) and \(\mathcal A>0\), and assume the
cutoff-parent estimates of
Lemma~\ref{lem:mixed-parent-inputs} through ordinary time orders
\(0\leq m\leq(T-1)_+\) and spatial order \(\max\{S+T,2\}\). Assume
\(\Lambda\geq\Lambda_{\rm flow}(b,\beta)\), whose threshold is
independent of derivative order, and take the flow and phase from
Lemmas~\ref{lem:cutoff-parent-flow} and
\ref{lem:transported-child-phase} with
\(J\geq\max\{S+T+1,3\}\). Thus
\eqref{eq:actual-parent-flow-derivatives}--
\eqref{eq:actual-parent-flow-inverse} and
\eqref{eq:actual-parent-phase-derivatives}--
\eqref{eq:actual-parent-phase-affine} hold through every order used below.
Let
\(\varphi\in C^\infty([s,1]\times\mathbb T^2)\), with
\(\operatorname{supp}\varphi(t)\Subset\Omega_t\) for every \(t\), and assume
\begin{equation}
  \|D^m\varphi\|_{C^k(B_r)}
  \leq C_{m,k}L^{C_{m,k}}\mathcal A\lambda^{2k\beta}
  \qquad(0\leq m\leq T,\ m+k\leq S+T).
  \label{eq:diagonal-material-coefficient-class}
\end{equation}
Then, for \(s\leq\sigma\leq t\leq1\),
\begin{equation}
  \left\|D_t^p\partial_\sigma^n
    \widetilde H_{\sigma,t}[\varphi(\sigma)]\right\|_{C^j(B_r)}
  \leq C_{p,n,j}L^{C_{p,n,j}}
    \mathcal A\lambda^{2j\beta},
  \label{eq:mixed-normalized-weight}
\end{equation}
whenever \(p+n\leq T\) and \(p+n+j\leq S+T\). Here \(D_t\) is the
material derivative in the output variables \((x,t)\). For a fixed
admissible triple \((p,n,j)\), the estimate uses datum bounds
\eqref{eq:diagonal-material-coefficient-class} only for
\(0\leq m\leq n\) and \(0\leq k\leq j\). At positive total endpoint order
\(p+n\), it uses ordinary parent time derivatives only through order
\(p+n-1\); at endpoint order zero, only order-zero parent coefficients
enter. The propagated coefficient has support compactly contained in \(\Omega_t\).
The conclusion constants may depend on the finite collection of datum
constants that occurs in the stated ranges.
\end{lemma}

\begin{corollary}[Ordinary-time oscillatory estimates]
\label{lem:mixed-weighted-propagator}
Under the hypotheses of Lemma~\ref{lem:normalized-weighted-propagator}, one
has
\begin{equation}
  \|\partial_t^p\varphi\|_{C^j(B_r)}
  \leq C_{p,j}L^{C_{p,j}}
    \mathcal A\lambda^{2j\beta},
  \label{eq:ordinary-coefficient-class}
\end{equation}
for \(0\leq p\leq T\) and \(0\leq j\leq S\). Suppose in addition that the
phase hypotheses of Lemma~\ref{lem:transported-phase-estimates} and, taking
the coefficient there to be \(\varphi\), the coefficient hypotheses of
Corollary~\ref{cor:oscillatory-product-estimates} hold with output orders
\(T,S\).
Then, for the same ranges, \(\mathfrak t\in\{\sin,\cos\}\),
\(\vartheta\in\mathbb R\), and \(\ell\in\mathbb Z\),
\(1\leq\ell\leq2^{M+1}\),
\begin{equation}
  \left\|\partial_t^p\partial_x^\kappa
    \bigl[\varphi\,\mathfrak t
      (\lambda\ell X+\vartheta)\bigr]\right\|_{C^0}
  \leq C_{p,j}L^{C_{p,j}}\mathcal A
    \lambda^{j+p(1-2\beta)},
  \qquad |\kappa|=j.
  \label{eq:ordinary-live-packet}
\end{equation}
\end{corollary}

\begin{lemma}[Weighted Duhamel estimate]
\label{lem:weighted-duhamel-estimate}
Fix \(T,S\in\mathbb N_0\) and \(\mathcal A>0\). Assume the cutoff-parent
estimates of Lemma~\ref{lem:mixed-parent-inputs} through ordinary time order
\((T-1)_+\) and spatial order \(\max\{S+T,2\}\). Assume
\(\Lambda\geq\Lambda_{\rm flow}(b,\beta)\), independently of
\(T,S\), and take the flow and phase in
Lemmas~\ref{lem:cutoff-parent-flow} and
\ref{lem:transported-child-phase} with
\(J\geq\max\{S+T+1,3\}\). Let
\(\varphi\in C^\infty([s,1]\times\mathbb T^2)\), with
\(\operatorname{supp}\varphi(t)\Subset\Omega_t\) for every \(t\), and
assume that its pulled-back supports are uniformly compact:
\begin{equation}
  K_\varphi
  :=\overline{\bigcup_{\sigma\in[s,1]}
      \Phi(s;\sigma,\operatorname{supp}\varphi(\sigma))}
  \Subset B_R.
  \label{eq:duhamel-uniform-pulled-back-support}
\end{equation}
Assume also
\begin{align*}
  \|\varphi\|_{C^k(B_r)}
  &\leq C_kL^{C_k}\mathcal A\lambda^{2k\beta}
  &&(0\leq k\leq S+T),\\
  \|D^m\varphi\|_{C^k(B_r)}
  &\leq C_{m,k}L^{C_{m,k}}\mathcal A\lambda^{2k\beta}
  &&(1\leq m\leq T-1,\ m+k\leq S+T-1).
\end{align*}
The second range is empty when \(T=0\) or \(T=1\). Set
\[
  h(\sigma)=a(\sigma)\varphi(\sigma),
  \qquad
  g(t)=\int_s^tH_{\sigma,t}[h(\sigma)]\,d\sigma.
\]
Then
\begin{equation}
  \max_{\mathcal D\in\{D,\partial_t\}}
  \|\mathcal D^pg(t)\|_{C^j(B_r)}
  \leq C_{p,j}a(t)L^{C_{p,j}}
    \mathcal A\lambda^{2j\beta},
  \label{eq:weighted-duhamel-mixed-bound}
\end{equation}
for \(0\leq p\leq T\) and \(0\leq j\leq S\).
For a fixed output pair \((p,j)\), only the undifferentiated Duhamel source
\(\varphi\) through spatial order \(j+p\) and the Duhamel-source
derivative pairs
\[
  (m,j+p-1-m),\qquad 0\leq m\leq p-1,
\]
enter. In particular, for \(p\geq1\), no order-\(p\) material or ordinary
time derivative of \(\varphi\) is assumed; for \(p=0\), only the
undifferentiated source is used. The function \(g(t)\) has support compactly
contained in \(\Omega_t\).
\end{lemma}

\subsection{The finite correction cascade}

Products of corrections with the common phase \(X\) generate finitely many
harmonics of \(\lambda X\). We encode their coefficients in arrays so that
the next Duhamel correction can cancel the nonzero harmonics, while the zero
harmonic and the exact velocity remainders remain in the perturbation
residual.

For \(H\geq0\), a fixed phase-harmonic coefficient array is
\[
  \mathbf g
  =\left(g_0,(g^s_h,g^c_h)_{1\leq h\leq H}\right),
\]
with evaluation
\begin{equation}
  \operatorname{Ev}_{\lambda,X}\mathbf g
  :=g_0+\sum_{h=1}^H
    \left(g^s_h\sin(\lambda hX)+g^c_h\cos(\lambda hX)\right).
  \label{eq:phase-array-evaluation}
\end{equation}
This is a presentation rather than an intrinsic decomposition: evaluation
need not be injective. When \(X\) is odd, we call the array odd if \(g_0\)
and every cosine coefficient \(g_h^c\) are odd while every sine coefficient
\(g_h^s\) is even, componentwise for vector-valued arrays. A band with a
fixed phase shift is first put into this convention by the angle-addition
identities. Arrays of different lengths are identified by appending zero
coefficients.

Their product \(\mathbf g\star\mathbf h\) is defined formally: multiply the
displayed sine and cosine slots, apply the product-to-sum identities, and
collect coefficients with the same nonnegative harmonic index.  Thus
\[
  \operatorname{Ev}_{\lambda,X}(\mathbf g\star\mathbf h)
  =\operatorname{Ev}_{\lambda,X}\mathbf g\,
   \operatorname{Ev}_{\lambda,X}\mathbf h.
\]
For each coordinate \(d\), define formal spatial differentiation by
\[
\begin{aligned}
  (\partial_d^{\rm arr}\mathbf g)_0
    &=\partial_dg_0,\\
  (\partial_d^{\rm arr}\mathbf g)^s_\ell
    &=\partial_dg^s_\ell
      -\lambda\ell(\partial_dX)g^c_\ell,\\
  (\partial_d^{\rm arr}\mathbf g)^c_\ell
    &=\partial_dg^c_\ell
      +\lambda\ell(\partial_dX)g^s_\ell.
\end{aligned}
\]
Then
\[
  \operatorname{Ev}_{\lambda,X}
    (\partial_d^{\rm arr}\mathbf g)
  =\partial_d\operatorname{Ev}_{\lambda,X}\mathbf g.
\]
All array operations act on the chosen presentations, not on equivalence
classes modulo evaluation.  We write
\(\boldsymbol\nabla\mathbf g
=(\partial_1^{\rm arr}\mathbf g,\partial_2^{\rm arr}\mathbf g)\).  If
\(\mathbf a\) and
\(\mathbf b\) are vector valued, then
\(\mathbf a\mathbin{\boldsymbol{\cdot}}\mathbf b\) denotes the scalar array
obtained by
applying \(\star\) componentwise and summing.

Define operations on arrays by
\begin{equation}
  \Pi_0\mathbf g=(g_0,(0,0)_{1\leq\ell\leq H}),
  \qquad
  \Pi_{>0}\mathbf g=\mathbf g-\Pi_0\mathbf g.
  \label{eq:formal-harmonic-projections}
\end{equation}
These operations are not defined on arbitrary functions.  In particular,
\(\operatorname{Ev}_{\lambda,X}(\Pi_0\mathbf g)=g_0\) is a spatially varying
coefficient, not a Fourier zero mode.  When \(u^{[M]}\) is applied band by
band, we retain the vector-valued array produced by the formulas of
Proposition~\ref{prop:localized-periodic-velocity-expansion} and denote it by
\(\mathbf u^{[M]}[\mathbf g]\).

Fix an even \(f\in C_c^\infty(B_1)\), equal to one on \(B_{1/2}\), and a
nondecreasing \(W\in C^\infty(\mathbb R)\) such that
\(W=0\) on \((-\infty,0]\) and \(W=1\) on \([1,\infty)\).  Set
\begin{equation}
  w(t):=W(L^8(t-s)),
  \qquad
  \varphi_s(x):=a(s)\lambda^{\beta-1}f(\lambda^{2\beta}x).
  \label{eq:activation-and-seed}
\end{equation}
The principal correction and its activation error are
\begin{equation}
  \zeta_1(t):=w(t)\mathcal S_{s,t}[\varphi_s;1,0],
  \qquad
  \mathcal A_0(t):=w'(t)\mathcal S_{s,t}[\varphi_s;1,0].
  \label{eq:principal-correction}
\end{equation}
They are smooth after extension by zero to \([0,s]\).  Equation
\eqref{eq:origin-weight} gives
\(H_{s,t}[\varphi_s](0)=a(t)\lambda^{\beta-1}\). Hence the coefficient
of the activated correction \(\zeta_1\) at the origin is exactly
\(w(t)a(t)\lambda^{\beta-1}\).

Suppose \(\zeta_1,\ldots,\zeta_j\) have been defined.  Let \(\mathbf z_j\)
be the fixed coefficient array whose evaluation is
\begin{equation}
  \zeta_j
  =\operatorname{Ev}_{\lambda,X}\mathbf z_j
  =\sum_{\ell=1}^{2^j}
   \left(g^s_{\ell,j}\sin(\lambda\ell X)
        +g^c_{\ell,j}\cos(\lambda\ell X)\right).
  \label{eq:correction-presentation}
\end{equation}
For \(j=1\), the \(\ell=2\) slot is zero by the padding convention above.
Put
\[
  \mathbf T_j:=\sum_{i=1}^j\mathbf z_i,
  \qquad
  \Theta_j:=\operatorname{Ev}_{\lambda,X}\mathbf T_j,
  \qquad
  \mathbf T_0=0,\quad \Theta_0=0.
\]

Let \(u^{[M]}\) be the truncated periodic velocity expansion in
Proposition~\ref{prop:localized-periodic-velocity-expansion}. Define the quadratic-increment
array and its evaluation by
\begin{equation}
  \mathbf N_j
  :=\mathbf u^{[M]}[\mathbf z_j]\mathbin{\boldsymbol{\cdot}}
       \boldsymbol\nabla\mathbf T_j
    +\mathbf u^{[M]}[\mathbf T_{j-1}]\mathbin{\boldsymbol{\cdot}}
       \boldsymbol\nabla\mathbf z_j,
  \qquad
  \mathcal N_j:=\operatorname{Ev}_{\lambda,X}\mathbf N_j.
  \label{eq:quadratic-increment}
\end{equation}
Since \(u^{[0]}[z]=m_0(t)z\), define its bandwise coefficient array on
\(\mathbf z_j\) by
\[
  \mathbf u^{[0]}[\mathbf z_j]
  :=\left(0,
    \bigl(m_0(t)g^s_{\ell,j},m_0(t)g^c_{\ell,j}\bigr)_{1\leq\ell\leq2^j}
  \right).
\]
Regard \(\nabla P\) as a vector-valued array with only its zero-harmonic
coefficient, and set
\[
  \mathbf R_j^P
  :=\bigl(\mathbf u^{[M]}[\mathbf z_j]
          -\mathbf u^{[0]}[\mathbf z_j]\bigr)
       \mathbin{\boldsymbol{\cdot}}\boldsymbol\nabla P.
\]
Define the structured-defect array and function by
\begin{equation}
  \mathbf Q_j:=\Pi_{>0}\mathbf N_j+\mathbf R_j^P,
  \qquad
  \mathcal Q_j:=\operatorname{Ev}_{\lambda,X}\mathbf Q_j.
  \label{eq:structured-defect}
\end{equation}
Thus \(\mathcal Q_j\) is the correctable nonzero-harmonic defect.
The part deposited directly into the perturbation residual is
\begin{align}
  \mathcal E_j:={}&
    \operatorname{Ev}_{\lambda,X}(\Pi_0\mathbf N_j)\notag\\
  &+(u_{\mathbb T}-u^{[M]})[\zeta_j]\cdot\nabla P\notag\\
  &+(u_{\mathbb T}-u^{[M]})[\zeta_j]\cdot\nabla\Theta_j\notag\\
  &+(u_{\mathbb T}-u^{[M]})[\Theta_{j-1}]\cdot\nabla\zeta_j.
  \label{eq:deposited-packet-error}
\end{align}
The same field \(P\) supplies both the transport velocity and the gradient
in \eqref{eq:packet-propagator-identity}.

The array \(\mathbf Q_j\) has nonzero harmonics at most \(2^{j+1}\); write
\[
  \mathbf Q_j
  =\left(0,(h^s_{\ell,j},h^c_{\ell,j})_{1\leq\ell\leq2^{j+1}}\right).
\]
For \(1\leq j<M\), define
\begin{equation}
  \zeta_{j+1}(t)
  :=-\sum_{\ell=1}^{2^{j+1}}\int_s^t
    \left(
      \mathcal S_{\sigma,t}[h^s_{\ell,j}(\sigma);\ell,0]
      +\mathcal S_{\sigma,t}[h^c_{\ell,j}(\sigma);\ell,\pi/2]
    \right)\,d\sigma.
  \label{eq:duhamel-correction}
\end{equation}
Apply \(D+u^{[0]}[\cdot]\cdot\nabla P\) to each integral.  Identity
\eqref{eq:packet-propagator-identity} annihilates the integrand, while the
upper endpoint is respectively
\(h^s_{\ell,j}\sin(\lambda\ell X)\) or
\(h^c_{\ell,j}\cos(\lambda\ell X)\).  Summing the fixed array coefficients
therefore gives
\begin{equation}
  D\zeta_{j+1}
  +u^{[0]}[\zeta_{j+1}]\cdot\nabla P
  =-\mathcal Q_j.
  \label{eq:duhamel-cancellation}
\end{equation}
This is an explicit recursion: \(\mathcal Q_j\) depends only on
\(\zeta_1,\ldots,\zeta_j\), which are known before
\(\zeta_{j+1}\) is defined. It is also triangular after differentiation.
For an order-\(p\), \(C^k\) output, the weighted Duhamel estimate uses the
undifferentiated defect through \(C^{k+p}\) and positive defect time orders
only through \(p-1\); the exact pairs are recorded below in
Lemma~\ref{lem:density-first-mixed-closure}.

\begin{lemma}[Parity of the array operations]
\label{lem:array-parity}
Suppose that \(X\) and \(P\) are odd. Applying
Proposition~\ref{prop:localized-periodic-velocity-expansion} harmonic by
harmonic to an odd array produces an odd vector-valued array. The grouped
exact velocity remainder is odd. Moreover, the propagator
\(H_{\sigma,t}\) preserves parity, and, when initialized by an odd array,
the recursion
\eqref{eq:principal-correction}--\eqref{eq:duhamel-correction} produces odd
\(\zeta_j,\mathcal Q_j,\mathcal E_j\).
\end{lemma}

\begin{proof}
For a sine input, the coefficient recursion in the proof of
Proposition~\ref{prop:localized-periodic-velocity-expansion} gives
\[
  C_\gamma^s(-x)=(-1)^{|\gamma|}C_\gamma^s(x),
  \qquad
  C_\gamma^c(-x)=(-1)^{|\gamma|+1}C_\gamma^c(x);
\]
the cosine input has the reversed parities. This follows inductively because
\(\partial_dX\) is even. The vanishing identities
\eqref{eq:periodic-profile-identities} then leave even sine coefficients and
odd cosine coefficients in the velocity array. The Fourier multiplier symbol
of the exact velocity is even, so the exact velocity and its grouped
remainder preserve oddness.

The flow of the odd field \(u_{\mathbb T}(P)\) is odd, while \(c_P\) is
even; hence \eqref{eq:weighted-coefficient-propagator} preserves parity.
Finally, products with gradients, the harmonic projections, and the Duhamel
step preserve the stated parities. Induction in \(j\) proves the last claim.
\end{proof}

\begin{lemma}[Exact cascade identity]
\label{lem:exact-cascade}
For \(1\leq J\leq M\),
\begin{align}
 &\partial_t\Theta_J
  +u_{\mathbb T}(P)\cdot\nabla\Theta_J
  +u_{\mathbb T}(\Theta_J)\cdot\nabla P
  +u_{\mathbb T}(\Theta_J)\cdot\nabla\Theta_J\notag\\
 &\hspace{35mm}
  =\mathcal A_0+\sum_{j=1}^J\mathcal E_j+\mathcal Q_J.
  \label{eq:exact-cascade-identity}
\end{align}
\end{lemma}

\begin{proof}
For \(J=1\), one has \(\Theta_1=\zeta_1\), and
\eqref{eq:packet-propagator-identity} together with
\eqref{eq:principal-correction} gives
\[
  D\zeta_1+u^{[0]}[\zeta_1]\cdot\nabla P=\mathcal A_0.
\]
The remaining parent interaction splits exactly as
\[
  (u_{\mathbb T}-u^{[0]})[\zeta_1]\cdot\nabla P
  =(u^{[M]}-u^{[0]})[\zeta_1]\cdot\nabla P
   +(u_{\mathbb T}-u^{[M]})[\zeta_1]\cdot\nabla P.
\]
Since \(\mathbf T_0=0\), the finite part of
\(u_{\mathbb T}(\zeta_1)\cdot\nabla\zeta_1\) is
\(\mathcal N_1\), and its grouped exact remainder is the third term in
\eqref{eq:deposited-packet-error}. Splitting \(\mathbf N_1\) by
\(\Pi_{>0}+\Pi_0\) and using
\eqref{eq:structured-defect}--\eqref{eq:deposited-packet-error} gives
\[
  \mathcal A_0+\mathcal Q_1+\mathcal E_1,
\]
which proves the base case.

Assume the identity at level \(J<M\). The exact nonlinear increment after
adding \(\zeta_{J+1}\) is
\[
  \partial_t\zeta_{J+1}
  +u_{\mathbb T}(P)\cdot\nabla\zeta_{J+1}
  +u_{\mathbb T}(\zeta_{J+1})\cdot\nabla P
  +u_{\mathbb T}(\zeta_{J+1})\cdot\nabla\Theta_{J+1}
  +u_{\mathbb T}(\Theta_J)\cdot\nabla\zeta_{J+1}.
\]
By \eqref{eq:duhamel-cancellation}, the first two terms split as
\begin{align*}
  &D\zeta_{J+1}
   +u_{\mathbb T}(\zeta_{J+1})\cdot\nabla P\\
  &\quad=-\mathcal Q_J
   +(u^{[M]}-u^{[0]})[\zeta_{J+1}]\cdot\nabla P
   +(u_{\mathbb T}-u^{[M]})[\zeta_{J+1}]\cdot\nabla P.
\end{align*}
The last two terms in the increment are
\begin{align*}
  \mathcal N_{J+1}
  &+(u_{\mathbb T}-u^{[M]})[\zeta_{J+1}]
      \cdot\nabla\Theta_{J+1}\\
  &+(u_{\mathbb T}-u^{[M]})[\Theta_J]
      \cdot\nabla\zeta_{J+1}.
\end{align*}
The positive-harmonic part of \(\mathcal N_{J+1}\), together with the
nonleading finite parent interaction, is \(\mathcal Q_{J+1}\). Its
zero-harmonic part and the three grouped exact remainders are
\(\mathcal E_{J+1}\). Thus the increment is exactly
\[
  -\mathcal Q_J+\mathcal Q_{J+1}+\mathcal E_{J+1}.
\]
It cancels the terminal \(+\mathcal Q_J\) in the induction hypothesis and
proves the claim. In particular, the two ordered quadratic cross terms are
both contained in \(\mathcal N_{J+1}\).
\end{proof}

Define the new perturbation and its exact residual by
\begin{equation}
  \theta_{q+1}:=\Theta_M,
  \qquad
  E_{q+1}^{\rm pert}
  :=\mathcal A_0+\sum_{j=1}^M\mathcal E_j+\mathcal Q_M.
  \label{eq:packet-residual-cascade}
\end{equation}
Equivalently, Lemma~\ref{lem:exact-cascade} states the intrinsic identity
\begin{align}
  E_{q+1}^{\rm pert}={}&
  \partial_t\theta_{q+1}
  +u_{\mathbb T}(P_{q+1})\cdot\nabla\theta_{q+1}
  +u_{\mathbb T}(\theta_{q+1})\cdot\nabla P_{q+1}\notag\\
  &+u_{\mathbb T}(\theta_{q+1})\cdot\nabla\theta_{q+1}.
  \label{eq:exact-packet-residual}
\end{align}
Both \(\theta_{q+1}\) and \(E_{q+1}^{\rm pert}\) vanish to infinite
order in time at \(t=s_{q+1}\). Indeed, since \(W\) vanishes to infinite order at zero, we have
\(\lvert\zeta_1(t)\rvert\leq C_N(t-s)^N\) for every \(N\), and
\eqref{eq:duhamel-correction} gives inductively
\[
  \lvert\zeta_{j+1}(t)\rvert\leq C_N(t-s)^{N+1}.
\]
The same argument after differentiation, followed by
\eqref{eq:packet-residual-cascade}, proves the assertion.

\subsection{Descending derivative orders and residual estimates}

Appendix~\ref{sec:smooth-force} proves the all-order parent and weighted
propagator estimates used in this subsection. Put \(\gamma=1-2\beta\). One
correction step uses \(M+T+2\) additional coefficient derivatives and one
additional phase derivative; the following descending derivative orders
reserve that loss.
For requested output orders \(T,K\), define
\begin{equation}
  \mathcal K_{M+1}:=K,
  \qquad
  \mathcal K_j:=K+(M-j+1)(M+T+2)
  \quad(1\leq j\leq M).
  \label{eq:descending-derivative-blocks}
\end{equation}
The identity \(\mathcal K_j=\mathcal K_{j+1}+M+T+2\) supplies the additional
coefficient orders required by
Proposition~\ref{prop:localized-periodic-velocity-expansion} at the next
correction. Lemma~\ref{lem:mixed-parent-inputs} supplies every cutoff-parent
order; Lemma~\ref{lem:transported-child-phase}, via
Lemma~\ref{lem:cutoff-parent-flow}, supplies the additional phase derivative.
At the passage from \(\zeta_j\) to \(\mathcal Q_j\), the exact orders used
by that proposition are
\begin{equation}
  N_{\mathrm{prof}}=\mathcal K_{j+1}+M+T+2=\mathcal K_j,
  \qquad
  N_{\mathrm{phase}}=\mathcal K_j+1.
  \label{eq:descending-profile-and-phase-orders}
\end{equation}
Set
\begin{equation}
  \Lambda_{\rm cas}(b,\beta)
  :=\max\left\{
    \Lambda_{\rm seed}(b,\beta),
    \Lambda_{\rm geom}(b,\beta),
    \Lambda_{\rm flow}(b,\beta),
    \Lambda_{\rm pkt}(b,\beta)
  \right\}.
  \label{eq:cascade-threshold}
\end{equation}
This threshold is independent of \(T,K\); derivative orders affect only the
constants in the estimates below.

\begin{lemma}[Correction estimates]
\label{lem:correction-estimates}
Assume the adjacent-layer setup in
Definition~\ref{def:adjacent-layer-setup}, fix \(T,K\in\mathbb N_0\), and let the
derivative orders \(\mathcal K_j\) be defined by
\eqref{eq:descending-derivative-blocks}. Assume
\(\Lambda\geq\Lambda_{\rm cas}(b,\beta)\). Take the local parent estimates
\eqref{eq:mixed-cutoff-local-material-estimates} in the ranges
\[
  \begin{aligned}
    D^pv:\quad&
      0\leq p\leq(T-1)_+,\quad d\in\mathbb N_0,\quad
      p+d\leq\mathcal K_1+T,\\
    D^p\nabla P:\quad&
      0\leq p\leq T,\quad d\in\mathbb N_0,\quad
      p+d\leq\mathcal K_1+T.
  \end{aligned}
\]
and take \eqref{eq:actual-parent-flow-derivatives}--
\eqref{eq:actual-parent-phase-affine} with
\(J\geq\mathcal K_1+T+1\). These inputs are supplied, respectively, by
Lemma~\ref{lem:mixed-parent-inputs} and by
Lemmas~\ref{lem:cutoff-parent-flow}--\ref{lem:transported-child-phase}.
Consequently \eqref{eq:periodic-packet-velocity-input}--
\eqref{eq:periodic-packet-phase-input} hold for every \(1\leq j\leq M\)
with \(N_{\rm prof}=\mathcal K_j\) and
\(N_{\rm phase}=\mathcal K_j+1\).
Let \(P,D,X,H_{\sigma,t}\), the seed, and the correction cascade be defined
by \eqref{eq:actual-parent}--\eqref{eq:duhamel-correction}, and use
\eqref{eq:amplitude-no-collapse}. For
\(\mathcal D\in\{D,\partial_t\}\), \(p,k\in\mathbb N_0\),
\(0\leq p\leq T\), \(1\leq j\leq M\), \(1\leq\ell\leq2^j\), and
\(k+p\leq\mathcal K_j+T\), the coefficients in
\eqref{eq:correction-presentation} satisfy
\begin{equation}
  \|\mathcal D^p g^s_{\ell,j}\|_{C^k}
  +\|\mathcal D^p g^c_{\ell,j}\|_{C^k}
  \leq C_{p,k,j}a(t)L^{C_{p,k,j}}
    \lambda^{-1+(2k+2-j)\beta}.
  \label{eq:correction-coefficient-bound}
\end{equation}
If
\[
  \mathcal Q_j
  =\sum_{\ell=1}^{2^{j+1}}
   \left(h^s_{\ell,j}\sin(\lambda\ell X)
        +h^c_{\ell,j}\cos(\lambda\ell X)\right),
\]
then
\begin{equation}
  \|\mathcal D^p h^s_{\ell,j}\|_{C^k}
  +\|\mathcal D^p h^c_{\ell,j}\|_{C^k}
  \leq C_{p,k,j}a(t)L^{C_{p,k,j}}
    \lambda^{-1+(2k+1-j)\beta},
  \label{eq:structured-coefficient-bound}
\end{equation}
  for \(\mathcal D\in\{D,\partial_t\}\), \(p,k\in\mathbb N_0\),
  \(0\leq p\leq T\), \(1\leq j\leq M\),
  \(1\leq\ell\leq2^{j+1}\), and
  \(k+p\leq\mathcal K_{j+1}+T\).  All presented coefficients are supported in
the transported perturbation region. Every correction \(\zeta_j\) is odd,
as are \(\mathcal Q_j\) and \(\mathcal E_j\).
\end{lemma}

\begin{proof}
The seed coefficient has size
\(a(t)\lambda^{-1+\beta+2k\beta}\), which is
\eqref{eq:correction-coefficient-bound} for \(j=1\). The propagator identity
shows inductively that every correction coefficient has the form
\[
  g_{\ell,i}=a\,\widetilde g_{\ell,i}.
\]
The amplitude equation gives
\(\lvert\partial_t^ma^{\pm1}\rvert\leq C_mL^{C_m}a^{\pm1}\);
the same estimate applies under material differentiation because
\(Da=\dot a\).
Thus the normalized coefficients satisfy the constant-size hypotheses of
Proposition~\ref{prop:localized-periodic-velocity-expansion} and
Corollary~\ref{cor:common-phase-interactions}, with size
\(\lambda^{-1+(2-i)\beta}\). After applying those results band by band, we
restore one factor \(a\) in a linear or background term and two factors in a
quadratic term. The no-collapse estimate and \(a(1)=1\) give
\(a(t)^2\leq Ca(t)\).

For the induction step, assume that every correction with
\(1\leq i\leq j\) has the preceding coefficient size. The corrections share
the parent, transported phase, and compact transported coefficient support;
their harmonics are at most \(2^{j+1}\leq2^{M+1}\). Every pair in
\(\mathcal N_j\) has indices \(i_1+i_2\geq j+1\).  The common-phase
estimate \eqref{eq:periodic-packet-anisotropic} gives
\[
  -2+(4-i_1-i_2)\beta+1-2\beta+2k\beta
  \leq-1+(2k+1-j)\beta.
\]
The background estimate \eqref{eq:periodic-packet-background} gives another
factor \(\lambda^{-\beta}\) in the interaction with \(\nabla P\).
Fix an individual output pair \((p,k)\) with
  \(k+p\leq\mathcal K_{j+1}+T\).  The localized velocity expansion at this
pair is applied with output parameters \((T,K)=(p,k)\), as are
Corollaries~\ref{cor:oscillatory-product-estimates} and
\ref{cor:common-phase-interactions}. Their constants may depend on
\((p,k)\), but not on \(\lambda\). This application requires
coefficient order
\[
  k+M+p+2
  \leq\mathcal K_{j+1}+T+M+2
  =\mathcal K_j
\]
and phase order one larger. Both are supplied by the parent and phase
hypotheses of the lemma. The preceding exponent calculation
therefore proves \eqref{eq:structured-coefficient-bound} throughout its
diagonal range.

If \(j<M\), the Duhamel step constructs the next correction. Fix \((p,k)\)
with \(k+p\leq\mathcal K_{j+1}+T\), and let \(h\) be one coefficient of
\(\mathcal Q_j\). Denote the corresponding Duhamel coefficient in
\(\zeta_{j+1}\) by
\[
  g_{j+1}(t)=-\int_s^tH_{\sigma,t}[h(\sigma)]\,d\sigma.
\]
Write
\[
  h=a\varphi,
  \qquad
  \varphi=\frac{h}{a},
  \qquad
  \mathfrak a_j=\lambda^{-1+(1-j)\beta}.
\]
The amplitude equation gives
\(\lvert\partial_t^m(a^{-1})\rvert\leq C_mL^{C_m}a^{-1}\).
Here is the support induction used in this step.  Put
\[
  K_0:=\operatorname{supp}\bigl(f(\lambda^{2\beta}\,\cdot)\bigr)
  \Subset B_R,
  \qquad
  K_t:=\Phi(t;s,K_0).
\]
The seed coefficient is supported in \(K_t\).  Products, spatial
derivatives, and the array projections \(\Pi_0\) and \(\Pi_{>0}\) do not
enlarge coefficient support, so the local truncated-expansion formulas
show that every coefficient of \(\mathcal Q_j(t)\), including \(h(t)\), is
supported in \(K_t\).  Division by the time-dependent scalar \(a(t)\) does
not change this support.  Consequently
\[
  \Phi(s;\sigma,\operatorname{supp}\varphi(\sigma))
  \subset K_0
  \qquad (\sigma\in[s,1]),
\]
which is \eqref{eq:duhamel-uniform-pulled-back-support}.  Moreover, flow
composition shows that each integrand in the Duhamel formula is supported
in
\[
  \Phi(t;\sigma,K_\sigma)
  =\Phi(t;s,K_0)=K_t.
\]
Thus the new coefficient has the same transported support, completing the
support induction.  The
order-zero defect bound through spatial order \(k+p\), and the lower-time
defect bounds at
\[
  (m,k+p-1-m),\qquad 0\leq m\leq p-1,
\]
imply precisely the Duhamel-source hypotheses of
Lemma~\ref{lem:weighted-duhamel-estimate}, with its parameters
\((T,S,\mathcal A)=(p,k,\mathfrak a_j)\). The undifferentiated pair belongs
to the defect range because
\(k+p\leq\mathcal K_{j+1}+T\), and every endpoint pair belongs to it
because
\[
  m+(k+p-1-m)=k+p-1<\mathcal K_{j+1}+T.
\]
The local parent estimates from Lemma~\ref{lem:mixed-parent-inputs}, applied
through ordinary time order \((p-1)_+\) and spatial order \(k+p\), supply
exactly the cutoff-parent hypotheses of
Lemma~\ref{lem:weighted-duhamel-estimate}.
Consequently,
\[
  \max_{\mathcal D\in\{D,\partial_t\}}
  \|\mathcal D^p g_{j+1}\|_{C^k}
  \leq C_{p,k,j}a(t)L^{C_{p,k,j}}
    \lambda^{-1+(2k+1-j)\beta},
\]
which is \eqref{eq:correction-coefficient-bound} at correction index
\(j+1\). For \(p=0\), this step uses the already formed order-zero defect;
for \(p\geq1\), it uses only defect orders at most \(p-1\).

Thus, for \(1\leq j<M\), the time-order-zero proof alternates the coefficient
estimate at correction index \(j\), the corresponding structured-defect
estimate, and the coefficient estimate at index \(j+1\); after the last
coefficient, it forms the final uncancelled structured defect at index
\(M\). At each
\(p\geq1\), after all lower-time coefficient and defect estimates are known,
it first obtains the order-\(p\) coefficient estimates from the seed datum
prescribed at the fixed initial time \(s\) and the preceding Duhamel step,
and only then forms the order-\(p\)
structured-defect estimates by the interaction calculation. Parity and
support follow from the odd phase, even weight, harmonic projections, and
the transported coefficient support.  The exact periodic multiplier
preserves oddness; no support property is assigned to its grouped remainder.
\end{proof}

For each correction index \(j\), let \(\mathbf C_{p,j}\) denote the
order-\(p\) coefficient estimate \eqref{eq:correction-coefficient-bound} in
the diagonal range \(k+p\leq\mathcal K_j+T\), and let
\(\mathbf R_{p,j}\) denote the order-\(p\) structured-defect estimate
\eqref{eq:structured-coefficient-bound} in the range
\(k+p\leq\mathcal K_{j+1}+T\).

\begin{lemma}[Within-child mixed closure]
\label{lem:density-first-mixed-closure}
Assume precisely the hypotheses and order ranges of
Lemma~\ref{lem:correction-estimates}, including
\(\Lambda\geq\Lambda_{\rm cas}(b,\beta)\). Within the single child
construction, the estimates close in the order
\[
  \mathbf C_{0,1},\mathbf R_{0,1},
  \ldots,\mathbf C_{0,M},\mathbf R_{0,M},
\]
followed, for each \(1\leq p\leq T\), by
\[
  \mathbf C_{p,1},\ldots,\mathbf C_{p,M},
  \quad
  \mathbf R_{p,1},\ldots,\mathbf R_{p,M}.
\]
For \(1\leq j<M\), the order-\(p\), \(C^k\) coefficient at index \(j+1\)
uses \(\mathbf R_{0,j}\) through \(C^{k+p}\) and, if \(p\geq2\),
\(\mathbf R_{m,j}\) through \(C^{k+p-1-m}\) for \(1\leq m\leq p-1\).
Thus each positive-time coefficient pass uses only defect time orders below
\(p\); the sharp order-\(p\) parent estimate enters only afterwards, when
the order-\(p\) defects and force-increment terms are formed.
\end{lemma}

\paragraph{Proof outline.}
For the seed datum at time \(s\), set
\begin{equation}
  \widehat\varphi_s:=\frac{\varphi_s}{a(s)}
  =\lambda^{\beta-1}f(\lambda^{2\beta}x),
  \qquad
  \widehat\varphi(\sigma,x)
  :=\widehat\varphi_s(\Phi(s;\sigma,x)).
  \label{eq:normalized-seed-extension}
\end{equation}
Then \(D\widehat\varphi=0\), its support is compactly contained in
\(\Omega_\sigma\), and its pulled-back support is the fixed compact set
\(\operatorname{supp}\widehat\varphi_s\Subset B_R\). It therefore supplies
\(\mathbf C_{0,1}\) and every fixed-\(s\) endpoint derivative.
At order zero, each newly obtained coefficient forms its defect, which is
then the Duhamel datum for the next coefficient.

For \(p\geq1\), assume all lower time orders. The precise source pairs in
Lemma~\ref{lem:weighted-duhamel-estimate} are
\[
  (0,k+p),\qquad (m,k+p-1-m)\quad(1\leq m\leq p-1).
\]
They are available from the order-zero enlarged spatial range and the
strictly lower time orders. Hence every \(\mathbf C_{p,j}\) is proved before
any \(\mathbf R_{p,j}\). The interaction estimates then give all
\(\mathbf R_{p,j}\) and the residual terms. The complete derivative-range
verification is given in Appendix~\ref{sec:smooth-force}.

Thus the input chain is explicit: Lemma~\ref{lem:mixed-parent-inputs}
supplies all cutoff-parent spatial orders, Lemmas~\ref{lem:cutoff-parent-flow}
and \ref{lem:transported-child-phase} supply the phase orders, and
\(K_q\geq d(\beta)\) is used only for the final unsmoothed residual and
commutator estimates. For the next layer, the logical order is the
order-zero cascade, the origin geometry of \(\theta_{q+1}\), and only then
the crossing and thresholds for level \(q+2\).

\begin{lemma}[Packet residual estimate]
\label{lem:packet-residual-estimate}
Assume the adjacent-layer setup in
Definition~\ref{def:adjacent-layer-setup} and
\(\beta\leq1/40\). Set
\[
  m_{q+1}:=\left\lfloor\frac1{4\beta}\right\rfloor-10.
\]
For each \(m,k\in\mathbb N_0\) with \(m+k\leq m_{q+1}\), assume the
hypotheses of
Proposition~\ref{prop:localized-periodic-velocity-expansion} with output
orders \(T=m\), \(K=k+1\); the extra spatial order is reserved for the
gradients in \eqref{eq:deposited-packet-error}. Use
Lemma~\ref{lem:correction-estimates} with
\[
  T=m,\qquad K=k+M+3.
\]
These additional \(M+2\) orders supply the coefficient range needed when
\eqref{eq:ordinary-live-packet} is applied to \(\mathcal Q_M\). The
cutoff-parent and phase inputs for this enlarged cascade range come from
Lemma~\ref{lem:mixed-parent-inputs} and
Lemmas~\ref{lem:cutoff-parent-flow}--\ref{lem:transported-child-phase};
only the final grouped remainder uses the finite unsmoothed range authorized
by \(K_q\geq d(\beta)\). The exact velocity remainder in
\eqref{eq:periodic-packet-leak} is kept as one grouped global remainder and
is differentiated only by ordinary time derivatives.
There is a finite threshold
\[
  \Lambda_{\rm res}=\Lambda_{\rm res}(b,\beta)
  \geq\max\{\Lambda_{\rm cas}(b,\beta),\Lambda_{\rm pkt}(b,\beta)\}
\]
obtained by taking the maximum over this finite set of output pairs, such
that, whenever \(\Lambda\geq\Lambda_{\rm res}\),
\begin{equation}
  \|\partial_t^mE_{q+1}^{\rm pert}\|_{C^k}
  \leq\frac12\lambda^{-1/4}
  \qquad(m+k\leq m_{q+1}).
  \label{eq:packet-residual-mixed-smallness}
\end{equation}
\end{lemma}

\begin{proof}
Use the named decomposition
\eqref{eq:packet-residual-cascade}. Before powers of \(L\) and finite
constants are absorbed, its three parts satisfy, for \(1\leq j\leq M\),
\begin{align}
  \|\partial_t^m\mathcal A_0\|_{C^k}
  &\leq C_\beta L^C
    \lambda^{k+m\gamma-1+\beta-2/\beta},
    \notag\\
  \|\partial_t^m\mathcal E_j\|_{C^k}
  &\leq C_{\beta,j}L^C\left(
    \lambda^{-1+(2k+1-j)\beta}\right.
    \notag\\
  &\hspace{31mm}\left.
    +\lambda^{m+2(k+1)(1-\beta)-M\beta}
    +\lambda^{m+2(k+1)(1-\beta)+1-M\beta}
    \right),
    \notag\\
  \|\partial_t^m\mathcal Q_M\|_{C^k}
  &\leq C_\beta L^C
    \lambda^{k-1+(1-M)\beta+m\gamma}.
  \label{eq:named-residual-estimates}
\end{align}
On \(\operatorname{supp}w'\), one has \(0\leq t-s\leq L^{-8}\).
Equation \eqref{eq:new-amplitude-equation} and
\eqref{eq:actual-parent-gradient-geometry} give
\[
  \left|\frac{\dot a}{a}\right|\leq C\Lambda^b,
  \qquad
  \frac{a(t)}{a(s)}
  \leq\exp(C\Lambda^bL^{-8})\leq C_\beta.
\]
Thus \(a(t)\leq C_\beta\lambda^{-2/\beta}\) on
\(\operatorname{supp}w'\). The first estimate in
\eqref{eq:named-residual-estimates} follows from
this bound and \eqref{eq:ordinary-live-packet}; derivatives of \(w\) cost
only powers of \(L\). For \(\mathcal E_j\), the first displayed summand is
the zero-harmonic coefficient
\(\operatorname{Ev}_{\lambda,X}(\Pi_0\mathbf N_j)\), estimated by
\eqref{eq:periodic-packet-anisotropic}. In the order of
\eqref{eq:deposited-packet-error}, the second displayed summand bounds
\((u_{\mathbb T}-u^{[M]})[\zeta_j]\cdot\nabla P\), while the third bounds
the analogous terms for the packets \(\zeta_j,\Theta_j,\Theta_{j-1}\)
multiplied by \(\nabla\Theta_j\) or \(\nabla\zeta_j\). Leibniz' rule and
\eqref{eq:periodic-packet-leak} give, respectively,
\[
  CL^C\lambda^{m+2(k+1)(1-\beta)-M\beta},
  \qquad
  CL^C\lambda^{m+2(k+1)(1-\beta)-1+2\beta-M\beta};
\]
the latter is bounded by the displayed exponent with \(+1\).
These estimates use
\eqref{eq:periodic-packet-leak}, with output orders \(T=m\), \(K=k+1\),
and the coefficient bounds from Lemma~\ref{lem:correction-estimates}; the
normalization by \(a(t)\)
in its proof matches the constant-size packet hypotheses. The grouped
remainder is global and is differentiated only by ordinary time derivatives.
Finally, the estimate for \(\mathcal Q_M\) is
\eqref{eq:structured-coefficient-bound} at \(j=M\), followed by
\eqref{eq:ordinary-live-packet} to include derivatives falling on the
oscillatory factor. Its coefficient demand is
\(p'+k'\leq m+k+M+3\), which is exactly the terminal range supplied by the
additional orders in Lemma~\ref{lem:packet-residual-estimate}.

For \(m,k\in\mathbb N_0\) with \(m+k\leq m_{q+1}\), the exponents in the
three named estimates are bounded above as follows:
\begin{equation}
  \begin{aligned}
    \mathcal A_0:\quad&
      -\frac7{4\beta}-11+\beta,\\
    \mathcal E_j:\quad&
      -\frac12-20\beta,\qquad
      -\frac1{2\beta}-18-2\beta,\qquad
      -\frac1{2\beta}-17-2\beta,\\
    \mathcal Q_M:\quad&
      -\frac3{4\beta}-11+\beta.
  \end{aligned}
  \label{eq:residual-negative-powers}
\end{equation}
Indeed, \(M\beta\geq\beta^{-1}\),
\(m+2k\leq2m_{q+1}\), and \(j\geq1\) in the
\(\mathcal E_j\) zero-harmonic term. The number of
corrections and harmonics depends only on \(\beta\).  A finite threshold
therefore absorbs all logarithmic factors and proves
\eqref{eq:packet-residual-mixed-smallness}.
\end{proof}

\subsection{The cutoff commutator and exact stage recursion}

Recall the cutoff commutator \(\mathcal C_\mu\) from
\eqref{eq:periodic-cutoff-commutator-definition}.
It is odd and has zero mean when \(h\) is odd.  Since
\(u_{\mathbb T}\) preserves Fourier support,
\begin{equation}
  \operatorname{supp}\widehat{\mathcal C_\mu(h)}\subset B_{4\mu}.
  \label{eq:commutator-bandwidth}
\end{equation}
If \(\operatorname{supp}\widehat h\subset B_{2\nu}\) and
\(\mu\geq4\nu\), then
\[
  S_\mu h=h,
  \qquad
  S_\mu\bigl(u_{\mathbb T}(h)\cdot\nabla h\bigr)
  =u_{\mathbb T}(h)\cdot\nabla h,
\]
and hence
\begin{equation}
  \mathcal C_\mu(h)=0.
  \label{eq:bandlimited-commutator-zero}
\end{equation}

\begin{lemma}[Mixed cutoff commutator]
\label{lem:mixed-cutoff-commutator}
Let \(\mu\geq1\) and
\(h\in C^\infty([0,1]\times\mathbb T^2)\).  For integers \(m,k\geq0\) and
\(s_0\geq1\),
\begin{equation}
  \|\partial_t^m\mathcal C_\mu(h)\|_{C^k}
  \leq C_{m,k,s_0}\mu^{-s_0}
  \left(
    \max_{0\leq p\leq m}
      \|\partial_t^p h\|_{C^{k+s_0+4}}
  \right)^2.
  \label{eq:mixed-commutator-estimate}
\end{equation}
\end{lemma}

\begin{proof}
Put \(h_>=(I-S_\mu)h\).  Exact expansion gives
\begin{align*}
  \mathcal C_\mu(h)={}&
   (I-S_\mu)(u_{\mathbb T}(h)\cdot\nabla h)
   -u_{\mathbb T}(h_>)\cdot\nabla h
   -u_{\mathbb T}(h)\cdot\nabla h_>
   +u_{\mathbb T}(h_>)\cdot\nabla h_>.
\end{align*}
Set
\[
  \mathfrak H
  :=\max_{0\leq p\leq m}
    \|\partial_t^ph\|_{C^{k+s_0+4}}.
\]
After applying \(\partial_t^m\) and Leibniz' rule, the periodic Schauder and
tail estimates give, for every resulting summand,
\[
  \|(I-S_\mu)(u_{\mathbb T}(f)\cdot\nabla g)\|_{C^k}
  +\|u_{\mathbb T}((I-S_\mu)f)\cdot\nabla g\|_{C^k}
  \leq C\mu^{-s_0}\mathfrak H^2,
\]
where \(f\) and \(g\) are time derivatives of \(h\).  The same bound holds
when the tail falls on the differentiated factor; a term with two tails is
smaller because \(\mu\geq1\).  The explicit gradient uses one derivative,
and \eqref{eq:periodic-spectral-tail} uses three more beyond the target
order.  This proves \eqref{eq:mixed-commutator-estimate}.
\end{proof}

\begin{corollary}[Stage commutator estimate]
\label{cor:stage-commutator-estimate}
Assume the adjacent-layer setup in
Definition~\ref{def:adjacent-layer-setup}, and take
\(h=\rho^{(q)}\). Assume \(\beta\leq1/40\), and set
\[
  m_{q+1}:=\left\lfloor\frac1{4\beta}\right\rfloor-10.
\]
There is a finite threshold
\[
  \Lambda_{\rm com}=\Lambda_{\rm com}(b,\beta)
\]
such that, whenever \(\Lambda\geq\Lambda_{\rm com}\) and
\(m,k\in\mathbb N_0\) satisfy \(m+k\leq m_{q+1}\),
\begin{equation}
  \|\partial_t^m\mathcal C_\mu(\rho^{(q)})\|_{C^k}
  \leq\frac12\lambda^{-1/4}.
  \label{eq:commutator-mixed-smallness}
\end{equation}
\end{corollary}

\begin{proof}

Choose
\[
  s_0=k+2m+4+\left\lceil\frac2\beta\right\rceil.
\]
The largest mixed parent order required is
\[
  k+s_0+4+2m
  =2k+4m+8+\left\lceil\frac2\beta\right\rceil
  \leq\frac3\beta-31
  <\frac4{\beta^2},
\]
where \(m+k\leq m_{q+1}\) and \(\beta\leq1/40\) were used.
The stage bound \eqref{eq:admissible-stage-mixed-bound}, together with
\(K_q\geq d(\beta)\), therefore applies. Its crude form,
together with \(\lambda_q\leq\lambda^{\beta/8}\), bounds the square in
\eqref{eq:mixed-commutator-estimate} by
\[
  4
  \lambda^{\frac\beta4(k+s_0+4+2m)}.
\]
Since \(\mu^{-s_0}=\lambda^{-\beta s_0/2}\), the resulting exponent is
\[
  \frac\beta4(k+2m+4-s_0)
  =-\frac\beta4\left\lceil\frac2\beta\right\rceil
  \leq-\frac12.
\]
Thus increasing the threshold absorbs only \(4C_{m,k,s_0}\) and proves
\eqref{eq:commutator-mixed-smallness}.
The comparison \(3/\beta-31<4/\beta^2\) holds for
\(\beta\leq1/8\); the stronger assumption \(\beta\leq1/40\) is used here
only to ensure \(m_{q+1}\geq0\).
\end{proof}

Set
\begin{equation}
  G_{q+1}
  :=\mathcal C_{\mu_{q+1}}(\rho^{(q)})+E_{q+1}^{\rm pert}.
  \label{eq:stage-force-increment}
\end{equation}
The new stage is
\begin{equation}
  \boxed{\quad
  \rho^{(q+1)}
  :=P_{q+1}+\theta_{q+1}
  =S_{\mu_{q+1}}\rho^{(q)}+\theta_{q+1}.
  \quad}
  \label{eq:full-stage-recursion}
\end{equation}
Its force is defined by
\begin{equation}
  \boxed{\quad
  F^{(q+1)}
  :=S_{\mu_{q+1}}F^{(q)}+G_{q+1}.
  \quad}
  \label{eq:forcing-recursion}
\end{equation}
Indeed, the stage-\(q\) equation and
\eqref{eq:periodic-cutoff-commutator-definition} imply
\[
  \partial_tP_{q+1}
  +u_{\mathbb T}(P_{q+1})\cdot\nabla P_{q+1}
  =S_{\mu_{q+1}}F^{(q)}
    +\mathcal C_{\mu_{q+1}}(\rho^{(q)}).
\]
Adding \eqref{eq:exact-packet-residual} proves
\begin{equation}
  \partial_t\rho^{(q+1)}
  +u_{\mathbb T}(\rho^{(q+1)})\cdot\nabla\rho^{(q+1)}
  =F^{(q+1)},
  \label{eq:finite-stage-equation}
\end{equation}
with every cross term accounted for.

\subsection{The one-step result}

\begin{lemma}[Activation before the terminal interval and terminal normalization]
\label{lem:activation-before-terminal-normalization}
Assume the adjacent-layer setup in
Definition~\ref{def:adjacent-layer-setup}. Let \(s_{q+1}\) and
\(a_{q+1}\) be
supplied by Lemmas~\ref{lem:perpendicularity-crossing} and
\ref{lem:amplitude-gain}.  Set
\(p_{\rm term}=e(\pi/2+\lambda^{-\beta/8})\), and let \(X_{q+1}\) be defined by
Lemma~\ref{lem:transported-child-phase}.  There is a finite threshold
\[
  \Lambda_{\rm activation}=\Lambda_{\rm activation}(b,\beta)
\]
such that, whenever \(\Lambda\geq\Lambda_{\rm activation}\),
\begin{equation}
  s_{q+1}+L^{-8}
  <\underline t_{q+1}
  =1-\widetilde C\lambda^{-3\beta/4}.
  \label{eq:activation-before-child-window}
\end{equation}
Moreover, uniformly for \(t\in I_{q+1}\),
\begin{equation}
  |a_{q+1}(t)-1|
  +|\nabla X_{q+1}(0,t)-p_{\rm term}|
  \leq C L^C\lambda^{-3\beta/4}=o(1).
  \label{eq:terminal-amplitude-phase-normalization}
\end{equation}
\end{lemma}

\begin{proof}
Write \(s:=s_{q+1}\), \(a:=a_{q+1}\), and \(X:=X_{q+1}\).
The lower bound in \eqref{eq:perpendicularity-window} gives
\[
  1-s\geq\frac1{6C_0}\Lambda^{-3b/4}.
\]
Equation~\eqref{eq:next-log-frequency} gives
\[
  L=k_-\beta\Lambda^{b/8},
  \qquad
  L^{-8}=(k_-\beta)^{-8}\Lambda^{-b}.
\]
Consequently \(L^{-8}=o(\Lambda^{-3b/4})\), while
\[
  \lambda^{-3\beta/4}
  =\exp\left(-\frac{3\beta L}{4}\right)
  =o(\Lambda^{-3b/4}).
\]
This proves \eqref{eq:activation-before-child-window} above a finite threshold.

On \(I_{q+1}\), Lemma~\ref{lem:mixed-parent-inputs} gives
\[
  |c_P(0,t)|+\|D_xv(t)\|_{C^0}\leq CL^C.
\]
Since \(a(1)=1\), integration of \(\dot a/a=-c_P(0,t)\) yields
\[
  |\log a(t)|\leq CL^C(1-t)
  \leq CL^C\lambda^{-3\beta/4}.
\]
The phase jet solves
\[
  \partial_t\nabla X(0,t)
  =-D_xv(0,t)^{\mathsf T}\nabla X(0,t),
  \qquad
  \nabla X(0,1)=p_{\rm term}.
\]
A short-time Gr\"onwall estimate proves
\eqref{eq:terminal-amplitude-phase-normalization}.
\end{proof}

\begin{proposition}[One-step extension]
\label{prop:one-step}
Let an admissible family be given through level \(q\). Choose
\[
  \beta_{q+1}\in
  \left[\frac{\beta_q}{2},\beta_q\right]
\]
and an integer \(k_{q+1}\geq d(\beta_{q+1})\). Assume
\(K_q\geq d(\beta_{q+1})\), and define \(\lambda_{q+1}\) by
\eqref{eq:deterministic-frequency-law}. For the fields defined by
\eqref{eq:full-stage-recursion}, let \(\mathcal B_{q+1}(J)\) denote the
property
\[
  \sup_{t\in[0,1]}
  \|\partial_t^m\rho^{(q+1)}(t)\|_{C^j}
  \leq2\lambda_{q+1}^{j+2m}
  \quad\text{for all }j,m\in\mathbb N_0
  \text{ with }1\leq j+2m\leq J,
\]
and define
\begin{equation}
  K_{q+1}:=\max\left(
    \{0\}\cup
    \{J\in\mathbb N_0:J\leq k_{q+1},\quad
      \mathcal B_{q+1}(J)\}\right).
  \label{eq:maximal-cumulative-mixed-order}
\end{equation}
There is a finite threshold
\[
  \lambda_{\mathrm{th},\mathbb T}
  =\lambda_{\mathrm{th},\mathbb T}
    (\beta_q,\beta_{q+1},k_{q+1};f,\chi,W).
\]
Choose it to dominate
\(\Lambda_{\rm sep},\Lambda_{\rm cas},\Lambda_{\rm activation}\), and, when
\(\beta_{q+1}\leq1/40\), \(\Lambda_{\rm res},\Lambda_{\rm com}\), together
with the finitely many absorptions for \(j+2m\leq k_{q+1}\). The first
group is independent of derivative order; only the last finite group depends
on \(k_{q+1}\). Whenever
\(\lambda_q\geq\lambda_{\mathrm{th},\mathbb T}\), the fields in
\eqref{eq:full-stage-recursion} and \eqref{eq:forcing-recursion} extend the
admissible family through level \(q+1\). The activation interval ends
before \(I_{q+1}\) by \eqref{eq:activation-before-child-window}, and the
new perturbation and its exact residual are smooth and vanish to infinite
order in time at \(s_{q+1}\). The following additional estimates hold
uniformly for \(t\in[0,1]\).

\begin{enumerate}[label=\textup{(\roman*)}]
  \item For every fixed \(j,m\geq0\), the a priori constants
  \(C_{j,m}(\beta_{q+1})\) and \(N_{j,m}(\beta_{q+1})\) fixed in
  Definition~\ref{def:admissible-layered-family} and
  \eqref{eq:sharp-constant-envelope} give exactly
  \begin{equation}
    \|\partial_t^m\theta_{q+1}(t)\|_{C^j}
    \leq C_{j,m}(\beta_{q+1})
      L_{q+1}^{N_{j,m}(\beta_{q+1})}
      \lambda_{q+1}^{j+m(1-2\beta_{q+1})-1+\beta_{q+1}}.
    \label{eq:mixed-perturbation-production}
  \end{equation}
  These fixed-order estimates require no enlargement of
  \(\lambda_{\mathrm{th},\mathbb T}\). The prescribed family already takes
  the maximum over the possible parent exponents
  \(\beta_q\in[\beta_{q+1},2\beta_{q+1}]\), so its constants depend only on
  \(\beta_{q+1}\). For every pair with
  \(j+2m\leq k_{q+1}\),
  \begin{equation}
    \|\partial_t^m\theta_{q+1}(t)\|_{C^j}
    \leq\lambda_{q+1}^{j+2m}
    \qquad(j+2m\leq k_{q+1}).
    \label{eq:mixed-perturbation-finite-order-bound}
  \end{equation}
  In particular, for \(0\leq j\leq k_{q+1}\),
  \begin{equation}
    \|\theta_{q+1}(t)\|_{C^j}\leq\lambda_{q+1}^j,
    \qquad
    \|\theta_{q+1}(t)\|_{C^0}
       \leq\lambda_{q+1}^{-1+2\beta_{q+1}},
    \qquad
    \|\theta_{q+1}(t)\|_{C^1}
       \leq\lambda_{q+1}^{2\beta_{q+1}}.
    \label{eq:new-perturbation-size}
  \end{equation}

  \item If \(\beta_{q+1}\leq1/40\), define
  \(
    m_{q+1}:=\lfloor(4\beta_{q+1})^{-1}\rfloor-10.
  \)
  The force increment satisfies
  \begin{equation}
    \|\partial_t^mG_{q+1}(t)\|_{C^k}
    \leq\lambda_{q+1}^{-1/4}
    \qquad(m+k\leq m_{q+1}).
    \label{eq:force-increment-smallness}
  \end{equation}

  \item One has
  \begin{equation}
    K_{q+1}\geq\min\{K_q,k_{q+1}\}
    \geq d(\beta_{q+1}).
    \label{eq:cumulative-order-propagation}
  \end{equation}
\end{enumerate}
\end{proposition}

\begin{proof}
We verify the conclusions in the order in which their ingredients enter the
construction.  The exact cascade identity and the cutoff-commutator algebra
give \eqref{eq:finite-stage-equation}. The seed support, the flow inclusion
\eqref{eq:one-step-core-flow}, the infinite-order vanishing of the activation
cutoff at \(s_{q+1}\), and the parity statements in
Lemma~\ref{lem:correction-estimates} show that the new fields are smooth,
odd, and mean zero, and give the support and infinite-order-vanishing
assertions required by Definition~\ref{def:admissible-layered-family}.

The identity \(S_\mu P_q=P_q\) leaves the level-\(q\) reference-angle law
\eqref{eq:admissible-angle-law} exact.  The newest-perturbation cutoff-tail
estimate \eqref{eq:parent-newest-spectral-tail} perturbs by only
\(O(\lambda_q^{-4})\) the parent gradient and velocity-Jacobian geometry entering
the child equation. Lemma~\ref{lem:perpendicularity-crossing} gives
\(s_{q+1}\) and transversality; Lemma~\ref{lem:amplitude-gain} and
Corollary~\ref{cor:deterministic-seed-smallness} give the no-collapse and seed
bounds, including \eqref{eq:seed-amplitude-smallness}. The perturbation-flow
bounds come from Lemmas~\ref{lem:cutoff-parent-flow} and
\ref{lem:transported-child-phase}.

Proposition~\ref{prop:localized-periodic-velocity-expansion} and Lemma
\ref{lem:correction-estimates}, together with
Corollary~\ref{lem:mixed-weighted-propagator},
organized by the within-child induction in
Lemma~\ref{lem:density-first-mixed-closure}, give
\eqref{eq:mixed-perturbation-production} with exactly the a priori family.
Indeed, at time order zero the cutoff-parent input comes from the
constant-free \(j\leq2,m=0\) stage bound and Bernstein's inequality. At time
order \(m\geq1\), the coefficient pass uses parent and defect time orders at
most \(m-1\). Higher cutoff-parent spatial orders again follow from
Bernstein's inequality, and the flow and phase thresholds are independent of
their derivative order. Thus the raw constants
\(\widehat C_{j,m}(b,\beta),\widehat N_{j,m}(b,\beta)\) generated at
\((j,m)\) use only the lower-time constants allowed in
\eqref{eq:sharp-constant-envelope} and are bounded by the prescribed
\(C_{j,m}(\beta),N_{j,m}(\beta)\). No order-dependent constant is
absorbed into a frequency gain or enlarges the all-order threshold.
At time order zero this proves \eqref{eq:new-perturbation-size} after increasing
the finite threshold. More generally, if \(j+2m\leq k_{q+1}\), then
\[
  C_{j,m}(\beta)L^{N_{j,m}(\beta)}
  \lambda^{j+m(1-2\beta)-1+\beta}
  \leq\frac12\lambda^{j+2m}
\]
above one finite threshold, simultaneously for all such pairs.  This proves
\eqref{eq:mixed-perturbation-finite-order-bound} and part~(i).  Equations
\eqref{eq:packet-residual-mixed-smallness}, proved in
Lemma~\ref{lem:packet-residual-estimate}, and
\eqref{eq:commutator-mixed-smallness} give
\eqref{eq:force-increment-smallness}. For these applications, the enlarged
cascade orders are supplied by Lemma~\ref{lem:mixed-parent-inputs} and
Lemmas~\ref{lem:cutoff-parent-flow}--\ref{lem:transported-child-phase}; the
final grouped-residual orders are supplied by
Corollary~\ref{cor:finite-packet-inputs}, while the commutator's uncut
derivatives come directly from
\eqref{eq:admissible-stage-mixed-bound}, as checked in the final subsection
of Appendix~\ref{sec:smooth-force}. This proves part~(ii).

Lemma~\ref{lem:activation-before-terminal-normalization} shows that the
activation cutoff equals one throughout \(I_{q+1}\). Oddness gives
\(X_{q+1}(0,t)=0\), and
\eqref{eq:origin-weight} gives
\begin{equation}
  \nabla\zeta_1(0,t)
  =a_{q+1}(t)\lambda_{q+1}^{\beta_{q+1}}
    \nabla X_{q+1}(0,t).
  \label{eq:principal-origin-gradient}
\end{equation}
Set
\begin{equation}
  A_{q+1}(t)=B_{q+1}(t)
  :=a_{q+1}(t)\lambda_{q+1}^{\beta_{q+1}}
    |\nabla X_{q+1}(0,t)|.
  \label{eq:terminal-principal-amplitudes}
\end{equation}
Writing
\(\nabla X_{q+1}(0,t)/|\nabla X_{q+1}(0,t)|
=e(\alpha_{q+1}(t))\), the leading
origin-jet identity gives
\begin{equation}
  D_xu^{[0]}[\zeta_1](0,t)
  =C_0B_{q+1}(t)\cos\alpha_{q+1}(t)
    \mathcal M(\alpha_{q+1}(t)).
  \label{eq:principal-origin-velocity-jet}
\end{equation}
The terminal normalization places both principal amplitudes in
\([\frac34\lambda_{q+1}^{\beta_{q+1}},
  \frac54\lambda_{q+1}^{\beta_{q+1}}]\) above a finite
threshold.

For \(j\geq2\), the correction estimate and one oscillatory-factor derivative
give
\[
  |\nabla\zeta_j(0,t)|
  +\|D_xu^{[M]}[\zeta_j](0,t)\|
  \leq C_\beta L^C\lambda^{(2-j)\beta}
  \leq C_\beta L^C.
\]
For the principal correction, every nonleading expansion term gains
\(\lambda^{-\beta}\):
\[
  \|D_x(u^{[M]}-u^{[0]})[\zeta_1](0,t)\|
  \leq C_\beta L^C\lambda^{-\beta}.
\]
For every \(1\leq j\leq M\), the grouped exact velocity remainder satisfies
\[
  \|D_x(u_{\mathbb T}-u^{[M]})[\zeta_j]\|_{C^0}
  \leq C_\beta L^C
    \lambda^{3-(j+2)\beta-M\beta}
  \leq C_\beta L^C\lambda^{-5}.
\]
Thus the sum of the higher-correction density jets and all nonleading
velocity jets is bounded by \(C_\beta L^C\). A final threshold makes this
at most the errors required in
\eqref{eq:admissible-gradient-geometry}--\eqref{eq:admissible-jacobian-geometry},
proving the complete origin geometry.

Finally, the \(j=1,2\), \(m=0\) cases of
\eqref{eq:admissible-stage-mixed-bound} at the preceding stage give
\[
  \|P_{q+1}\|_{C^1}
  +\|u_{\mathbb T}(P_{q+1})\|_{C^1}
  \leq C\lambda_q^2.
\]
Equation~\eqref{eq:deterministic-frequency-law} makes this fixed power at most
\(\lambda_{q+1}^{\beta_{q+1}/8}\) above a finite threshold. This is the
cutoff-parent bound in \eqref{eq:admissible-cutoff-parent-bound}.

Parent admissibility and the terminal-time separation give
\[
  \widetilde C\lambda_{q+1}^{-3\beta_{q+1}/4}
  \leq\frac12\widetilde C\lambda_q^{-3\beta_q/4}
  \leq\frac12.
\]
Thus \eqref{eq:admissible-terminal-window-domain} propagates to the child,
so its terminal window also lies in \([0,1]\).

Put \(J_*=\min\{K_q,k_{q+1}\}\). If
\(1\leq j+2m\leq J_*\), cutoff boundedness and
\eqref{eq:admissible-stage-mixed-bound} at level \(q\) give
\[
  \|S_\mu\partial_t^m\rho^{(q)}\|_{C^j}
  \leq2C_j\Lambda^{j+2m}.
\]
There are finitely many such pairs, so frequency separation makes this at
most \(\frac12\lambda^{j+2m}\) simultaneously for all of them. The
new-perturbation bound above supplies the other half. Hence
\eqref{eq:cumulative-order-propagation} follows from
\eqref{eq:full-stage-recursion}, proving part~(iii).  The scale selection,
angle law, activation-interval separation, derivative bounds, and origin
geometry verified
above establish every remaining clause of the admissible-family definition.
\end{proof}

\section{Global iteration and blow-up}
\label{sec:global-iteration}

Retain the fixed cutoff \(\chi\), spatial profile \(f\), and activation
profile \(W\) from Sections~\ref{sec:velocity-and-jets} and
\ref{sec:one-layer}. The initial base oscillation parameter will be chosen so
that every perturbation support is contained in the coordinate chart of
Proposition~\ref{prop:localized-periodic-velocity-expansion}.
The all-order mixed estimates used in the iteration are proved in
Appendix~\ref{sec:smooth-force}.

\subsection{Initialization and the recursive exponent choice}
\label{subsec:global-parameters}

Set
\begin{equation}
  \beta_0:=\frac18,
  \qquad
  k_0=K_0:=257=d(\beta_0),
  \qquad
  \alpha_0:=\frac\pi2+\lambda_0^{-1/64}.
  \label{eq:base-parameters}
\end{equation}
For sufficiently large \(\lambda_0\), define in the central chart
\begin{equation}
  \theta_0(x)
  :=\lambda_0^{-7/8}f(\lambda_0^{1/4}x)
    \sin\bigl(\lambda_0e(\alpha_0)\cdot x\bigr),
  \qquad
  P_0:=0,
  \qquad
  \rho^{(0)}:=\theta_0,
  \label{eq:periodic-base-density}
\end{equation}
and extend \(\theta_0\) periodically by zero. Put
\begin{equation}
  F^{(0)}
  :=u_{\mathbb T}(\theta_0)\cdot\nabla\theta_0.
  \label{eq:periodic-base-force}
\end{equation}

\begin{lemma}[Periodic base layer]
\label{lem:periodic-base-layer}
There is a finite threshold \(\lambda_*(f,\chi)\) such that, whenever
\(\lambda_0\geq\lambda_*(f,\chi)\), the fields in
\eqref{eq:periodic-base-density}--\eqref{eq:periodic-base-force} form an
admissible level-zero stage.
\end{lemma}

\begin{proof}
Increase \(\lambda_*(f,\chi)\), if necessary, so that
\(\lambda_*(f,\chi)\geq e\) and, whenever
\(\lambda_0\geq\lambda_*(f,\chi)\),
\[
  \widetilde C\lambda_0^{-3/32}\leq1.
\]
Since \(\beta_0=1/8\), this is
\eqref{eq:admissible-terminal-window-domain} and gives
\(I_0\subset[0,1]\).
The envelope in \eqref{eq:periodic-base-density} is supported strictly
inside the coordinate chart. Hence its extension by zero is smooth even
though the local phase covector need not belong to the integer lattice. The
density is odd and therefore has zero mean. Its support lies in
\(B_{\lambda_0^{-1/4}}\), and
\[
  \|\theta_0\|_{C^j}
  \leq C_{f,j}\lambda_0^{j-7/8}
  \leq\lambda_0^j
  \qquad(0\leq j\leq257),
\]
above a finite threshold. The pair is stationary and satisfies the
level-zero equation by \eqref{eq:periodic-base-force}.
Moreover, \(F^{(0)}\) is odd because
\(u_{\mathbb T}(\theta_0)\) is odd and \(\nabla\theta_0\) is even.

Since \(f=1\) near the origin,
\[
  \nabla\theta_0(0)=\lambda_0^{1/8}e(\alpha_0).
\]
On \(I_0\), define the witnesses in the admissible origin geometry by
\begin{align*}
  A_0(t)&=B_0(t)=\lambda_0^{1/8},\\
  \alpha_0(t)&\equiv\frac\pi2+\lambda_0^{-1/64},\\
  E_0^\rho(t)&=0,\\
  E_0^u(t)&=
    D_xu_{\mathbb T}(\theta_0)(0)
    -C_0\lambda_0^{1/8}\cos\alpha_0\,\mathcal M(\alpha_0).
\end{align*}
For this application, set
\(P=0\), \(v:=u_{\mathbb T}(P)\), and \(\mathcal G:=\nabla P\), so
\(v=\mathcal G=0\).
Apply Proposition~\ref{prop:localized-periodic-velocity-expansion} with
this \(P\), the linear phase
\(e(\alpha_0)\cdot x\), and \(\beta_0=1/8\); its hypotheses hold with
\(v=\mathcal G=0\), exactly linear phase on the amplitude support,
\(\operatorname{supp}f(\lambda_0^{1/4}\,\cdot)\Subset B_{r_0}\), and
\(\nu_1=\lambda_0\geq e\), by
Lemma~\ref{lem:transported-phase-estimates} and
Corollary~\ref{cor:oscillatory-product-estimates}.
Its leading Euclidean term gives
the shear matrix in
\eqref{eq:admissible-jacobian-geometry}. The smooth periodic-kernel
contribution is absorbed into the grouped remainder after exactly
\(M_0\) integrations by parts. The nonleading expansion terms, grouped
remainder, and smooth-kernel term contribute entrywise to the Jacobian error
at most
\[
  C_f\lambda_0^{-1/8},\qquad
  C_f\lambda_0^{-43/8},\qquad
  C_f\lambda_0^{-391/8},
\]
respectively.
Thus, above a finite threshold, each of the
four components of \(E_0^u\) has absolute value at most
\(\lambda_0^{1/64}\). Profile scaling
gives the support condition and the sharp perturbation estimates, while
stationarity
gives every positive-time mixed bound. The preceding bounds verify the
remaining level-zero clauses of
Definition~\ref{def:admissible-layered-family}; in particular, \(P_0=0\)
gives the required previous-stage \(C^1\) estimates and the constant
reference-angle law.
\end{proof}

The exponent \(\beta_q\) controls the size of the \(q\)-th perturbation and
the number of derivative estimates required at that scale. The integer
\(k_q\) is the largest prescribed weighted mixed-derivative order for that
perturbation; \(K_q\) is the largest weighted mixed-derivative order verified
for the complete stage.
At a fixed exponent we raise \(k_q\) as the base oscillation parameters grow.
Once \(K_q\) is large enough to authorize the smaller exponent and the
corresponding finite threshold has been reached, we replace \(\beta_q\) by
\(\beta_q/2\).

Fix a finite threshold function
\[
  \lambda_{\mathrm{th},\mathbb T}(\beta,\beta',k;f,\chi,W),
  \qquad
  0<\frac\beta2\leq\beta'\leq\beta\leq\frac18,
  \qquad
  k\geq d(\beta'),
\]
which is nondecreasing in \(k\) and dominates every finite requirement in
Proposition~\ref{prop:one-step}. With parent exponent \(\beta\) and child
exponent \(\beta'\), it dominates
\[
  \begin{gathered}
    \Lambda_{\rm sep}(\beta),\quad
    \Lambda_{\rm geom}(\beta,\beta'),\quad
    \Lambda_{\rm flow}(\beta,\beta'),\quad
    \Lambda_{\perp}(\beta,\beta'),\\
    \Lambda_{\rm gain}(\beta,\beta'),\quad
    \Lambda_{\rm seed}(\beta,\beta'),\quad
    \Lambda_{\rm pkt}(\beta,\beta'),\quad
    \Lambda_{\rm cas}(\beta,\beta'),\\
    \Lambda_{\rm res}(\beta,\beta'),\quad
    \Lambda_{\rm com}(\beta,\beta'),\quad
    \Lambda_{\rm activation}(\beta,\beta'),
  \end{gathered}
\]
whenever the named threshold is used, together with the finite scale,
support, profile, and mixed-order absorptions for \(j+2m\leq k\). Each
depends only on \((\beta,\beta',k;f,\chi,W)\). The arbitrary fixed-order
sharp estimates have order-dependent constants but impose no additional
threshold.
Increase \(\lambda_0\), if necessary, so that
\[
  \lambda_0\geq
  \max\left\{
    \lambda_*(f,\chi),
    \lambda_{\mathrm{th},\mathbb T}
      (\beta_0,\beta_0,d(\beta_0);f,\chi,W)
  \right\}.
\]

Suppose that level \(q\) has been constructed and write
\(\beta=\beta_q\). If
\begin{equation}
  K_q\geq d(\beta/2)
  \quad\text{and}\quad
  \lambda_q\geq
  \max\left\{
    \lambda_{\mathrm{th},\mathbb T}
      (\beta,\beta/2,d(\beta/2);f,\chi,W),
    \lambda_{\mathrm{th},\mathbb T}
      (\beta/2,\beta/2,d(\beta/2);f,\chi,W)
  \right\},
  \label{eq:halving-condition}
\end{equation}
choose
\begin{equation}
  \beta_{q+1}:=\frac{\beta_q}{2},
  \qquad
  k_{q+1}:=d(\beta_{q+1}).
  \label{eq:halving-choice}
\end{equation}
Otherwise choose
\begin{equation}
  \beta_{q+1}:=\beta_q,
  \qquad
  k_{q+1}:=
  \max\left\{
    k\in\mathbb N:
    d(\beta_q)\leq k\leq q+d(\beta_q),\
    \lambda_q\geq
      \lambda_{\mathrm{th},\mathbb T}
        (\beta_q,\beta_q,k;f,\chi,W)
  \right\}.
  \label{eq:retention-condition}
\end{equation}
The admissible set in \eqref{eq:retention-condition} is nonempty. At
\(q=0\), the candidate \(d(\beta_0)\) is authorized by the base threshold.
If level \(q\) follows a halving, the second threshold in
\eqref{eq:halving-condition} authorizes \(d(\beta_q)\). If it follows a
retention step, the previously chosen \(k_q\) remains admissible: it lies
below the enlarged cap, while
\[
  \lambda_q\geq\lambda_{q-1}
  \geq\lambda_{\mathrm{th},\mathbb T}
    (\beta_q,\beta_q,k_q;f,\chi,W).
\]
Once
\(\beta_{q+1}\) has been fixed, define \(\lambda_{q+1}\) by
\eqref{eq:deterministic-frequency-law}, apply
Proposition~\ref{prop:one-step}, and define
\(K_{q+1}\) by \eqref{eq:maximal-cumulative-mixed-order}.

\begin{lemma}[Continuation]
\label{lem:global-continuation}
Let the parameter and mixed-derivative-order sequences be generated from
\eqref{eq:base-parameters} by \eqref{eq:halving-condition}--
\eqref{eq:retention-condition}, \eqref{eq:deterministic-frequency-law}, and
\eqref{eq:maximal-cumulative-mixed-order}. The recursion continues for every
\(q\).
On every infinite consecutive set of indices on which \(\beta_q\) is
constant, one has \(k_q\to\infty\).
\end{lemma}

\begin{proof}
At a halving step, condition \eqref{eq:halving-condition} authorizes both that step
and the first retained step at the new exponent. The nonemptiness verified
immediately after \eqref{eq:retention-condition} handles every retention
step. Thus the recursion cannot stop.

Fix \(\bar\beta\) and a derivative order \(k\geq d(\bar\beta)\). The number
\(\lambda_{\mathrm{th},\mathbb T}
(\bar\beta,\bar\beta,k;f,\chi,W)\) is finite. By
\eqref{eq:admissible-frequency-separation},
\(\lambda_{q+1}\geq4\lambda_q\), so \(\lambda_q\to\infty\); the cap
\(q+d(\bar\beta)\) also tends to infinity. Consequently \(k_q\geq k\) at all
sufficiently late indices of a constant-\(\bar\beta\) block.
\end{proof}

\subsection{Iterated cutoff formulas and parameter growth}

The density formula \eqref{eq:iterated-density-formula} was derived directly from
the finite-stage invariant.  The same cutoff induction gives, for every
\(Q\geq1\),
\begin{equation}
  F^{(Q)}
  =S_{\mu_1}F^{(0)}
   +\sum_{i=1}^{Q-1}S_{\mu_{i+1}}G_i
   +G_Q.
  \label{eq:iterated-force-formula}
\end{equation}
Indeed, \(\mu_j\geq2\mu_i\) whenever \(j>i\), so
\[
  S_{\mu_j}S_{\mu_i}=S_{\mu_i}.
\]
Thus every perturbation and every force increment is truncated exactly
once and is then fixed. In particular,
\begin{equation*}
  \operatorname{supp}\widehat{S_{\mu_{i+1}}\theta_i}
  \subset B_{2\mu_{i+1}},
  \qquad
  \operatorname{supp}\widehat{S_{\mu_{i+1}}G_i}
  \subset B_{2\mu_{i+1}}.
\end{equation*}
Moreover, \eqref{eq:commutator-bandwidth} and
\(\mu_{i+1}\geq4\mu_i\) imply
\begin{equation*}
  S_{\mu_{i+1}}
    \mathcal C_{\mu_i}(\rho^{(i-1)})
  =\mathcal C_{\mu_i}(\rho^{(i-1)}).
\end{equation*}

\begin{lemma}[Decay of the scale exponent]
\label{lem:parameter-growth}
For the parameter and mixed-derivative-order sequences generated by
\eqref{eq:halving-condition}--\eqref{eq:retention-condition},
\eqref{eq:deterministic-frequency-law}, and
\eqref{eq:maximal-cumulative-mixed-order}, one has
\begin{equation}
  \beta_q\downarrow0,
  \qquad
  x_q:=\beta_q\log\lambda_q\longrightarrow\infty.
  \label{eq:beta-decay}
\end{equation}
After increasing the base threshold,
\begin{equation}
  x_{q+1}\geq64x_q,
  \qquad
  L_q^4\leq\lambda_q^{\beta_q},
  \qquad
  r_{q+1}\leq\lambda_q^{-8}.
  \label{eq:effective-scale-growth}
\end{equation}
Consequently \(r_q\to0\), \(\underline t_q\uparrow1\), and the intervals
\([\underline t_q,\underline t_{q+1}]\) cover a terminal neighborhood of
one.
\end{lemma}

\begin{proof}
If \(\beta_q\not\to0\), then, because each step either retains \(\beta_q\)
or replaces it by \(\beta_q/2\), only finitely many halvings occur. Hence
\(\beta_q=\bar\beta>0\) for every sufficiently large \(q\).
Since \(\beta_{q+1}=\beta_q\) at every sufficiently late stage, the halving
criterion fails there. Its frequency condition already holds, so its first
condition fails and \(K_q<d(\bar\beta/2)\), hence
\[
  K_q\leq
  D_*:=\left\lceil\frac{16}{\bar\beta^2}\right\rceil,
\]
at all sufficiently late stages. Lemma~\ref{lem:global-continuation} gives
\(k_q>D_*+1\). Put \(n=K_q+1\). By maximality in
\eqref{eq:maximal-cumulative-mixed-order}, there are integers \(j,m\geq0\) with
\(j+2m=n\), and a time \(\tau_q\), such that
\begin{equation}
  \|\partial_t^m\rho^{(q)}(\tau_q)\|_{C^j}
  >2\lambda_q^n.
  \label{eq:maximal-order-failure}
\end{equation}

Choose \(q_0\) so that \(k_i>D_*+1\) for every \(i\geq q_0\). Let
\(C_{\rm head}\) bound the fixed contribution of indices \(i<q_0\), uniformly
over the finitely many possible values of \(n\). The cutoff
operators commute with \(\partial_t\) and have a uniform \(C^j\)-operator
norm for each \(j\leq D_*+1\). Hence \eqref{eq:iterated-density-formula} and the
perturbation mixed estimates
\eqref{eq:admissible-perturbation-mixed-bound} give
\[
  \|\partial_t^m\rho^{(q)}(\tau_q)\|_{C^j}
  \leq C_{\rm head}
    +C_{D_*}\sum_{i=q_0}^{q-1}\lambda_i^n+\lambda_q^n.
\]
On a constant-\(\bar\beta\) block, the deterministic frequency law implies
\(\lambda_q/\lambda_{q-1}\to\infty\). Since \(n\) ranges over the finite
set \(\{1,\ldots,D_*+1\}\),
\[
  C_{\rm head}
  +C_{D_*}\sum_{i=q_0}^{q-1}\lambda_i^n
  =o(\lambda_q^n),
\]
uniformly in the possible \(n\). This contradicts
\eqref{eq:maximal-order-failure}. Therefore \(\beta_q\downarrow0\).

The selected frequency separation gives
\[
  \lambda_{q+1}^{\beta_{q+1}/8}\geq4\lambda_q.
\]
Taking logarithms and using \(\beta_q\leq1/8\) yields
\[
  x_{q+1}\geq8\log\lambda_q
  =\frac8{\beta_q}x_q
  \geq64x_q.
\]
Equation~\eqref{eq:deterministic-frequency-law} also gives
\[
  L_{q+1}
  =\frac{\beta_{q+1}}{192\pi}
    \exp(x_q/8),
  \qquad
  \log L_{q+1}\leq\frac{x_q}{8}.
\]
After arranging the base case,
\(4\log L_{q+1}\leq x_{q+1}\), which proves
\(L_{q+1}^4\leq\lambda_{q+1}^{\beta_{q+1}}\). Therefore
\[
  r_{q+1}
  =\lambda_{q+1}^{-2\beta_{q+1}}L_{q+1}^4
  \leq\lambda_{q+1}^{-\beta_{q+1}}
  \leq\lambda_q^{-8}.
\]
The remaining conclusions follow from \eqref{eq:terminal-window} and the
neighboring time-scale separation.
\end{proof}

\subsection{The limiting density and force}
\label{subsec:limit-solution}

Define
\begin{equation}
  \rho
  :=S_{\mu_1}\theta_0
    +\sum_{i=1}^{\infty}S_{\mu_{i+1}}\theta_i,
  \qquad
  F
  :=S_{\mu_1}F^{(0)}
    +\sum_{i=1}^{\infty}S_{\mu_{i+1}}G_i.
  \label{eq:limiting-series}
\end{equation}

\begin{proposition}[Global limit]
\label{prop:global-limit}
For the admissible stages produced by the recursive exponent choice, let
\(\rho,F\) be
defined by \eqref{eq:limiting-series}. Then
\(F\in C^\infty([0,1]\times\mathbb T^2)\). For every \(T<1\), the sequences
\((\rho^{(q)})_q\) and \((F^{(q)})_q\) are eventually constant in the stage
index on \([0,T]\), the limiting density is smooth there, and
\begin{equation}
  \partial_t\rho
  +u_{\mathbb T}(\rho)\cdot\nabla\rho
  =F
  \quad\text{on }[0,T]\times\mathbb T^2.
  \label{eq:global-limit-equation}
\end{equation}
The initial datum is
\begin{equation}
  \rho_{\rm in}
  =\rho(0)
  =S_{\mu_1}\theta_0.
  \label{eq:limiting-initial-datum}
\end{equation}
The function \(\rho_{\rm in}\) is smooth, odd, and has zero spatial mean.
For every \(t\in[0,1]\), the function \(F(t,\cdot)\) is odd and has zero
spatial mean.
There is one function
\[
  \rho_*\in\bigcap_{0\leq\eta<1}C^\eta(\mathbb T^2)
\]
such that, for every \(0\leq\eta<1\),
\begin{equation}
  \rho(t)\longrightarrow\rho_*
  \quad\text{in }C^\eta(\mathbb T^2)
  \quad\text{as }t\uparrow1.
  \label{eq:endpoint-convergence}
\end{equation}
\end{proposition}

\begin{proof}
For \(\eta=0\), cutoff boundedness follows from
\eqref{eq:periodic-cutoff-boundedness}; for \(0<\eta<1\), it follows from
\eqref{eq:periodic-cutoff-holder-boundedness}. Interpolation
between the two sharp estimates in \eqref{eq:new-perturbation-size} gives
\begin{equation*}
  \|S_{\mu_{i+1}}\theta_i\|_{C^\eta}
  \leq C_\eta\lambda_i^{-1+2\beta_i+\eta}.
\end{equation*}
By \eqref{eq:beta-decay}, the exponent on the right is at most
\(-(1-\eta)/2\) for every sufficiently large \(i\). The first series in
\eqref{eq:limiting-series} therefore converges uniformly in \(C^\eta\),
including at \(t=1\), because
\(\lambda_{i+1}\geq4\lambda_i\) by
\eqref{eq:admissible-frequency-separation}. Set \(\rho_*:=\rho(1)\).
Uniform smallness of the series tails and continuity of every finite head
prove \eqref{eq:endpoint-convergence}; this definition of \(\rho_*\) is
independent of \(\eta\).

Fix \(m,k\geq0\). Since \(\beta_i\to0\), choose \(i_0=i_0(m,k)\) so that
for every \(i\geq i_0\),
\[
  \beta_i\leq\frac1{40},
  \qquad
  m+k\leq m_i:=\left\lfloor\frac1{4\beta_i}\right\rfloor-10.
\]
By admissibility and \eqref{eq:forcing-recursion}, the base term and the
finitely many summands with \(i<i_0\), including every term formed while
\(\beta_i>1/40\), are smooth and form a finite head of
\eqref{eq:limiting-series}. On the remaining tail, Proposition
\ref{prop:one-step} and \eqref{eq:force-increment-smallness} give
\[
  \|\partial_t^mG_i\|_{C^k}\leq\lambda_i^{-1/4},
\]
throughout this tail. The second series in
\eqref{eq:limiting-series} therefore converges in \(C_t^mC_x^k\): the
cutoffs are uniformly bounded on every fixed \(C^k\), commute with time
derivatives, and \(\lambda_{i+1}\geq4\lambda_i\) makes
\(\sum_i\lambda_i^{-1/4}\) finite.
Since \(m,k\) are arbitrary,
\[
  F\in C^\infty([0,1]\times\mathbb T^2).
\]

Fix \(T<1\). Choose \(q\) sufficiently large that every perturbation with
index at least \(q-1\) vanishes on \([0,T]\). Then
\(\rho^{(q-1)}\) is already Fourier truncated to \(B_{2\mu_{q-1}}\), and
\eqref{eq:periodic-product-bandwidth} puts its nonlinearity in
\(B_{4\mu_{q-1}}\). Consequently,
\begin{equation*}
  S_{\mu_q}\rho^{(q-1)}=\rho^{(q-1)},
  \qquad
  \mathcal C_{\mu_q}(\rho^{(q-1)})=0
  \quad\text{on }[0,T].
\end{equation*}
The perturbation residual \(E_q^{\rm pert}\) also vanishes there.

The newest untruncated term of \(F^{(q-1)}\) is then a cutoff commutator at
scale \(\mu_{q-1}\), and \eqref{eq:commutator-bandwidth} places it in
\(B_{4\mu_{q-1}}\). Every older force piece already has a lower Fourier
cutoff. Since \(\mu_q\geq4\mu_{q-1}\),
\begin{equation*}
  S_{\mu_q}F^{(q-1)}=F^{(q-1)}
  \quad\text{on }[0,T].
\end{equation*}
Equations \eqref{eq:full-stage-recursion} and
\eqref{eq:forcing-recursion} now show that
\[
  \rho^{(q)}=\rho^{(q-1)},
  \qquad
  F^{(q)}=F^{(q-1)}
  \quad\text{on }[0,T].
\]
These identities propagate by induction. Suppose \(n\geq q\) and
\(\rho^{(n)}=\rho^{(q-1)}\),
\(F^{(n)}=F^{(q-1)}\) on \([0,T]\). The common density and force have
Fourier supports in \(B_{2\mu_{q-1}}\) and \(B_{4\mu_{q-1}}\),
respectively. Since
\(\mu_{n+1}\geq\mu_q\geq4\mu_{q-1}\), the next cutoff fixes both,
\(\mathcal C_{\mu_{n+1}}(\rho^{(n)})=0\), and
\(\theta_{n+1}=E_{n+1}^{\rm pert}=0\) on \([0,T]\). The two stage
recursions therefore give the same identities at level \(n+1\). Thus both
sequences are eventually constant in the stage index on \([0,T]\).
The corresponding partial sums in \eqref{eq:limiting-series} are
\(\rho^{(Q)}-\theta_Q\) and \(F^{(Q)}-G_Q\). For sufficiently large \(Q\),
both \(\theta_Q\) and \(G_Q\) vanish on \([0,T]\), so the two series agree
there with the eventual stages. The finite-stage equation then gives
\eqref{eq:global-limit-equation} without passage through a nonlinear limit.

Every perturbation after the base layer vanishes at \(t=0\), which proves
\eqref{eq:limiting-initial-datum}. Each \(\theta_i\), \(F^{(0)}\), and
\(G_i\) is odd, and the even cutoffs preserve oddness. Uniform convergence
therefore preserves oddness in both limiting series; every odd periodic
scalar has zero spatial mean. The first cutoff also makes \(\rho_{\rm in}\)
smooth.

At time zero,
\(E_1^{\rm pert}=0\), and
\[
  S_{\mu_1}F^{(0)}
  +\mathcal C_{\mu_1}(\theta_0)
  =u_{\mathbb T}(S_{\mu_1}\theta_0)
   \cdot\nabla S_{\mu_1}\theta_0.
\]
The commutator is supported in \(B_{4\mu_1}\), so
\(S_{\mu_2}\mathcal C_{\mu_1}(\theta_0)
 =\mathcal C_{\mu_1}(\theta_0)\).
For \(i\geq2\), the density \(\rho^{(i-1)}(0)\) is Fourier truncated to
\(B_{2\mu_{i-1}}\), while \(\mu_i\geq4\mu_{i-1}\). Thus the Fourier truncated
commutator identity \eqref{eq:bandlimited-commutator-zero} gives
\(\mathcal C_{\mu_i}(\rho^{(i-1)})(0)=0\); moreover,
\(E_i^{\rm pert}(0)=0\). Hence \(G_i(0)=0\) for every \(i\geq2\). Every
perturbation with positive index and every force increment \(G_i\) with
\(i\geq2\) vanishes on a neighborhood of \(t=0\), so
\(\partial_t\rho(0)=0\) and
\[
  F(0)=u_{\mathbb T}(\rho_{\rm in})\cdot\nabla\rho_{\rm in}.
\]
Thus the limiting equation also holds at the initial time.
\end{proof}

\subsection{Density-gradient lower bound on terminal intervals}
\label{subsec:full-limit-blowup}

For
\[
  t\in[\underline t_q,\underline t_{q+1}],
\]
every perturbation above \(q+1\) vanishes.  The iterated cutoff formula
therefore gives
\begin{equation}
  \rho(t)
  =P_q(t)
   +S_{\mu_{q+1}}\theta_q(t)
   +S_{\mu_{q+2}}\theta_{q+1}(t).
  \label{eq:terminal-interval-decomposition}
\end{equation}
Set
\begin{equation}
  y_q(t):=2r_{q+1}e(\alpha_q(t)).
  \label{eq:terminal-evaluation-point}
\end{equation}
Since \(\operatorname{supp}\theta_{q+1}(t)\subset B_{r_{q+1}}\) and
\(r_{q+1}\leq\lambda_q^{-8}\), this point lies outside the support of
\(\theta_{q+1}\),
but remains close enough to the origin to retain the level-\(q\) geometry.
For the density estimate only the cutoff tail of \(\theta_{q+1}\) remains;
for the velocity estimate, \(u_{\mathbb T}(\theta_{q+1})\) is nonlocal and
must also be bounded.
All constants in the cutoff-tail estimates below are absolute after the
fixed cutoff is chosen; in particular, they are independent of \(q\),
\(\beta_q\), and \(\beta_{q+1}\).

\begin{proposition}[Density-gradient lower bound]
\label{prop:density-gradient-blowup}
For the limiting density \(\rho\) produced by the recursive exponent choice,
there is
\(q_*\in\mathbb N\) such that, for every \(q\geq q_*\) and every
\(t\in[\underline t_q,\underline t_{q+1}]\),
\begin{equation}
  \|\nabla\rho(t)\|_{L^\infty}
  \geq\frac18\lambda_q^{\beta_q}.
  \label{eq:density-terminal-lower-bound}
\end{equation}
Consequently,
\begin{equation}
  \lim_{t\uparrow1}\|\nabla\rho(t)\|_{L^\infty}=\infty.
  \label{eq:density-gradient-blowup}
\end{equation}
\end{proposition}

\begin{proof}
For large \(q\), the point \(y_q(t)\) lies in the fixed coordinate chart, and
the child \(\theta_{q+1}\) vanishes in a neighborhood of it.

Apply \eqref{eq:periodic-spectral-tail} with \(j=1\), \(s=3\), the
perturbation \(C^7\) estimate
\eqref{eq:admissible-perturbation-mixed-bound}, and
\(\mu_{q+2}\geq\lambda_{q+1}^4\). This gives
\begin{equation}
  \|\nabla(S_{\mu_{q+2}}-I)\theta_{q+1}\|_{L^\infty}
  \leq C\lambda_{q+1}^{-5}.
  \label{eq:child-density-tail}
\end{equation}
The cutoff-tail estimate \eqref{eq:newest-perturbation-cutoff-tail} gives
\begin{equation*}
  \|\nabla(S_{\mu_{q+1}}-I)\theta_q\|_{L^\infty}
  \leq C\lambda_q^{-4}.
\end{equation*}
By the mean-value theorem, \eqref{eq:terminal-evaluation-point}, the \(C^2\)
case of \eqref{eq:new-perturbation-size}, and
\eqref{eq:effective-scale-growth},
\begin{equation}
  |\partial_{e(\alpha_q(t))}\theta_q(y_q(t),t)
    -\partial_{e(\alpha_q(t))}\theta_q(0,t)|
  \leq2r_{q+1}\|\theta_q(t)\|_{C^2}
  \leq2\lambda_q^{-6}.
  \label{eq:newest-density-spatial-change}
\end{equation}
The cutoff-parent bound \eqref{eq:admissible-cutoff-parent-bound} gives
\[
  \|P_q(t)\|_{C^1}\leq\lambda_q^{\beta_q/8}.
\]
Combining \eqref{eq:terminal-interval-decomposition}, the three preceding error
estimates, and \eqref{eq:admissible-gradient-geometry}, we obtain
\[
  |\partial_{e(\alpha_q(t))}\rho(y_q(t),t)|
  \geq
  \frac12\lambda_q^{\beta_q}
  -C\lambda_q^{\beta_q/8}
  -C\lambda_q^{-4}
  -C\lambda_q^{-6}
  -C\lambda_{q+1}^{-5}.
\]
This is at least \(\lambda_q^{\beta_q}/8\) for all sufficiently large
\(q\), uniformly throughout the interval. Lemma~\ref{lem:parameter-growth}
proves
\eqref{eq:density-gradient-blowup}.
\end{proof}

\subsection{Velocity-gradient lower bound on terminal intervals}

\begin{proposition}[Velocity-gradient lower bound]
\label{prop:velocity-gradient-blowup}
For the limiting density \(\rho\) produced by the recursive exponent choice,
there are
\(c>0\) and \(q_*\in\mathbb N\) such that, for every \(q\geq q_*\) and every
\(t\in[\underline t_q,\underline t_{q+1}]\),
\begin{equation}
  \|D_xu_{\mathbb T}(\rho(t))\|_{L^\infty}
  \geq c\lambda_q^{7\beta_q/8}.
  \label{eq:velocity-terminal-lower-bound}
\end{equation}
Consequently,
\begin{equation}
  \lim_{t\uparrow1}
  \|D_xu_{\mathbb T}(\rho(t))\|_{L^\infty}
  =\infty.
  \label{eq:velocity-gradient-blowup}
\end{equation}
\end{proposition}

\begin{proof}
Since
\[
  \mathcal M(\alpha)
  =-e^\perp(\alpha)e(\alpha)^{\mathsf T},
\]
its entry in the first row and second column is \(\sin^2\alpha\). The angle
bounds in
Lemma~\ref{lem:cutoff-parent-geometry}, the bound
\(B_q(t)\geq\frac12\lambda_q^{\beta_q}\) from
Definition~\ref{def:admissible-layered-family}, and
\eqref{eq:admissible-jacobian-geometry} give
\[
  C_0B_q|\cos\alpha_q|\sin^2\alpha_q
  \geq\frac{C_0}{4}(1-o(1))\lambda_q^{7\beta_q/8},
\]
uniformly in the terminal interval. Retaining fixed slack, we obtain
\begin{equation}
  |\partial_2u_{\mathbb T,1}(\theta_q)(0,t)|
  \geq
  \frac{C_0}{8\pi}\lambda_q^{7\beta_q/8}
  -\lambda_q^{\beta_q/8}.
  \label{eq:newest-velocity-entry}
\end{equation}
The mean-value theorem, \eqref{eq:periodic-jacobian-bound} applied to each
\(\partial_a\theta_q\), the \(C^3\) case of
\eqref{eq:new-perturbation-size}, and
\eqref{eq:effective-scale-growth} give
\begin{equation*}
  |D_xu_{\mathbb T}(\theta_q)(y_q(t),t)
       -D_xu_{\mathbb T}(\theta_q)(0,t)|
  \leq Cr_{q+1}\lambda_q^3
  \leq C\lambda_q^{-5}.
\end{equation*}
The cutoff tail \eqref{eq:newest-perturbation-cutoff-tail} contributes at most
\(C\lambda_q^{-4}\), while
\eqref{eq:admissible-cutoff-parent-bound} bounds the cutoff-parent contribution by
\(\lambda_q^{\beta_q/8}\).

Since \(\lvert y_q(t)\rvert=2r_{q+1}\) and
\(\operatorname{supp}\theta_{q+1}(t)\subset B_{r_{q+1}}\),
\[
  \operatorname{dist}_{\mathbb T^2}
  \bigl(y_q(t),\operatorname{supp}\theta_{q+1}(t)\bigr)
  \geq r_{q+1}
\]
for all sufficiently large \(q\). The differentiated
periodic kernel is the differentiated Euclidean kernel, bounded by
\(C|h|^{-3}\), plus a smooth function.  Integration over the child support
therefore gives
\[
  |D_xu_{\mathbb T}(\theta_{q+1})(y_q(t),t)|
  \leq C\frac{\|\theta_{q+1}(t)\|_{C^0}}{r_{q+1}}
       +C\|\theta_{q+1}(t)\|_{L^1}.
\]
The perturbation-size estimate \eqref{eq:new-perturbation-size} and the definition
of \(r_{q+1}\) now give
\begin{align*}
  |D_xu_{\mathbb T}(\theta_{q+1})(y_q(t),t)|
  &\leq
  C L_{q+1}^{-4}
    \lambda_{q+1}^{-1+4\beta_{q+1}}
    +C\|\theta_{q+1}(t)\|_{L^1}
  \\
  &\leq1,
\end{align*}
for all sufficiently large \(q\), because
\(\beta_{q+1}\leq1/8\). Applying
\eqref{eq:periodic-spectral-tail} with \(j=2\), \(s=3\), and the
perturbation \(C^8\) estimate
\eqref{eq:admissible-perturbation-mixed-bound}, followed by the periodic
Schauder estimate, gives
\begin{equation}
  \|D_xu_{\mathbb T}
    ((S_{\mu_{q+2}}-I)\theta_{q+1})\|_{L^\infty}
  \leq C\lambda_{q+1}^{-4}.
  \label{eq:child-velocity-tail}
\end{equation}

Combining \eqref{eq:terminal-interval-decomposition} with
\eqref{eq:newest-velocity-entry} and the preceding error estimates yields
\[
  \|D_xu_{\mathbb T}(\rho(t))\|_{L^\infty}
  \geq c\lambda_q^{7\beta_q/8},
\]
for every sufficiently large \(q\), uniformly throughout the interval. Since
the intervals cover a terminal neighborhood of one and
\(\lambda_q^{7\beta_q/8}\to\infty\), this proves
\eqref{eq:velocity-gradient-blowup}.
\end{proof}

\begin{proof}[Proof of Theorem~\ref{thm:main}]
Lemma~\ref{lem:periodic-base-layer}, Proposition~\ref{prop:one-step}, and
the recursive exponent choice construct admissible stages for
every \(q\). Proposition~\ref{prop:global-limit} gives an odd smooth datum,
an odd force in \(C^\infty([0,1]\times\mathbb T^2)\), a classical solution
on \([0,1)\), and convergence to an endpoint in every \(C^\eta\),
\(0\leq\eta<1\). Propositions~\ref{prop:density-gradient-blowup} and
\ref{prop:velocity-gradient-blowup} give the two gradient divergences
asserted in the theorem.
\end{proof}

\section{Acknowledgments}
\label{sec:acknowledgments}

L.A.\ would like to thank Anthropic for support, because the use of Claude sped our collaboration's understanding of the correct and incorrect ideas in the literature and the implementations of our own ideas decisively. LA would also like to thank Mehmet Demirtaş and Federico Pasqualotto for helpful conversations, and Ralph Furman, Eric Price, as well as very many Anthropic collaborators more generally, for work advancing scientific capabilities used by this collaboration.

 T.B.\ and M.P.C.\ were supported by the NSF grants DMS-2243205 and DMS-2244879, in addition to the Simons Foundation Mathematical and Physical Sciences collaborative grant `Wave Turbulence'. T.B.\ would like to thank Diego C\'ordoba and Luis
Mart\'inez-Zoroa for the numerous enlightening discussions regarding their groundbreaking work which itself formed the foundations for this paper.

M.P.C.\ was supported at New York University by a Junior Fellowship of the Simons Society of Fellows, funded by Simons Foundation International and administered by the Simons Foundation.

% levent: this will roughly be my boilerplate but i will also thank Dario alongside the others in the Euler / further NS papers for something like camaraderie:)

% levent: my hope is that we sprinkle some Codex / things making OpenAI people happy in here somewhere too

\appendix
\section{Mixed estimates and smoothness of the force}
\label{sec:smooth-force}

At the step from level \(q\) to level \(q+1\), use the one-step aliases
\[
  \begin{gathered}
    \beta=\beta_{q+1},\quad \lambda=\lambda_{q+1},\quad
    L=L_{q+1},\quad M=M_{q+1}=\lceil\beta^{-2}\rceil,
    \quad \mu=\mu_{q+1},\quad r=r_{q+1},\\
    P=P_{q+1},\quad v=v_q=u_{\mathbb T}(P),\quad
    \Phi=\Phi_q,\quad X=X_{q+1},\quad
    \alpha=\alpha_{q+1},\quad a=a_{q+1},\quad s=s_{q+1},\\
    \gamma=1-2\beta,\qquad D=\partial_t+v\cdot\nabla.
  \end{gathered}
\]
The estimates below distinguish derivatives of perturbation coefficients from
derivatives of the oscillatory factor. A coefficient derivative costs only
powers of \(L\) and \(\lambda^{2\beta}\) per spatial derivative, whereas an
ordinary time derivative of the oscillatory factor costs
\(\lambda^\gamma\).

We first derive all-order bounds for the cutoff parent and its normalized
coefficient propagator. We then verify the endpoint and derivative-range
details for the two-pass mixed-time induction stated in
Lemma~\ref{lem:density-first-mixed-closure}. The final subsection verifies
the two finite uncut-parent ranges authorized by
\(K_q\geq d(\beta)\): the final grouped velocity remainder and the cutoff
commutator.

\subsection{All-orders parent estimates}

\begin{proof}[Proof of Lemma~\ref{lem:mixed-parent-inputs}]
At time order zero, the constant-free stage bound gives
\[
  \|\rho^{(q)}\|_{C^2}\leq2\lambda_q^2.
\]
The transfer identity turns \(\lambda_q^2\) into a power of the child
logarithm \(L\), and the periodic Schauder estimate gives the corresponding
velocity bound. If \(q=0\), all positive time derivatives vanish.

Now assume \(q\geq1\) and fix \(1\leq p\leq T\). Differentiating
\eqref{eq:iterated-density-formula} and using commutation with the cutoffs
gives
\[
  \partial_t^p\rho^{(q)}
  =S_{\mu_1}\partial_t^p\theta_0
   +\sum_{i=1}^{q-1}
      S_{\mu_{i+1}}\partial_t^p\theta_i
   +\partial_t^p\theta_q.
\]
For every \(0\leq i\leq q\),
\eqref{eq:admissible-perturbation-sharp-bound} gives
\[
  \|\partial_t^p\theta_i\|_{C^2}
  \leq C_{2,p}(\beta_i)L_i^{N_{2,p}(\beta_i)}
  \lambda_i^{1+\beta_i+p(1-2\beta_i)}.
\]
Monotonicity gives \(\beta_i\in[\beta_q,1/8]\). The a priori envelope
\eqref{eq:sharp-constant-envelope} therefore bounds all constants and
logarithmic exponents in terms of \(\beta_q\) and \(p\). Moreover,
\[
  1+\beta_i+p(1-2\beta_i)\leq\frac98+p.
\]
Thus a single exponent may be used before summation. For fixed \(A>0\)
and \(N\geq0\), lacunarity gives
\[
  \sum_{i=0}^q L_i^N\lambda_i^A
  \leq C_{A,N}L_q^N\lambda_q^A.
\]
Apply this with \(A=9/8+p\). The transfer identity
\eqref{eq:frequency-transfer} converts the resulting parent-frequency power
into a power of the child logarithm \(L\), proving the density part of
\eqref{eq:mixed-parent-low-norms}. The periodic Schauder estimate gives the
velocity part.

Since \(P=S_\mu\rho^{(q)}\) and \(S_\mu\) commutes with ordinary time
derivatives, the Fourier-cutoff Bernstein inequality and the preceding
\(C^2\) bound give
\[
  \|\partial_t^pP\|_{C^{j+1}}
  +\|\partial_t^pv\|_{C^{j+1}}
  \leq C_{p,j}L^{C_{p,j}}\mu^j.
\]
This proves \eqref{eq:mixed-cutoff-velocity}--
\eqref{eq:mixed-cutoff-gradient}. Oddness gives \(v(0,t)=0\), so the
mean-value theorem proves \eqref{eq:mixed-cutoff-velocity-core}.

Finally, induct on \(p+j\), using, for
\(\nu\in\{1,2\}\),
\([D,\partial_{x_\nu}]=-(\partial_{x_\nu}v)\cdot\nabla\).
All factors are evaluated on \(B_r\). The ordinary cutoff bounds control
spatial derivatives, while the transport factor satisfies
\(\|v\|_{C^0(B_r)}\leq CL^Cr\); hence its extra derivative costs at most
\(r\mu\leq1\), up to a power of \(L\). This proves
\eqref{eq:mixed-cutoff-local-material-estimates}. Each positive material order
uses ordinary parent time derivatives no higher than that order. In addition,
\(D^pv\) is odd, so it vanishes at the origin; its local \(C^0\) bound follows
from its \(C^1\) bound and the mean-value theorem.
\end{proof}

The preceding lemma supplies arbitrary spatial orders from the Fourier
cutoff and the a priori low-order parent bound. The inequality
\(K_q\geq d(\beta)\) is used only for the finite uncut-parent derivative
ranges in the residual and cutoff-commutator estimates.

\subsection{The weighted coefficient propagator}

\begin{proof}[Proof of Lemma~\ref{lem:normalized-weighted-propagator}]
Put \(\widetilde c_P(x,t)=c_P(x,t)-c_P(0,t)\). Since \(P\) is odd,
\(\widetilde c_P\) is even and has zero constant and linear jets at the
origin. Lemma~\ref{lem:mixed-parent-inputs} and Taylor's theorem give, for
\(y\in B_{\rho_{\rm ch}}\),
\[
  |\partial_t^m\widetilde c_P(y,t)|
  \leq C_mL^{C_m}\mu^2|y|^2,
\]
whereas, globally,
\[
  \|\partial_t^m\widetilde c_P\|_{C^k(\mathbb T^2)}
  \leq C_{m,k}L^{C_{m,k}}\mu^k
  \qquad(k\geq1).
\]
Here \(c_P=m_0(t)\cdot\nabla P\). Differentiating the angle equation
\eqref{eq:new-angle-transport} and using Lemma~\ref{lem:mixed-parent-inputs}
controls every time derivative of \(m_0(t)\) by a power of \(L\); at
positive order \(m\), only parent time orders through \(m-1\) occur. This
proves the two displayed bounds.
The same estimates hold with material derivatives in place of ordinary
time derivatives. Indeed, material differentiation preserves evenness, and
the material bounds in Lemma~\ref{lem:mixed-parent-inputs} give the same
Taylor estimate at the origin. Let \(\widetilde\Phi(\tau;t,x)\) denote the
coordinate lift furnished by Lemma~\ref{lem:cutoff-parent-flow}, invoked
with \(J\geq\max\{S+T+1,3\}\), and use the same notation
\(\widetilde c_P\) for its periodic coordinate lift. For
\(x\in B_r\) and \(\tau\) between \(\sigma\) and \(t\),
\eqref{eq:actual-parent-buffered-flow} gives
\[
  |\widetilde\Phi(\tau;t,x)|\leq CL^2r.
\]
The fixed threshold \(\Lambda_{\rm flow}\) is chosen so that
\[
  \sup_{\substack{x\in B_r\\
        \tau\text{ between }\sigma\text{ and }t}}
  |\widetilde c_P(\widetilde\Phi(\tau;t,x),\tau)|
  \leq CL^C\mu^2r^2\leq1.
\]
For each fixed \(m\geq1\), every required ordinary or material derivative
satisfies
\[
  \sup_{\substack{x\in\Omega_t\\
        \tau\text{ between }\sigma\text{ and }t}}
  |\mathcal D^m\widetilde c_P(\widetilde\Phi(\tau;t,x),\tau)|
  \leq C_mL^{C_m}\mu^2r^2\leq C_mL^{C_m},
  \qquad \mathcal D\in\{\partial_t,D\}.
\]
Here \(\mu^2r^2=\lambda^{-3\beta}L^8\); no positive-order smallness is
used. Also
\(\mu^k\leq\lambda^{2k\beta}\). The characteristic formula is
\[
  \widetilde H_{\sigma,t}[\varphi(\sigma)](x)
  =\varphi(\Phi(\sigma;t,x),\sigma)
   \exp\left(-\int_\sigma^t
     \widetilde c_P(\Phi(\tau;t,x),\tau)\,d\tau\right).
\]
The tube estimate gives the zeroth-order exponential bound. On
\(\Omega_\tau\Subset B_r\), the positive-order \(C^k\)-bounds above for
\(\widetilde c_P\), together with
\eqref{eq:actual-parent-flow-derivatives}, control every spatial derivative
of the exponential.
Since \(\varphi(\sigma)\) is compactly supported in
\(\Omega_\sigma\Subset B_r\), its stated \(B_r\)-bounds control every
nonzero term after composition. Finally,
\(\Phi(t;\sigma,\Omega_\sigma)=\Omega_t\) carries the support.

For a datum depending on the lower-endpoint variable \(\sigma\), set
\(U_{\sigma,t}=\widetilde H_{\sigma,t}[\varphi(\sigma)]\). Write
\[
  D_\sigma:=\partial_\sigma+v(\cdot,\sigma)\cdot\nabla
\]
for the material derivative acting on the lower-endpoint datum, and set
\(\mathcal L_\sigma:=D_\sigma+\widetilde c_P(\cdot,\sigma)\). The exact
endpoint identities are
\[
  \partial_\sigma U_{\sigma,t}
  =\widetilde H_{\sigma,t}[\mathcal L_\sigma\varphi(\sigma)],
  \qquad
  D_tU_{\sigma,t}=-\widetilde c_P(t)U_{\sigma,t}.
\]
Iteration of the first identity gives
\[
  \partial_\sigma^nU_{\sigma,t}
  =\widetilde H_{\sigma,t}
    [\mathcal L_\sigma^n\varphi(\sigma)].
\]
Every datum factor in \(\mathcal L_\sigma^n\varphi\) is
\(D_\sigma^m\varphi\) with \(m\leq n\); every differentiated parent factor
has material order at most \(n-1\). Iteration of the second identity
expresses \(D_t^p\partial_\sigma^nU_{\sigma,t}\) as a finite sum of
\(\widetilde H_{\sigma,t}[\mathcal L_\sigma^n\varphi]\) multiplied by
products of
\(\widetilde c_P,D\widetilde c_P,\ldots,D^{p-1}\widetilde c_P\).
Thus an output \(C^j\) norm uses datum pairs only with \(m\leq n\) and
spatial order at most \(j\). The diagonal condition gives
\(m+j\leq n+j\leq p+n+j\leq S+T\). Every parent factor has ordinary time
order at most \(p+n-1\), and all needed spatial orders are at most
\(p+n+j\leq S+T\). The preceding coefficient, flow, and product bounds
therefore prove \eqref{eq:mixed-normalized-weight}. The case \(p=n=0\)
uses only the characteristic formula and order-zero parent bounds.
\end{proof}

\begin{proof}[Proof of Corollary~\ref{lem:mixed-weighted-propagator}]
To pass from material to ordinary derivatives, use
\(\partial_t=D-v\cdot\nabla\). On \(B_r\), a transport factor acting on a
coefficient costs
\[
  r\lambda^{2\beta}=L^4.
\]
Commutators with spatial derivatives are controlled by
\eqref{eq:mixed-cutoff-velocity}, proving
\eqref{eq:ordinary-coefficient-class}. If the derivative falls on the
oscillatory factor, its cost is
\[
  r\lambda=L^4\lambda^{1-2\beta}=L^4\lambda^\gamma.
\]
The phase and parameter bounds in
Lemma~\ref{lem:transported-phase-estimates} prove
\eqref{eq:ordinary-live-packet}.
\end{proof}

\begin{proof}[Proof of Lemma~\ref{lem:weighted-duhamel-estimate}]
The identity
\[
  H_{\sigma,t}[a(\sigma)\varphi(\sigma)]
  =a(t)\widetilde H_{\sigma,t}[\varphi(\sigma)]
\]
shows that \(g=aJ\), where
\[
  J(t)=\int_s^t
    \widetilde H_{\sigma,t}[\varphi(\sigma)]\,d\sigma.
\]
Equivalently,
\[
  (D+\widetilde c_P)J=\varphi,
  \qquad J(s)=0.
\]
For each \(0\leq k\leq S+T\), apply the fixed-endpoint case of
Lemma~\ref{lem:normalized-weighted-propagator} with output parameters
\((T,S)=(0,k)\). Its local datum dependence uses only the order-zero
\(C^k\) bound for the Duhamel source \(\varphi\). The resulting estimate
is uniform in \(\sigma\), and
integration over \([s,t]\), whose length is at most one, gives
\[
  \|J\|_{C^k(B_r)}
  \leq C_kL^{C_k}\mathcal A\lambda^{2k\beta}
  \qquad(0\leq k\leq S+T).
\]
For \(m\geq1\), applying \(D^{m-1}\) to the equation for \(J\) gives the
exact recurrence
\[
  D^mJ
  =D^{m-1}\varphi
   -\sum_{\ell=0}^{m-1}\binom{m-1}{\ell}
      (D^\ell\widetilde c_P)D^{m-1-\ell}J.
\]
Induction in \(m\), using the Duhamel-source hypotheses and the parent bounds,
yields
\[
  \|D^mJ\|_{C^k(B_r)}
  \leq C_{m,k}L^{C_{m,k}}\mathcal A\lambda^{2k\beta}
  \qquad(m\leq T,\ m+k\leq S+T).
\]
Indeed, the Duhamel source term at this pair is
\(D^{m-1}\varphi\) in \(C^k\), so its time--space index satisfies
\((m-1)+k\leq S+T-1\).

To pass to ordinary time derivatives, use \(\partial_t=D-v\cdot\nabla\).
An order-\(p\) ordinary derivative in \(C^j\) uses \(D^mJ\) through spatial
order \(j+p-m\), \(0\leq m\leq p\). The term \(m=0\) uses the
undifferentiated Duhamel source \(\varphi\) through order \(j+p\); for
\(m\geq1\), the recurrence uses the Duhamel-source derivative pair
\((m-1,j+p-m)\). Each extra spatial
derivative is paired with a transport coefficient.
Equations~\eqref{eq:mixed-cutoff-velocity} and
\eqref{eq:mixed-cutoff-velocity-core}, together with
\[
  r\lambda^{2\beta}=L^4,
\]
show that these conversions cost only powers of \(L\). This proves the
material and ordinary bounds for \(J\). Finally, the amplitude equation
gives \(\lvert\partial_t^ma\rvert\leq C_mL^{C_m}a\); the product rule
proves \eqref{eq:weighted-duhamel-mixed-bound}. Characteristic transport
and \eqref{eq:duhamel-uniform-pulled-back-support} give
\[
  \operatorname{supp}g(t)
  \subset\Phi(t;s,K_\varphi)\Subset\Omega_t,
\]
which is the support assertion.
\end{proof}

For requested output orders \(T,K\), the cutoff-parent estimates supply
every spatial derivative through
\[
  N_{\mathrm{prof}}=K+M+T+2,
  \qquad
  N_{\mathrm{phase}}=N_{\mathrm{prof}}+1.
\]
The additional phase derivative is used when a differentiated phase term is
formed; it is supplied by the cutoff before
Proposition~\ref{prop:localized-periodic-velocity-expansion} is invoked.

\subsection{Technical proof of the density-first mixed induction}

\begin{proof}[Proof of Lemma~\ref{lem:density-first-mixed-closure}]
The normalized seed in \eqref{eq:normalized-seed-extension} has fixed
pulled-back compact support, and
Lemma~\ref{lem:normalized-weighted-propagator} uses only its value at
\(\sigma=s\). Equations~\eqref{eq:periodic-packet-anisotropic} and
\eqref{eq:periodic-packet-background}, followed by
\eqref{eq:duhamel-correction}, give the order-zero alternating chain
displayed after Lemma~\ref{lem:density-first-mixed-closure}. We verify the
Duhamel datum ranges and derivative endpoints needed at positive time order.

Fix the current step \(q\to q+1\), let \(p\geq1\), and suppose that its
coefficient and structured-defect estimates are known at every time order
below \(p\). For an order-\(p\) output-time derivative with the propagator's
lower endpoint fixed at \(s\), the propagator uses datum time order zero and
parent time orders at most \(p-1\), by
Lemma~\ref{lem:normalized-weighted-propagator}. The flow, phase, and angle
have the same lower-parent dependence by their evolution equations in
Lemmas~\ref{lem:cutoff-parent-flow} and
\ref{lem:transported-child-phase}.

For \(1\leq j<M\), let \(h_j\) be the corresponding coefficient of
\(\mathcal Q_j\), and write the coefficient of the \((j+1)\)-st correction
as
\[
  g_{j+1}(t)=-\int_s^tH_{\sigma,t}[h_j(\sigma)]\,d\sigma,
  \qquad
  \varphi_j=\frac{h_j}{a}.
\]
Fix an output pair \((p,k)\) with
\(k+p\leq\mathcal K_{j+1}+T\). The order-zero assertion
\(\mathbf R_{0,j}\) supplies \(\varphi_j\) through spatial order \(k+p\).
For \(1\leq m\leq p-1\), the already proved assertion
\(\mathbf R_{m,j}\), together with the derivative bounds for \(a^{-1}\),
supplies
\[
  \|D^m\varphi_j\|_{C^{k'}}
  \leq C_{m,k',j}L^{C_{m,k',j}}
    \lambda^{-1+(1-j)\beta+2k'\beta}
  \qquad(m+k'\leq k+p-1).
\]
These are exactly the Duhamel-source hypotheses of
Lemma~\ref{lem:weighted-duhamel-estimate} with \((T,S)=(p,k)\). Thus
\(\mathbf C_{p,j+1}\) uses only \(\mathbf R_{0,j}\) at spatial order
\(k+p\) and \(\mathbf R_{m,j}\) at spatial order \(k+p-1-m\),
\(1\leq m\leq p-1\). In the case \(p=1\), the positive-time
Duhamel-source range is empty. Hence every order-\(p\) correction coefficient
in the current child is obtained without an order-\(p\) structured defect.

In the defect pass, Lemma~\ref{lem:mixed-parent-inputs} supplies the
same-order cutoff-parent inputs from the prescribed low-order sharp family
and the Fourier cutoff. The finite expansion then estimates the
common-phase and background interactions by
\eqref{eq:periodic-packet-anisotropic}--
\eqref{eq:periodic-packet-background}, keeps the exact velocity remainder
grouped as in \eqref{eq:periodic-packet-leak}, and applies
Lemma~\ref{lem:mixed-cutoff-commutator}. This proves every
\(\mathbf R_{p,j}\) and the order-\(p\) force-increment estimate for this
step. No same-order defect is returned to the density pass, so induction in
\(p\) closes.

The descending derivative orders are consistent because
\(\mathcal K_j=\mathcal K_{j+1}+M+T+2\). For
\(k+p\leq\mathcal K_{j+1}+T\), the localized velocity expansion uses
coefficient order
\[
  k+M+p+2\leq\mathcal K_j
\]
and one additional phase derivative. Both are supplied by the parent and
phase hypotheses.
\end{proof}

Factor each correction coefficient as \(g=a\widetilde g\), as in the proof
of Lemma~\ref{lem:correction-estimates}. The normalized coefficient has the
constant size required by \eqref{eq:ordinary-live-packet}; restoring the
factor \(a\) gives \eqref{eq:mixed-perturbation-production} for every fixed
\(m,k\geq0\).
This is \eqref{eq:admissible-perturbation-sharp-bound}. In this estimate all
order-dependent factors remain in \(C_{k,m}(\beta)L^{N_{k,m}(\beta)}\);
the a priori envelope \eqref{eq:sharp-constant-envelope} was chosen to
dominate them, and none is absorbed into a frequency gain. The flow and
phase thresholds are independent of the requested derivative order. Thus
the sharp part of Proposition~\ref{prop:one-step} requires no enlargement of
its finite threshold for \((k,m)\). If
\(k+2m\leq k_{q+1}\), then
\[
  (k+2m)-\bigl(k+m(1-2\beta)-1+\beta\bigr)
  =m(1+2\beta)+1-\beta
  \geq\frac78.
\]
A finite threshold depending on \(k_{q+1}\) therefore absorbs the
logarithmic factor and gives
\eqref{eq:mixed-perturbation-finite-order-bound}.
The requested finite order \(k_{q+1}\) is arbitrary subject to
\(k_{q+1}\geq d(\beta)\): high spatial derivatives fall on the cutoff parent
and recursively generated perturbation coefficients, not directly on the
uncut complete parent \(\rho^{(q)}\).

\subsection{The finite force-derivative range}

Suppose
\[
  \beta\leq\frac1{40},
  \qquad
  m+k\leq
  m_{q+1}:=\left\lfloor\frac1{4\beta}\right\rfloor-10.
\]
For the differentiated perturbation residual, apply
Proposition~\ref{prop:localized-periodic-velocity-expansion} with \(T=m\) and
\(K=k+1\). The
required phase order is
\[
  N_{\mathrm{phase}}=M+m+k+4.
\]
Including the \(2m+2\) ordinary-to-material conversion margin in
Corollary~\ref{cor:finite-packet-inputs}, the total unsmoothed order is
\[
  N_{\mathrm{phase}}+2m+2=M+k+3m+6.
\]
Since \(m+k\leq(4\beta)^{-1}-10\),
\[
  M+k+3m+6
  \leq\left\lceil\beta^{-2}\right\rceil
      +\frac3{4\beta}-24
  <\left\lceil\frac4{\beta^2}\right\rceil+1
  =d(\beta).
\]
This count concerns only the grouped exact velocity remainder in
\eqref{eq:periodic-packet-leak}. Thus \(K_q\geq d(\beta)\) authorizes its
finite uncut-parent range. The other residual terms use the high spatial
orders of the cutoff parent from Lemma~\ref{lem:mixed-parent-inputs} and the
phase estimates from Lemmas~\ref{lem:cutoff-parent-flow} and
\ref{lem:transported-child-phase}. Their constants may depend on the
requested order, but their thresholds do not. The remaining uncut
\(C^2\) time derivatives in \eqref{eq:mixed-parent-low-norms} come from the
a priori sharp family.

The stage application in
Corollary~\ref{cor:stage-commutator-estimate} requires
mixed parent order at most
\[
  2k+4m+8+\left\lceil\frac2\beta\right\rceil
  \leq\frac3\beta-31
  <\frac4{\beta^2}.
\]
The required order is also strictly below \(d(\beta)\).
These derivatives are supplied by the constant-free stage bound
\eqref{eq:admissible-stage-mixed-bound} at level \(q\), already available
from admissibility before the child is constructed. Therefore
\eqref{eq:packet-residual-mixed-smallness} and
\eqref{eq:commutator-mixed-smallness} hold simultaneously and give
\begin{equation}
  \|\partial_t^mG_{q+1}\|_{C^k}
  \leq\lambda_{q+1}^{-1/4}
  \qquad(m+k\leq m_{q+1}).
  \label{eq:mixed-force-increment}
\end{equation}

The activation cutoff \(w\) vanishes to infinite order at \(s_{q+1}\).
Induction through the Duhamel
cascade shows that \(\theta_{q+1}\) and \(E_{q+1}^{\mathrm{pert}}\) vanish to
infinite order there. The term
\(\mathcal C_{\mu_{q+1}}(\rho^{(q)})\) contains no factor of the child
activation cutoff \(w\) and is already smooth. Whenever \(\theta_q=0\),
\[
  \rho^{(q)}=P_q=S_{\mu_q}\rho^{(q-1)},
\]
so \(\operatorname{supp}\widehat{\rho^{(q)}}\subset B_{2\mu_q}\).
The separation \eqref{eq:admissible-bandwidth-separation} gives
\(\mu_{q+1}\geq4\mu_q\), and
\eqref{eq:bandlimited-commutator-zero} makes the commutator zero.

Equations~\eqref{eq:mixed-perturbation-production},
\eqref{eq:mixed-perturbation-finite-order-bound}, and
\eqref{eq:mixed-force-increment} are the mixed-order conclusions asserted in
Proposition~\ref{prop:one-step}.  Their constants and logarithmic exponents
may depend on the fixed mixed order, but the all-order sharp estimate requires
no further frequency threshold. The infinite-order vanishing of
\(\theta_{q+1}\) and \(E_{q+1}^{\mathrm{pert}}\) at \(s_{q+1}\) follows
from the activation profile and the Duhamel cascade, while the exact residual
and cutoff-commutator
identities give the classical stage equation.

\section{Periodic multiplier, cutoff, and kernel estimates}

\begin{lemma}[Parity and periodic Schauder bounds]
\label{lem:periodic-parity-schauder}
Let \(\rho\in C^\infty(\mathbb T^2)\).  If
\(\rho(-x)=-\rho(x)\), then
\(u_{\mathbb T}(\rho)(-x)=-u_{\mathbb T}(\rho)(x)\).  In particular,
\[
  \partial^\gamma\rho(0)
  =\partial^\gamma u_{\mathbb T}(\rho)(0)=0
  \qquad\text{when }|\gamma|\text{ is even}.
\]
For every integer \(j\geq0\), every \(0<\eta<1\), and every
\(f\in C^{j,\eta}(\mathbb T^2)\),
\begin{equation}
  \|u_{\mathbb T}(f)\|_{C^{j,\eta}}
  \leq C_{j,\eta}\|f\|_{C^{j,\eta}}.
  \label{eq:periodic-schauder}
\end{equation}
In particular,
\begin{equation}
  \|D_xu_{\mathbb T}(f)\|_{L^\infty}
  \leq C\|f\|_{C^2}.
  \label{eq:periodic-jacobian-bound}
\end{equation}
\end{lemma}

\begin{proof}
Let \(\mathfrak Rf(x)=f(-x)\).  Evenness of this Fourier multiplier symbol gives
\(u_{\mathbb T}(\mathfrak Rf)=\mathfrak R u_{\mathbb T}(f)\), which proves
the parity assertion. On each periodic dyadic annulus, the rescaled Fourier
multiplier symbol has a smooth compactly supported extension. Let \(K_n\) denote its
periodized kernel. These kernels have uniformly bounded \(L^1\) norm and
integral zero, so
\[
  (K_n*f)(x)
  =\int_{\mathbb T^2}K_n(y)\bigl(f(x-y)-f(x)\bigr)\,dy.
\]
Applying this first-difference identity and treating the finitely many low
modes proves
\eqref{eq:periodic-schauder}.  Apply that estimate to each \(\partial_af\)
with H\"older exponent \(1/2\), and bound
\(\|\partial_af\|_{C^{0,1/2}}\) by \(C\|f\|_{C^2}\), to obtain
\eqref{eq:periodic-jacobian-bound}.
\end{proof}

\begin{lemma}[Cutoff tails and exact Fourier cutoffs]
\label{lem:periodic-cutoff-calculus}
Let \(\mu\geq1\) and \(f,g\in C^\infty(\mathbb T^2)\). For integers
\(j,s\geq0\),
\begin{align}
  \|S_\mu f\|_{C^j}
  &\leq C_j\|f\|_{C^j},
  \label{eq:periodic-cutoff-boundedness}\\
  \|(I-S_\mu)f\|_{C^j}
  &\leq C_{j,s}\mu^{-s}\|f\|_{C^{j+s+3}}.
  \label{eq:periodic-spectral-tail}
\end{align}
For every \(0<\eta<1\),
\begin{equation}
  \|S_\mu f\|_{C^\eta}
  \leq C_\eta\|f\|_{C^\eta}.
  \label{eq:periodic-cutoff-holder-boundedness}
\end{equation}
The cutoff preserves parity and the spatial mean and commutes with spatial
derivatives and \(u_{\mathbb T}\). On smooth time-dependent families, it
also commutes with \(\partial_t\). Moreover,
\begin{align}
  \operatorname{supp}\widehat{S_\mu f}
  &\subset\{k:|k|\leq2\mu\},
  \label{eq:periodic-cutoff-support}\\
  \nu\geq2\mu
  &\quad\Longrightarrow\quad S_\nu S_\mu=S_\mu.
  \label{eq:periodic-cutoff-nesting}
\end{align}
If
\(\operatorname{supp}\widehat f,\operatorname{supp}\widehat g
 \subset\{|k|\leq2\mu\}\), then
\begin{equation}
  \operatorname{supp}\widehat{
    u_{\mathbb T}(f)\cdot\nabla g}
  \subset\{|k|\leq4\mu\}.
  \label{eq:periodic-product-bandwidth}
\end{equation}
Consequently, when \(\nu\geq4\mu\),
\begin{equation}
  S_\nu\bigl(u_{\mathbb T}(f)\cdot\nabla g\bigr)
  =u_{\mathbb T}(f)\cdot\nabla g.
  \label{eq:periodic-product-frequency-cutoff}
\end{equation}
\end{lemma}

\begin{proof}
Poisson summation represents the kernel of \(S_\mu\) as a periodization of
a Schwartz function at scale \(\mu^{-1}\).  Its \(L^1(\mathbb T^2)\) norm
is bounded uniformly in \(\mu\), which proves
\eqref{eq:periodic-cutoff-boundedness}.  The same kernel bound gives
\[
  |S_\mu f(x)-S_\mu f(y)|
  \leq C[f]_{C^\eta}|x-y|^\eta,
\]
and hence \eqref{eq:periodic-cutoff-holder-boundedness}.  For the tail,
Fourier inversion and
three integrations beyond the requested derivative order give
\[
 \sum_{|k|\geq\mu}|k|^j|\widehat f(k)|
 \leq C_{j,s}\mu^{-s}\|f\|_{C^{j+s+3}},
\]
and on the transition support \(\mu\leq|k|\leq2\mu\), we have the same estimate.
This proves \eqref{eq:periodic-spectral-tail}. Evenness gives parity
preservation, \(\chi(0)=1\) gives mean preservation, multiplier form gives
commutation with spatial derivatives, scalarity gives commutation with
\(u_{\mathbb T}\), and time independence gives commutation with \(\partial_t\).

The support assertion follows directly from \eqref{eq:periodic-cutoff}.  On
that support, \(|k|/\nu\leq1\) when \(\nu\geq2\mu\), proving exact nesting.
The multiplier \(u_{\mathbb T}\) does not enlarge Fourier support.
Convolution of the two Fourier supports then gives
\eqref{eq:periodic-product-bandwidth}, and the identity \(\chi=1\) on \(B_1\)
gives \eqref{eq:periodic-product-frequency-cutoff}.
\end{proof}

\begin{lemma}[Periodic kernel decomposition]
\label{lem:periodic-kernel-decomposition}
For \(f\in C^\infty(\mathbb T^2)\) supported in \(U_0\), one has on \(U\)
\[
  V_{\mathbb T}f=V_{\mathbb R}f+\mathcal H_0f,
\]
where \(V_{\mathbb R}\) is defined by \eqref{eq:euclidean-V}, and
\[
  \mathcal H_0f(x)=\int_{U_0}H_0(x,y)f(y)\,dy
  \qquad\text{with }H_0\in C^\infty(U\times U_0).
\]
There is also a kernel \(H_{\rm far}\), smooth on a neighborhood of
\((\mathbb T^2\setminus U)\times\overline{U_0}\), such that, for
\(x\notin U\),
\[
  V_{\mathbb T}f(x)
  =\int_{U_0}H_{\rm far}(x,y)f(y)\,dy.
\]
The same decomposition holds componentwise for
\(u_{\mathbb T}\) through \eqref{eq:periodic-velocity-from-V}.
\end{lemma}

\begin{proof}
Let \(\mathcal G_{\mathbb T}\) be the mean-zero Green function of
\(-\Delta\) on \(\mathbb T^2\).  In a coordinate neighborhood of the
diagonal,
\[
  \mathcal G_{\mathbb T}(x-y)
  =-\frac1{2\pi}\log|x-y|+h(x,y).
\]
Indeed, subtracting the displayed logarithm cancels the delta mass in the
Green equation; the difference solves a Poisson equation with smooth
right-hand side, and interior elliptic regularity makes \(h\) smooth.
Since \(V_{\mathbb T}=2\pi\partial_1(-\Delta)^{-1}\), its local convolution
kernel with respect to ordinary Lebesgue measure is
\[
  -\frac{x_1-y_1}{|x-y|^2}+\widetilde H(x,y).
\]
After the substitution \(h=y-x\), the singular term is exactly
\eqref{eq:euclidean-V}.  Insert a cutoff equal to one near the diagonal.
The difference between the cutoff singular kernel and the full Euclidean
kernel is smooth on \(U\times U_0\).  Since
\(\overline{U_0}\Subset U\), the periodic kernel is smooth on a neighborhood
of \((\mathbb T^2\setminus U)\times\overline{U_0}\), which gives
\(H_{\rm far}\).
Finally, since \(\operatorname{supp}f\Subset U_0\), integration by parts in
\(y\) gives
\[
  \bigl(-\mathcal H_0(\partial_2f),\mathcal H_0(\partial_1f)\bigr)
  =\int_{U_0}
    \bigl(\partial_{y_2}H_0(x,y),-\partial_{y_1}H_0(x,y)\bigr)f(y)\,dy.
\]
The same identity with \(H_{\rm far}\) proves the componentwise assertion;
all resulting kernels are smooth.
\end{proof}


\begin{thebibliography}{99}

\bibitem{CordobaMartinezZoroa2025}
D.~C\'ordoba and L.~Mart\'inez-Zoroa,
\emph{Finite time singularities of smooth solutions for the 2D incompressible
porous media (IPM) equation with a smooth source},
arXiv:2410.22920v3 [math.AP], 2025.

\bibitem{CordobaGancedoOrive2007}
D.~C\'ordoba, F.~Gancedo, and R.~Orive,
\emph{Analytical behavior of two-dimensional incompressible flow in porous
media},
J. Math. Phys. \textbf{48} (2007), 065206.

\bibitem{KiselevYao2023}
A.~Kiselev and Y.~Yao,
\emph{Small scale formations in the incompressible porous media equation},
Arch. Ration. Mech. Anal. \textbf{247} (2023), Paper No.~1.

\bibitem{BianchiniCordobaMartinezZoroa2025}
R.~Bianchini, D.~C\'ordoba, and L.~Mart\'inez-Zoroa,
\emph{Non-existence and strong ill-posedness in \(H^2\) for the stable IPM
equation},
J. Funct. Anal. \textbf{289} (2025), no.~9, 111097.

\bibitem{ElgindiEuler2021}
T.~M. Elgindi,
\emph{Finite-time singularity formation for \(C^{1,\alpha}\) solutions to the
incompressible Euler equations on \(\mathbb R^3\)},
Ann. of Math. (2) \textbf{194} (2021), no.~3, 647--727.

\bibitem{BourgainLi2015}
J.~Bourgain and D.~Li,
\emph{Strong ill-posedness of the incompressible Euler equation in borderline
Sobolev spaces},
Invent. Math. \textbf{201} (2015), no.~1, 97--157.

\bibitem{CordobaMartinezZoroaZheng2024}
D.~C\'ordoba, L.~Mart\'inez-Zoroa, and F.~Zheng,
\emph{Finite time blow-up for the hypodissipative Navier--Stokes equations with
a force in \(L_t^1C_x^{1,\varepsilon}\cap L_t^\infty L_x^2\)},
arXiv:2407.06776v2 [math.AP], 2024.

\bibitem{CordobaDominguezLucasMartinezZoroa2026}
D.~C\'ordoba, \'O.~Dom\'inguez, J.~Lucas-Manch\'on, and
L.~Mart\'inez-Zoroa,
\emph{Blow-up at finite time for the generalized SQG equations in the Sobolev
well-posedness regime},
arXiv:2608.17192v1 [math.AP], 2026.

\bibitem{Elgindi2017}
T.~M. Elgindi,
\emph{On the asymptotic stability of stationary solutions of the inviscid
incompressible porous medium equation}, Arch. Ration. Mech. Anal.
\textbf{225} (2017), 573--599.

\bibitem{CastroCordobaLear2019}
A.~Castro, D.~C\'ordoba, and D.~Lear,
\emph{Global existence of quasi-stratified solutions for the confined
IPM equation}, Arch. Ration. Mech. Anal. \textbf{232} (2019), 437--471.

\bibitem{BianchiniCrinBaratPaicu2024}
R.~Bianchini, T.~Crin-Barat, and M.~Paicu,
\emph{Relaxation approximation and asymptotic stability of stratified
solutions to the IPM equation}, Arch. Ration. Mech. Anal. \textbf{248}
(2024), Article~2.

\bibitem{Park2025}
J.~Park,
\emph{Stability analysis of the incompressible porous media equation and
the Stokes transport system via energy structure}, Calc. Var. Partial
Differential Equations \textbf{64} (2025), Article~169.

\bibitem{CordobaFaracoGancedo2011}
D.~C\'ordoba, D.~Faraco, and F.~Gancedo,
\emph{Lack of uniqueness for weak solutions of the incompressible porous
media equation}, Arch. Ration. Mech. Anal. \textbf{200} (2011), 725--746.

\bibitem{Szekelyhidi2012}
L.~Sz\'ekelyhidi, Jr.,
\emph{Relaxation of the incompressible porous media equation}, Ann. Sci.
\'Ec. Norm. Sup\'er. \textbf{45} (2012), 491--509.

\bibitem{IsettVicol2015}
P.~Isett and V.~Vicol,
\emph{H\"older continuous solutions of active scalar equations},
Ann. PDE \textbf{1} (2015), Article~2.

\bibitem{ForsterSzekelyhidi2018}
C.~F\"orster and L.~Sz\'ekelyhidi, Jr.,
\emph{Piecewise constant subsolutions for the Muskat problem},
Comm. Math. Phys. \textbf{363} (2018), 1051--1080.

\bibitem{CastroCordobaFaraco2021}
A.~Castro, D.~C\'ordoba, and D.~Faraco,
\emph{Mixing solutions for the Muskat problem}, Invent. Math.
\textbf{226} (2021), 251--348.

\bibitem{CastroFaracoMengual2019}
A.~Castro, D.~Faraco, and F.~Mengual,
\emph{Degraded mixing solutions for the Muskat problem}, Calc. Var.
Partial Differential Equations \textbf{58} (2019), Article~58.

\bibitem{CastroFaracoMengual2022}
A.~Castro, D.~Faraco, and F.~Mengual,
\emph{Localized mixing zone for Muskat bubbles and turned interfaces},
Ann. PDE \textbf{8} (2022), Article~7.

\bibitem{Mengual2022}
F.~Mengual,
\emph{$H$-principle for the 2-dimensional incompressible porous media
equation with viscosity jump}, Anal. PDE \textbf{15} (2022), 429--476.

\bibitem{CastroFaracoGebhard2025}
A.~Castro, D.~Faraco, and B.~Gebhard,
\emph{Entropy solutions to the macroscopic incompressible porous media
equation}, Anal. PDE \textbf{18} (2025), 2241--2292.

\bibitem{CastroCordobaGancedoOrive2009}
A.~Castro, D.~C\'ordoba, F.~Gancedo, and R.~Orive,
\emph{Incompressible flow in porous media with fractional diffusion},
Nonlinearity \textbf{22} (2009), 1791--1815.

\bibitem{CollotPrangeTan2026}
C.~Collot, C.~Prange, and J.~Tan,
\emph{Stable self-similar singularity formation for infinite energy
solutions of the incompressible porous medium equations}, Comm. Math.
Phys. \textbf{407} (2026), Article~150.

\bibitem{KiselevSarsam2025}
A.~Kiselev and N.~A. Sarsam,
\emph{Finite time blow-up in a 1D model of the incompressible porous media
equation}, Nonlinearity \textbf{38} (2025), 055018.

\bibitem{Dembski2025}
K.~H. Dembski,
\emph{Singularity formation in the incompressible porous medium equation
without boundary mass}, arXiv:2511.01827, 2025.

\bibitem{ElgindiJeong2019}
T.~M. Elgindi and I.-J. Jeong,
\emph{Finite-time singularity formation for strong solutions to the
axi-symmetric 3D Euler equations}, Ann. PDE \textbf{5} (2019), Article~16.

\bibitem{ElgindiGhoulMasmoudi2021}
T.~M. Elgindi, T.-E. Ghoul, and N.~Masmoudi,
\emph{On the stability of self-similar blow-up for $C^{1,\alpha}$
solutions to the incompressible Euler equations on $\mathbb R^3$},
Cambridge J. Math. \textbf{9} (2021), 1035--1075.

\bibitem{LuoHouPNAS2014}
G.~Luo and T.~Y. Hou,
\emph{Potentially singular solutions of the 3D axisymmetric Euler
equations}, Proc. Natl. Acad. Sci. USA \textbf{111} (2014), 12968--12973.

\bibitem{LuoHouMMS2014}
G.~Luo and T.~Y. Hou,
\emph{Toward the finite-time blowup of the 3D axisymmetric Euler equations:
A numerical investigation}, Multiscale Model. Simul. \textbf{12} (2014),
1722--1776.

\bibitem{ChenHou2021}
J.~Chen and T.~Y. Hou,
\emph{Finite time blowup of 2D Boussinesq and 3D Euler equations with
$C^{1,\alpha}$ velocity and boundary}, Comm. Math. Phys. \textbf{383}
(2021), 1559--1667.

\bibitem{ChenHou2024}
J.~Chen and T.~Y. Hou,
\emph{On stability and instability of $C^{1,\alpha}$ singular solutions
to the 3D Euler and 2D Boussinesq equations}, Comm. Math. Phys.
\textbf{405} (2024), Article~94.

\bibitem{CordobaMartinezZoroaZheng2025}
D.~C\'ordoba, L.~Mart\'inez-Zoroa, and F.~Zheng,
\emph{Finite time singularities to the 3D incompressible Euler equations
for solutions in $C^\infty(\mathbb R^3\setminus\{0\})\cap
C^{1,\alpha}\cap L^2$}, Ann. PDE \textbf{11} (2025), Article~19.

\bibitem{ElgindiPasqualotto2023}
T.~M. Elgindi and F.~Pasqualotto,
\emph{From instability to singularity formation in incompressible fluids},
arXiv:2310.19780, 2023.

\bibitem{Chen2024Smoothness}
J.~Chen,
\emph{Remarks on the smoothness of the $C^{1,\alpha}$ asymptotically
self-similar singularity in the 3D Euler and 2D Boussinesq equations},
Nonlinearity \textbf{37} (2024), 065018.

\bibitem{CordobaMartinezZoroa2023}
D.~C\'ordoba and L.~Mart\'inez-Zoroa,
\emph{Blow-up for the incompressible 3D-Euler equations with uniform
$C^{1,1/2-\varepsilon}\cap L^2$ force}, arXiv:2309.08495, 2023.

\bibitem{Shkoller2026}
S.~Shkoller,
\emph{Incompressible Euler blowup at the $C^{1,1/3}$ threshold},
arXiv:2603.10945v3, 2026.

\bibitem{Chen2026PartI}
J.~Chen,
\emph{Asymptotically self-similar blowup for 3D incompressible Euler with
$C^{1,1/3-}$ velocity I: $C^\infty$ 1D limiting profiles},
arXiv:2605.15149, 2026.

\bibitem{Chen2026PartII}
J.~Chen,
\emph{Asymptotically self-similar blowup for 3D incompressible Euler with
$C^{1,1/3-}$ velocity II: 3D profiles, blowup, and limiting behavior},
arXiv:2605.15130v2, 2026.

\bibitem{BuckmasterDeLellisSzekelyhidiVicol2019}
T.~Buckmaster, C.~De~Lellis, L.~Sz\'ekelyhidi, Jr., and V.~Vicol,
\emph{Onsager's conjecture for admissible weak solutions}, Comm. Pure
Appl. Math. \textbf{72} (2019), 229--274.

\bibitem{WangLaiGomezSerranoBuckmaster2023}
Y.~Wang, C.-Y. Lai, J.~G\'omez-Serrano, and T.~Buckmaster,
\emph{Asymptotic self-similar blow-up profile for three-dimensional
axisymmetric Euler equations using neural networks}, Phys. Rev. Lett.
\textbf{130} (2023), 244002.

\bibitem{WangEtAl2025Discovery}
Y.~Wang et al.,
\emph{Discovery of unstable singularities}, arXiv:2509.14185, 2025.

\bibitem{WangLegerLaiBuckmaster2025}
Y.~Wang, T.~L\'eger, C.-Y. Lai, and T.~Buckmaster,
\emph{Resolving sharp gradients of unstable singularities to machine
precision via neural networks}, arXiv:2511.22819, 2025.

\end{thebibliography}
\end{document}